\documentclass[10pt]{amsart}
\usepackage{amsmath,amssymb,latexsym,cite,empheq}
\usepackage[small]{caption}
\usepackage{graphicx,color,mathrsfs,tikz}
\usepackage{subfigure,color}
\usepackage{cite}
\usepackage[colorlinks=true,urlcolor=blue,
citecolor=red,linkcolor=blue,linktocpage,pdfpagelabels,
bookmarksnumbered,bookmarksopen]{hyperref}
\usepackage[italian,english]{babel}
\usepackage{units}
\usepackage{enumitem}
\usepackage[left=2.6cm,right=2.6cm,top=2.75cm,bottom=2.75cm]{geometry}
\usepackage[hyperpageref]{backref}

\usepackage[colorinlistoftodos,prependcaption]{todonotes}
\usepackage{cancel}
\usepackage{etex}
\usepackage{latexsym}
\usepackage[utf8x]{inputenc}

\numberwithin{equation}{section}
\newtheorem{theorem}{Theorem}[section]
\newtheorem{proposition}[theorem]{Proposition}
\newtheorem{lemma}[theorem]{Lemma}
\newtheorem{remark}[theorem]{Remark}

\newtheorem{corollary}[theorem]{Corollary}
\newtheorem{definition}[theorem]{Definition}
\theoremstyle{definition}

\newcommand{\R}{{\mathbb R}}

\newcommand{\eps}{\varepsilon}
\newcommand{\scal}[2]{\langle {#1},{#2}\rangle}

\newcommand{\escal}[2]{\langle {#1},{#2}\rangle_{\eps}}

\newcommand{\He}{\mathbb{H}}

\newcommand{\egrad}{\nabla_{\eps}}

\newcommand{\ediv}{\mathrm{div}_{\eps}}

\newcommand{\eT}{T_{\eps}}

\date{\today}
\makeatletter
\@namedef{subjclassname@2020}{\textup{2020} Mathematics Subject Classification}
\makeatother

\begin{document}

\title[IMCF in the Heisenberg group]{Global weak solutions for the inverse mean curvature flow in the Heisenberg group}
\author[A.\ Pisante]{Adriano Pisante}
\author[E.\ Vecchi]{Eugenio Vecchi}

\address[A.\,Pisante]{Dipartimento di Matematica Guido Castelnuovo
 \newline\indent Sapienza Universit\`a di Roma \newline\indent
 Piazzale Aldo Moro 5, 00185 Roma, Italy}
 \email{pisante@mat.uniroma1.it}

 \address[E.\,Vecchi]{Dipartimento di Matematica
 \newline\indent Universit\`a di Bologna \newline\indent
 Piazza di Porta San Donato 5, 40126 Bologna, Italy}
 \email{eugenio.vecchi2@unibo.it}




%
\begin{abstract}
We consider the inverse mean curvature flow (IMCF) in the Heisenberg group $(\He^n, d_\varepsilon)$, where $d_\varepsilon$ is distance associated to either $| \cdot |_\varepsilon$, $\varepsilon>0$, the natural family of left-invariant Riemannian metrics, or with their sub-Riemannian counterparts for $\varepsilon=0$. For $\Omega \subseteq \He^n$ an open set with smooth boundary $\Sigma_0=\partial \Omega$ satisfying a uniform exterior gauge-ball condition and bounded complement we show existence of a global weak IMCF of generalized hypersurfaces $\{\Sigma^\varepsilon_s\}_{s \geq 0} \subseteq \mathbb{H}^n$ which are level sets of a proper globally Lipschitz function with logarithmic growth at infinity. Here, both in the Riemannian and in the sub-Riemannian setting, we adopt the weak formulation introduced by Huisken and Ilmanen in \cite{HuiskenIlmanen}, following the approach in \cite{Moser} due to Moser and based on the the link between IMCF and $p$-harmonic functions.  
\end{abstract}  
\maketitle

\section{Introduction}
Let $\He^n$ be the $(2n+1)$-dimensional Heisenberg group and let $\Omega \subset \He^n$ be an open set with $C^2$-smooth boundary $\Sigma_0=\partial\Omega$ such that its complement\footnote{We will use both notations for complement of sets, depending on the graphical convenience along the text.} $\Omega^c := \He^n \setminus \Omega$ is bounded.
Throughout the paper we assume that $\Omega$ satisfies an {\it exterior uniform gauge-ball
condition}  ($\mathbf{HP_\Omega}$) with parameter $R_0$ (see Definition \ref{exterior-condition}). 

In the present paper we are mainly interested in positive solutions to the boundary value problem
\begin{equation}\label{EulerForU}
\left\{ \begin{array}{ll}
            {\rm div}_0 \left( \dfrac{\nabla_0 u}{|\nabla_0 u|_0}\right) = |\nabla_0 u|_0 & \textrm{in } \Omega,\\
						u= 0 & \textrm{on } \partial \Omega, 
        \end{array}\right.
\end{equation}
where $\rm{div}_0$ is the horizontal divergence, $\nabla_0 u$ is the horizontal gradient of $u$ and its length is computed w.r.to the standard sub-Riemannian metric $| \cdot |_0$ on $\He^n$. In the model case of the complement of a  gauge-ball, i.e., $\Omega=\He^n \setminus \overline{B_{R_0}(g_0)}$ for some $g_0 \in \He^n$ and $R_0>0$, an explicit solution of \eqref{EulerForU} can be written in terms of the Kor\'anyi norm $\|\cdot\|$ as 
\begin{equation}
\label{explicitsolution}
U(g)=(Q-1) \log \frac{\| g_0^{-1} * g\|}{R_0} \, ,
\end{equation} 
where $Q=2n+2$ being from now on the homogeneous dimension of $\mathbb{H}^n$.

Problem \eqref{EulerForU} can be considered as the natural sub-Riemannian counterpart of
the level set formulation of the classical Inverse Mean Curvature Flow (IMCF)
used by Huisken and Ilmanen in their celebrated proof of the Riemannian Penrose inequality \cite{HuiskenIlmanen}. Indeed, a smooth embedded {\em horizontal} IMCF for a manifold $\Sigma$ of dimension $2n$ inside $\Omega$ starting from $\Sigma_0\subseteq \overline{\Omega}$ is a smooth map $F: \Sigma\times [0,S) \to \He^n$ such that (i) $\Sigma_0=F(\Sigma,0)$ and $\Sigma_s=F(\Sigma, s) \subseteq \Omega$ for $s \in (0,S)$,  (ii) $\{F(\cdot,s)\}_{s \in [0,S)}$ are embeddings whose time\footnote{Although counterintuitive, in the present paper we call $s$ the time variable in the IMCF, keeping the usual notation $t$ for the variable along the center of the group $\He^n$.} derivative for $s \in (0,S)$ is given by 
\[ \frac{\partial F}{\partial s}=-\frac{\mathcal{H}_H}{|\mathcal{H}_H|_0^2} \, ,\]
where $\mathcal{H}_H\neq 0$ at a point $(p,s) \in \Sigma \times (0,S)$ is the {\em horizontal mean curvature} vector of the hypersurface $\Sigma_s=F(\Sigma,s)$ at the point $F(p,s)$. In case of regular hypersurfaces, if a smooth function $u: \overline{\Omega} \to [0,\infty)$ with $\nabla_0 u \neq 0$ in $\overline{\Omega}$ satisfies \eqref{EulerForU} then the family  $\{\Sigma_s\}_{s \in [0,S)}$, with $\Sigma_s:=\{u =s\} $, $s \in [0,S)$ and $S>0$ small enough, is a smooth horizontal IMCF, with $\Sigma_s=F(\Sigma, s)$ and the map $F$ being the flow map associated to the vector field $\frac{\nabla_0 u}{|\nabla_0 u|_0^2}$. In particular, when the initial surface is $\Sigma_0=\partial B_{R_0}(g_0)$ then according to \eqref{explicitsolution} we have $\Sigma_s=\partial B_{R(s)}(g_0)$, where $\displaystyle{R(s)=R_0 e^\frac{s}{Q-1}}$ and $s \in [0,\infty)$, and indeed \eqref{EulerForU} is satisfied at least wherever $\nabla_0 U \neq 0$, i.e., at those points of the corresponding level sets which are not characteristic.

Thus, together with \cite{CuiZhao}, our paper provides a first attempt in understanding weak solutions for the horizontal IMCF in the sub-Riemannian Heisenberg group $\mathbb{H}^n$, in analogy with similar results for the horizontal MCF  (see \cite{CapoCitti}, \cite{CapoCitti2}, \cite{Ferrari} and \cite{Dirdravon}) and the horizontal Gauss flow  (see \cite{Haller}). 


Together with \eqref{EulerForU}, we also consider for $0<\varepsilon \leq 1$ positive solutions to the family of problems

\begin{equation}\label{EulerForUeps}
\left\{ \begin{array}{ll}
            {\rm div}_\varepsilon \left( \dfrac{\nabla_\varepsilon u}{|\nabla_\varepsilon u|_\varepsilon}\right) = |\nabla_\varepsilon u|_\varepsilon & \textrm{in } \Omega,\\
						u= 0 & \textrm{on } \partial \Omega, 
        \end{array}\right.
\end{equation}
where $\rm{div}_\varepsilon$ and $\nabla_\varepsilon u$ are computed w.r.to the standard left invariant Riemannian metrics $| \cdot |_\varepsilon$ and the equation \eqref{EulerForUeps} give the analogous level set formulation for the corresponding IMCFs. Here the actual interest in \eqref{EulerForUeps} is twofold: on the one hand viewing it as an $\varepsilon$-regularized version of \eqref{EulerForU} but on the other hand even in the case $\varepsilon=1$ a solution to \eqref{EulerForUeps} giving a family of hypersurfaces expanding by IMCF in a Riemannian manifold with an intrinsically sub-Riemannian nature at large scales (indeed, as it is well-known, the sub-Riemannian space being the tangent cone at infinity of the Riemannian one under Gromov-Hausdorff convergence of its rescaled copies under group dilations), a feature which deeply affects the large time behaviour of the flow.

As first pointed out in \cite{HuiskenIlmanen},  for $0\leq \varepsilon \leq 1$ equations \eqref{EulerForU}-\eqref{EulerForUeps} formally correspond to the Euler-Lagrange equations for minimizers of the (non differentiable) energy functionals
\begin{equation}
\label{HI-energyeps}
J^\varepsilon_u(w;K)= \int_K |\nabla_\varepsilon w|_\varepsilon+w |\nabla_\varepsilon u|_\varepsilon \, ,
\end{equation}
among functions $w$ which are compactly supported perturbations of $u$,  i.e., $\overline{\{ u\neq w\}} \subset  K$ and $K \subset \Omega$ any compact set.  Here and throughout the paper, all the volume integrals are computed w.r.to the Lebesgue measure $\mathcal{L}^{2n+1}$ in $\mathbb{R}^{2n+1}\simeq \He^n$ (i.e., up to a factor the unique Haar measure on $\He^n$), also denoted sometimes by $|\cdot|$ in the sequel, and it will be deliberately omitted in all the integrals. 


In the present paper, following \cite{HuiskenIlmanen} we consider solutions to \eqref{EulerForU} in the sense of the following definition.

\begin{definition}
\label{def:solution}
A locally $d_\varepsilon$-Lipschitz function $u:\overline{\Omega} \to \mathbb{R}$ is a weak solution of \eqref{EulerForU} or \eqref{EulerForUeps} if $u\equiv 0$ on $\partial \Omega$ and for every compact set $K \subset \Omega$ and every locally $d_\varepsilon$-Lipschitz function $w:\overline{\Omega} \to \mathbb{R}$ such that $u\equiv w$ in $\Omega \setminus K$ the inequality 
\begin{equation}
	\label{Ju-inequality}
	J^\varepsilon_u(u;K) \leq J^\varepsilon_u(w;K) \,  
\end{equation} 	
holds. A weak solution is \rm{proper} if $\displaystyle{\lim_{\|g\| \to \infty} u(g)=+\infty}$.
\end{definition}
\vskip10pt
 
 In order to deal with  \eqref{EulerForU}-\eqref{EulerForUeps} we adapt the very clever approach introduced in the important paper \cite{Moser} for the Euclidean case (see also \cite{Moser2}), and then fruitfully exploited both in the Riemannian and in Finsler geometries, see e.g. \cite{Mazzieri, Mazzieri4, Rigoli, Xia, DellaPietra, Cabezas}. Thus, for fixed $\varepsilon \in (0,1]$ we are are going to construct solutions $u^\varepsilon:\overline{\Omega} \subseteq \He^n \to \R$ of \eqref{EulerForUeps} as limits as $p \downarrow 1$ of solutions  $u^\varepsilon_{p}:\overline{\Omega} \subseteq \He^n \to \R$ to the regularized problems 

\begin{subequations}\label{eq:system}
\begin{empheq}[left=\empheqlbrace]{align}
  {\rm div}_\varepsilon \left( | \nabla_\varepsilon u^\varepsilon_p |_\varepsilon^{p-2} \nabla_\varepsilon {u^\varepsilon_p}\right) = | \nabla_\varepsilon u^\varepsilon_p|_\varepsilon^{p} &\quad \textrm{in } \Omega,\label{EqForUpeps}
  \\
  u^\varepsilon_p= 0 &\quad \textrm{on } \partial \Omega. \label{EqForUpeps-bis}
\end{empheq}
\end{subequations}

On the other hand, concerning \eqref{EulerForU} we construct a solution $u:\overline{\Omega} \subseteq \He^n \to \R$ as $u=\lim u^\varepsilon_p$ where we first let  $\varepsilon \downarrow 0$ and then $p \downarrow 1$ or we take both limits simultaneously.

As it will be shown below, for $0< \varepsilon \leq 1$ and  $1<p<Q$ problems \eqref{EqForUpeps} admit solutions $u_p^\varepsilon$ in $C(\overline{\Omega}) \cap C^{1,\beta}(\Omega)$ where the validity of the equations is a consequence of (but actually equivalent to, see Lemma \ref{equa-vs-min}) the local minimality property analogous to \eqref{Ju-inequality} for the energy functionals

\begin{equation}
\label{HIpeps-energy}
J^{p,\varepsilon}_{u_p^\varepsilon}(w;K)= \int_K \dfrac{|\nabla_\varepsilon w|_\varepsilon^p}{p}+w |\nabla_\varepsilon u_p^\varepsilon|_\varepsilon^p \, ,
\end{equation}
where $K \subset \Omega$ is any fixed compact and $\overline{\{ w \neq u_p^\varepsilon\}} \subseteq K$. 

According to the beautiful observation made in \cite{Moser}, for $p>1$ and $0 \leq \varepsilon \leq 1$ the substitution
\begin{equation}
\label{substitution}
u=(1-p) \log v \ ,
\end{equation}
turns solutions $u_p^\varepsilon$ of \eqref{EqForUpeps}-\eqref{EqForUpeps-bis} into $p$-harmonic functions $v_p^\varepsilon : \overline{\Omega} \subseteq \mathbb{H}^n \to \mathbb{R}$ satisfying

\begin{subequations}\label{eq:system2}
\begin{empheq}[left=\empheqlbrace]{align}
  {\rm div}_ \varepsilon \left( | \nabla_\varepsilon v^\varepsilon_p |_\varepsilon^{p-2} \nabla_\varepsilon v^\varepsilon_p \right) = 0 & \quad \textrm{in } \Omega, \label{EqForVpeps}
  \\
  v^\varepsilon_p= 1 & \quad \textrm{on } \partial \Omega. \label{EqForVpeps-bis}
\end{empheq}
\end{subequations}

Thus, on the Heisenberg group both in the Riemannian and in the sub-Riemannian case equations \eqref{EulerForU}, \eqref{EulerForUeps}, \eqref{EqForUpeps}-\eqref{EqForUpeps-bis}, \eqref{substitution}, and \eqref{EqForVpeps}-\eqref{EqForVpeps-bis} give a link between IMCF and asymptotic behaviour of $p$-harmonic functions as $p \downarrow 1$. 

Concerning problem \eqref{EqForVpeps}-\eqref{EqForVpeps-bis}, in the present paper we will only consider for $1<p<Q$ finite energy solutions, i.e., those minimizing
the energy functionals
\begin{equation}
\label{p-dirichlet}
E_p^{\varepsilon}(v)=\frac1p \int_\Omega | \nabla_\varepsilon v |_\varepsilon^p \, ,
\end{equation}
among Sobolev functions in 
\begin{equation}
\label{convex-set}
\dot{W}_{1,\varepsilon}^{1,p}(\Omega):=\{ v \in  \dot{W}_{\varepsilon}^{1,p}(\mathbb{H}^n) \quad s.t. \, \, v \equiv 1 \quad \hbox{a.e. in } \Omega^c \} \, .
 \end{equation}
 Here and in the sequel $\dot{W}_{\varepsilon}^{1,p}(\mathbb{H}^n)$ denotes for any $\varepsilon \in [0,1]$ the homogeneous Sobolev space of functions in $u \in L^{p^*}(\mathbb{H}^n)$,  $p^*=\frac{pQ}{Q-p}$, with distributional gradient $\nabla_{\varepsilon}u$ such that $|\nabla_\varepsilon u|_\varepsilon  \in L^p(\mathbb{H}^n)$, endowed with the norm $\| u \|_{\dot{W}_{\varepsilon}^{1,p}}= \| \,|\nabla_\varepsilon u|_\varepsilon \|_{L^p}$, the case $\varepsilon=0$ corresponding to the homogeneous horizontal Sobolev space $H\dot{W}^{1,p}(\He^n)$ first introduced by Folland and Stein.  Clearly such spaces are well defined because of the horizontal Sobolev inequality corresponding to the embedding $H\dot{W}^{1,p} (\mathbb{H}^n)\subset L^{p^*}(\mathbb{H}^n)$ and actually valid for $1\leq p<Q$ (see \eqref{unifsobolev} below).

Note that such optimal functions solving \eqref{EqForVpeps}-\eqref{EqForVpeps-bis} are nothing but the $p-$capacitary potentials of $ \Omega^c$. For $1<p<Q$ and $\varepsilon \geq 0$ existence and uniqueness for this minimization problem is a standard consequence of the direct method in the Calculus of Variations for the coercive weakly l.s.c. functional \eqref{p-dirichlet} on the weakly closed convex subset \eqref{convex-set} of the reflexive space $\dot{W}_{\varepsilon}^{1,p}(\mathbb{H}^n)$. For $\varepsilon>0$ the $C^{1,\beta}$-regularity up to the boundary is nowadays classical in the regularity theory for $p$-harmonic functions on Riemannian manifolds (see, e.g., \cite{Dibenedetto} and \cite{Lieberman}), with exponent $\beta=\beta(p,\varepsilon) \in (0,1)$. Interior regularity is also known in the sub-Riemannian case $\varepsilon=0$, see,  e.g., \cite{Zhong}, so that $p-$capacitary potentials actually belong to the space $C(\overline{\Omega})$, where continuity up to the boundary is a consequence of the exterior gauge ball condition ($\mathbf{HP_\Omega}$), and with horizontal gradient in $C^{0,\beta'}(\Omega)$ for some $\beta'\in (0,1)$. On the other hand,
continuity up to the boundary of the horizontal gradient is at present not available for \eqref{EqForVpeps} when $\varepsilon=0$ and this is indeed the major technical reason for us to address existence theory for \eqref{EulerForU} through the analysis of the two-parameter family of problems \eqref{EqForUpeps}-\eqref{EqForUpeps-bis} together with their $p-$harmonic counterparts \eqref{EqForVpeps}-\eqref{EqForVpeps-bis}.
\vskip5pt

The first result of the present paper gives a uniform pointwise gradient bound for the $p$-capacitary potentials $v^\varepsilon_p$ in a form of a differential Harnack inequality (sometimes also known as Cheng-Yau's inequality).
 
 \begin{theorem}
 \label{Thm:differential-harnack}
 Let $\Omega \subset \He^n$ be an open set with $C^2$-smooth boundary and bounded complement satisfying the exterior uniform gauge-ball condition ($\mathbf{HP_\Omega}$) with parameter $R_0$. For any $1<p\leq 2$ and $\varepsilon \in [0,1]$ let $v^\varepsilon_p \in \dot{W}^{1,p}_{1,\varepsilon}(\Omega)$ the unique finite energy solution to \eqref{EqForVpeps} corresponding to the minimizer of \eqref{p-dirichlet} in the set \eqref{convex-set}.
Then $v^\varepsilon_p \in C^1(\overline{\Omega})$ for any $\varepsilon \in (0,1]$, it is strictly positive in $\overline{\Omega}$ and there exists a constant $C>0$ depending only on $R_0$ and $Q$ such that   in $ \overline{\Omega}$ we have 
\begin{equation}
 \label{diffharnack}
 |\nabla_\varepsilon v^\varepsilon_p|_\varepsilon \leq \frac{C}{p-1}  v^\varepsilon_p  \, .
 \end{equation}
 Moreover, for $1<p\leq 2$ and $\varepsilon=0$  the corresponding finite energy solution $v^0_p  \in \dot{W}^{1,p}_{1,0}(\Omega)$  to \eqref{EqForVpeps} satisfies  $ v^0_p \in C(\overline{\Omega})$, $\nabla_0 v^0_p \in C(\Omega; \mathbb{R}^{2n})$ and \eqref{diffharnack} holds in $\Omega$ with the same $C>0$ above.
 \end{theorem}
 

In view of \eqref{substitution}, as a straightforward consequence of Theorem \ref{Thm:differential-harnack}  (together with Proposition \ref{eps-limit}) one gets the following uniform gradient bound for solutions to \eqref{EqForUpeps}-\eqref{EqForUpeps-bis}.

\begin{corollary}
\label{Cor:gradbound}
Let $\Omega \subset \He^n$ be an open set with $C^2$-smooth boundary and bounded complement such that $\Omega$ satisfying assumption ($\mathbf{HP_\Omega}$) with parameter $R_0$. For $1<p \leq 2$ and $\varepsilon \in [0,1]$ let $v^\varepsilon_p \in \dot{W}^{1,p}_{1,\varepsilon}(\Omega)$  the solution to \eqref{EqForVpeps} corresponding to the unique minimizers of \eqref{p-dirichlet} in the set \eqref{convex-set} and $u_p^\varepsilon$ given in turn by \eqref{substitution}. Then there exists $C>0$ depending only on $R_0$ and $Q$ such that 
\begin{equation}
\label{globalgradbound}
\| \, |\nabla_\varepsilon u_p^\varepsilon|_\varepsilon \|_{L^\infty (\Omega)} \leq C \, .
\end{equation}
As a consequence, $u_p^\varepsilon \in Lip_\varepsilon(\He^n)$ with $d_\varepsilon$-Lipschitz constant satisfying the uniform bound $Lip_\varepsilon (u_p^\varepsilon) \leq C$.
  \end{corollary}

As already suggested above, the case $\varepsilon=0$ in the theorem is obtained by a limiting argument from the case $\varepsilon \in (0,1]$ for which boundary regularity is available, as the minimizers $\{ v^\varepsilon_p\}$ are strongly convergent in $\dot{HW}^{1,p}(\mathbb{H}^{n})$ to $v^0_p$ as $\varepsilon \to 0$.
Concerning the latter, here the strategy in the proof  of Theorem \ref{Thm:differential-harnack} and in turn Corollary \ref{Cor:gradbound} is similar in the spirit to the one adopted in \cite{Moser}, i.e., one constructs barriers to obtain the bound \eqref{diffharnack} for $v^\varepsilon_p$ at the boundary $\partial \Omega$ and then propagate this bound in the interior by a suitable use of the maximum principle (with some extra care since $\Omega$ is not bounded).
However, compared to \cite{Moser} both the steps mentioned above turn out to be more difficult in the present case. Indeed, explicit solutions in terms of negative powers of the Kor\'anyi norm are available only for $\varepsilon=0$, whence, still under assumption ($\mathbf{HP_\Omega}$), a suitable perturbation of them must be used as barriers along $\partial \Omega$ in the full range $\varepsilon \in [0,1]$. On the other hand, as we explain below, propagating the gradient estimates to the interior turns out to be highly nontrivial in the present case. 

Actually, in contrast with \cite{Moser} (as well as  \cite{KotschwarNi}, \cite{CuiZhao} and \cite{Rigoli}),  we are not able to derive uniform interior gradient estimates working directly with $u^\varepsilon_p$ solving \eqref{EqForUpeps} and applying the maximum principle on the equation satisfied by $|\nabla_\varepsilon u^\varepsilon_p|^2$. Indeed, in contrast with the aforementioned papers where some lower bounds and asymptotic nonnegativity of the curvature at infinity is assumed, on $\He^n$ along the horizontal distribution one has  $Ric_\varepsilon (\cdot)\equiv -\frac{1}{2 \varepsilon^2}|\cdot|_\varepsilon^2$ everywhere. Moreover,  $Ric_\varepsilon (\cdot) \to -\infty $ as $\varepsilon \to 0$ in any horizontal direction, whereas $Ric_\varepsilon (\cdot) \to +\infty$ along the vertical direction, which forces us to devise a completely new approach adapted to the intrinsic structure of $\He^n$ to infer the global interior gradient estimates.
 
Here we find more convenient to work directly with $v^\varepsilon_p$ as a solution to \eqref{EqForVpeps} and to rely on the Bochner method. Thus we first consider the corresponding quantity $ (p-1)^2(T v^\varepsilon_p)^2-C^2 (v^\varepsilon_p)^2$, with the value $C=C(R_0)>0$ ensuring nonpositivity of such a quantity at the boundary by a barrier argument, and then we infer non positivity also in the interior by a simple maximum principle argument using  the elliptic operator $L_{v^\varepsilon_p}$ in \eqref{L_v}, i.e., the linearized $p$-Laplace operator obtained from \eqref{EqForVpeps} at $v^{\varepsilon}_{p}$.  
Once this Cheng-Yau's type inequality for the vertical derivative $T=\partial_t$ is obtained, we next obtain the analugue for the full gradient, i.e., \eqref{diffharnack}, by a similar but more involved maximum principle argument. Here the heart of the matter is to establish an improved Kato-type inequality for the squared norm of the rough Hessian $\left(\overline{D}^2 v^\varepsilon_p \right)_{i,j}=E_j (E_i v^\varepsilon_p)$ which is similar to those improved Kato inequality for the Frobenius norm of the covariant Hessian for $p$-harmonic functions or $p$-harmonic maps as already established in \cite{Naber}, \cite{ChChWei} and \cite{MaMi}. In turn this improved inequality yields a Bochner inequality for $L_{v^\varepsilon_p} \left(|\nabla_\varepsilon v^\varepsilon_p|_{\varepsilon}^2 +(Tv^\varepsilon_p)^2\right)$, the latter crucially relying  on the explicit structure the Ricci tensor $Ric_\varepsilon$ in $\He^n$ in combination with a key computation adapted from \cite{KotschwarNi} giving an explicit identity for $L_{v^\varepsilon_p} \left(|\nabla_\varepsilon v^\varepsilon_p|_{\varepsilon}^2 \right)$ alone.
 
 As for those in \cite{Moser}, both the steps our proof rely in an essential way on the $C^1$-boundary regularity of $p$-capacitary potentials on Riemannian manifolds, as they need to be differentiable at the boundary for the first step and must be of controlled size in tubular neighborhood of $\partial \Omega$ in order to obtain the second step (by running a contradiction argument well in the interior where even higher regularity is available and used). To our knowledge such boundary regularity property is not yet available in the sub-Riemannian case, therefore we have no direct proof of the theorem in the case $\varepsilon=0$ but we can only infer the same bound from \eqref{diffharnack} by letting $\varepsilon \to 0$, as already announced above. The same lack of boundary regularity affects the argument used for the recent existence result \cite[Proposition 5.3]{CuiZhao} for the boundary value problem \eqref{EulerForU} in $\He^1$, so that their proof seems to contain a serious gap and our aim here is to provide an alternative approach to such existence property for IMCF on $\mathbb{H}^n$ through the Riemannian approximation.    

Notice that in the Riemannian case local gradient bounds for $p$-harmonic functions  similar to \eqref{globalgradbound} were first obtained in the literature in \cite{KotschwarNi} on complete manifolds under lower bounds on the sectional curvature (see also \cite{WangZhang} and  \cite{Rigoli} for similar results under the weaker assumption of lower bounds on the Ricci curvature). However, as we have already recalled that along the horizontal distribution of $\He^n$ one has  $Ric_\varepsilon (\cdot)\equiv -\frac{1}{2 \varepsilon^2}|\cdot|_\varepsilon^2 \to -\infty $ as $\varepsilon \to 0$ it seems difficult to obtain a local Cheng-Yau's inequality in the sub-Riemannian case through the Riemannian approximation. Thus, although being global in nature but being uniform w.r.to $\varepsilon$, our gradient bounds here are clearly stronger and they do not follow from previous results in the existing literature. However, as an alternative route to the limiting argument used here, it would be interesting to obtain directly a sub-Riemannian local gradient estimates as done e.g. \cite{Baudoin} in the case of harmonic functions on Carnot groups, although the case of non quadratic energies seems at present awkward.\\

As a consequence of the uniformity w.r.to $\varepsilon$, our estimate allows us to pass to the limit both as $p \to 1$ and as $\varepsilon \to 0$, so obtaining the following existence theorem 
which is the main result of our paper.

\begin{theorem}
\label{Thm:existence}
Let $\Omega \subset \He^n$ be an open set with $C^2$-smooth boundary and bounded complement such that $\Omega$ satisfies an exterior uniform gauge-ball
condition ($\mathbf{HP_\Omega}$) with parameter $R_0$. There exists $C>0$ depending only on $R_0$ and $Q$ and for each $0\leq\varepsilon \leq 1$ there exists a $d_\varepsilon$-Lipschitz function $u^\varepsilon: \mathbb{H}^n \to [0,\infty)$ such that $u^\varepsilon \equiv 0$ in $\mathbb{H}^n \setminus \Omega$, $Lip_\varepsilon(u^\varepsilon) \leq C$ and $u^\varepsilon$ is a weak solution of \eqref{EulerForU} or \eqref{EulerForUeps}.

Moreover, if $B_{R_\ast}(g_0) \subset \Omega^c \subset B_{\bar{R}}(g_0)$ for some $0<R_\ast<2<\bar{R}$ then there exist $\widehat{C}>1$ depending only on $R_\ast$ and $\bar{R}$ and there exists $C_{0}>1$ depending only on $Q$, $R_\ast$, $R_0$, $\bar{R}$ and $ \partial \Omega$, such that we have
\begin{equation}
\label{two-side-bound-u}
  (Q-1) \log \frac{\|  g_0^{-1} *g\|}{\bar{R}} -\log C_{0}  \leq u^\varepsilon (g) \leq  (Q-1)  \log \frac{\| g_0^{-1} * g\|}{R_\ast}  +\varepsilon^4 \log \widehat{C}\, ,
\end{equation}
for every $\varepsilon \in [0,1]$ and every $g \in \overline{\Omega}$. As a consequence, for any $\varepsilon \in [0,1]$ the solutions $u^\varepsilon$ are proper and their level sets $\Sigma^\varepsilon_s:=\{ u^\varepsilon=s\}$ satisfy
\begin{equation}
\label{in-out-spheres}
 \Sigma^\varepsilon_s \cap B_{\varphi_\varepsilon(s)}(g_0)=\emptyset \,  \qquad \hbox{and}  \qquad \Sigma^\varepsilon_s \subset B_{\psi (s)}(g_0) \,  \quad \hbox{for any} \quad s \geq 0 \, ,
 \end{equation}
where $\displaystyle{\varphi_\varepsilon(s)=\frac{R_\ast}{\widehat{C}^{\varepsilon^4/Q-1}} \exp{\left(\frac{s}{Q-1}\right)} }$ and $\displaystyle{\psi (s)=C_0^\frac{1}{Q-1}  \bar{R} \exp{\left(\frac{s}{Q-1} \right)}}$. 
\end{theorem}

The proof of the previous result follows from a standard compactness argument based on Corollary \ref{Cor:gradbound}, passing to the limit the minimality property of the solutions $u^\varepsilon_p$ for the energy functionals $J^{p,\varepsilon}_{u^\varepsilon_p}(\cdot)$ in \eqref{HIpeps-energy} as $p \to 1$ or $(\varepsilon,p) \to (0,1)$ respectively. 
Concerning the two-sided bound \eqref{two-side-bound-u}, the construction of subsolutions for Riemannian $p$-capacitary potential performed in Proposition \ref{SubBarrier} is precise enough to get the upper bound in \eqref{two-side-bound-u} as $p \to 1$ in the whole range $\varepsilon \in [0,1]$ by an iterative argument of self-improving nature on expanding annuli. However, the lower bound in \eqref{two-side-bound-u} and in turn properness of solutions turn out to be more subtle, as we discuss below. 

In the sub-Riemannian case $\varepsilon=0$ properness can be readily inferred from the construction of barriers in terms of explicit $p$-capacitary potentials for each $p \in (1, Q)$ which give exact solutions of \eqref{EqForUpeps}-\eqref{EqForUpeps-bis} in the complement of a Kor\'anyi ball $B_{\bar{R}}(g_0)$ and in turn an exact solution as in \eqref{explicitsolution} as $p \to 1$. 
On the other hand, when $\varepsilon >0$ no such explicit solution is known for $p >1$, therefore in the whole range $\varepsilon \in [0,1]$ we have  to follow a different strategy, exploiting decay properties of the $p$-capacitary potentials $v^\varepsilon_p$ at infinity. Thus, we obtain the lower bound in \eqref{two-side-bound-u} by combining for $q=\frac{Q(p-1)}{Q-p/2}$ an $L^\infty-L^q$ bound on annuli together with a global weak-$L^\sigma$ bound, for $\sigma=\frac{Q(p-1)}{Q-p}>q$, with controlled dependence on $p>1$ and uniformly in $\varepsilon\in (0,1]$. As discussed in more details in Section 5, this approach allows to handle the whole range $\varepsilon \in [0,1]$, the inequalities \eqref{two-side-bound-u} being eventually a consequence of the pointwise decay $ v^\varepsilon_p (g)\simeq \|g\|^{\frac{p-Q}{p-1}}$ as $\| g\| \to \infty$ for the $p$-capacitary potentials which is the natural counterpart of the weak-$L^\sigma$ integrability.

The bounds in \eqref{two-side-bound-u} being sharp, they show that the weak IMCF is trapped between spheres that in view of  \eqref{in-out-spheres} do expand with the same speed depending only on the homogeneous dimension $Q=2n+2$ of $\He^n$, which is striking consequence of the sub-Riemannian nature of the space at infinity even for $\varepsilon \in (0,1]$. Of course the aforementioned bounds leave open the fundamental problem concerning the asymptotic behaviour of $u^\varepsilon$ as $\| g \| \to \infty$. In particular, it would be very interesting to know whether or not both the Riemannian and the sub-Riemannian IMCF evolution of the initial surface $\Sigma_0=\partial \Omega$ is asymptotically equivalent to a Kor\'anyi sphere with exponentially growing radius $\displaystyle{R(s)\sim e^\frac{s}{Q-1}}$ as the time $s \to \infty$ (i.e., whether or not once area-normalized under group dilations the surfaces converge to a Kor\'anyi sphere with fixed radius), in analogy with the results in \cite{Gerhardt} and \cite{Urbas} for the Euclidean case. 

A deeper analysis of the large time behaviour of the flow seems to require a much better understanding of the fine asymptotic behaviour at infinity of the Riemannian $p$-capacitary potentials $v^\varepsilon_p$. In particular, both for a complement of a Kor\'anyi ball and for more general exterior domains, it would be interesting to study their decay properties as $\| g \| \to \infty$ in terms of an asymptotic expansion with an explicit dependence with respect to $p$ as $p \to 1$, in analogy with the analysis performed in \cite{KV} in the Euclidean spaces and in \cite{Benatti} on a class of manifolds with asymptotically nonnegative Ricci curvature, at least for fixed $p>1$. As the fine behaviour of $p$-capacitary potential on Riemannian manifold is gaining considerable interest in recent years, see, e.g., \cite{Mazzieri}, \cite{Mazzieri2}, \cite{Mazzieri3}, \cite{Mazzieri4} and \cite{Rigoli}, we find this aspect worth of further investigation in the case of the Heisenberg group $\He^n$. At the same time in a subsequent paper for the solution given in Theorem \ref{Thm:existence} we aim to investigate  curvature bounds and higher regularity of the level sets $\left\{ \Sigma^\varepsilon_s\right\}_{s\geq 0} \subseteq \He^n$, with $\Sigma^\varepsilon_s =\{u^\varepsilon=s \}$ to be considered as curvature varifolds, together with the validity of the Geroch monotonicity formulas from \cite{Moser3} both in the Riemannian and the sub-Riemannian setting.  

The plan of the paper is as follows. In Section 2 we collect basic notations and facts about the Heisenberg group $\He^n$ and the relevant functional inequalities used here. In Section 3 we construct the relevant barriers for the boundary gradient estimates for $p$-capacitary potentials. In Section 4 we present two auxiliary results for $p$-harmonic functions, i.e., $L^\infty-L^p$ gradient estimates and Harnack inequality, which hold uniformly w.r.to the Riemannian approximation of the sub-Riemannian metric. In Section 5 we derive the weak-$L^\sigma$ bound and in turn the uniform pointwise two-sided bounds for the Riemannian $p$-capacitary potentials which hold globally in the domain. 
In Section 6 we derive the key Bochner inequality associated to linearized $p$-Laplacian in turn based on a refined Kato-type inequality for the rough Hessian. In Section 7 
for the Riemannian $p$-capacitary potentials we derive gradient estimate at the boundary $\partial \Omega$ and we transfer it to the whole $\overline{\Omega}$ obtaining the Cheng-Yau's inequality under a smallness assumption of the gradient at infinity. In Section 8 we finally prove the main results of the paper, i.e., Theorem \ref{Thm:differential-harnack} and Theorem \ref{Thm:existence}. 
\medskip


\section{Basic notations and preliminary results}
In this section we first collect the basic definitions and properties
of the Heisenberg group $\mathbb{H}^n$ which will be needed
throughout the paper, referring to \cite{CapognaBook} for a comprehensive introduction to the subject.

The Heisenberg group $\mathbb{H}^n$ is the non-abelian homogeneous Lie group $(\R^{2n+1}, \ast, \delta_{\lambda})$,
where $\ast$ is the group operation and $\left(\delta_{\lambda}\right)_{\lambda>0}$ is the anisotropic family
of dilations (group homomorphisms). More precisely, given two points $g = (x_1, \ldots, x_n, y_1, \ldots, y_n, t)$ 
and $g' = (x'_1,\ldots,x'_n, y'_1,\ldots, y'_n, t')$ in $\mathbb{H}^n$,
$$g \ast g' = \left( x_1 +x'_1, \ldots, x_n+x'_n, y_1 + y'_1, \ldots, y_n+y'_n,  t+ t' + \tfrac{1}{2} \sum_{i=1}^{n}(x_i y'_i - x'_i y_i) \right),$$
\noindent and 
$$\delta_{\lambda}(g) = \left( \lambda x_1, \ldots, \lambda x_n, \lambda y_1, \ldots, \lambda y_n, \lambda^2 t\right) \qquad \textrm{for every } \lambda >0.$$
As already done in the Introduction, throughout the paper we denote by $Q=2n+2$ the {\it homogeneous dimension} of the group $\mathbb{H}^n$.
Moreover, we will use the notation 
$$L_{g_0}(g) := g_0 \ast g, \quad g\in \mathbb{H}^n,$$
\noindent to denote the left translation by any given point $g_0 \in \He^n$, $g_0=0$ being the identity element. 

The Heisenberg group is the simplest model of the so called Carnot or stratified Lie groups, indeed it
presents a stratified structure at the level of its Lie algebra $\mathfrak{h}=\mathfrak{h}_0 \oplus \mathfrak{h}_1 \simeq T_0 \mathbb{H}^n$. By left translation $\mathfrak{h_0}$ corresponds to the horizontal distribution, i.e., to the span at any $g=(x_1, \ldots, x_n,y_1, \ldots, y_n,t)$ of the horizontal vector fields $\{X_1, \ldots, X_n, Y_1, \ldots, Y_n \}$,
\[ X_i=\partial_{x_i} - \tfrac{y_i}{2}\partial_{t}, \qquad Y_i=\partial_{y_i} +\tfrac{x_i}{2}\partial_{t} .\]

On the other hand at the identity $\mathfrak{h}_1=\mathbb{R} T$, where $T=\partial_t=[X_i,Y_i]$, for $i=1, \ldots, n$, is the vertical vector field and coincides with the commutator of horizontal vector fields. More generally, we have $[X_i,Y_j]=\delta_{i,j} T$ and $[T,T]=[X_i,X_j]=[Y_i,Y_j]=[X_i,T]=[Y_i,T]=0$ for $i,j=1, \ldots, n$, which will be useful below.

For every $\eps >0$ we introduce 
$$\eT := \eps T,$$
\noindent and we consider the family of Riemannian metrics $|\cdot|_\varepsilon$
such that $\{X_i,Y_i,\eT\}$ form an orthonormal basis. Thus, given $V = \sum_{i=1}^{n} \left[a_i X_i + b_i Y_i \right]+ c \eT$ and $W = \sum_{i=1}^{n} \left[a'_i X_i + b'_i Y_i\right] + c' \eT$,
we denote their scalar product and induced norms as
$$\scal{V}{W}_{\eps} = \sum_{i=1}^{n} \left[ a_i a'_i + b_i b'_i \right] + c c', \qquad |V|_\varepsilon^2=\scal{V}{V}_\varepsilon \, ,$$
with the same formula for $\varepsilon=0$ valid with $c=c'\equiv0$. 

Note that for any $\varepsilon>0$ with respect to the corresponding Levi-Civita connection $\nabla^\varepsilon$ one has for arbitrary $U,V,W \in \mathfrak{h}$ the identity
\[ \scal{\nabla^\varepsilon_U V}{W}_\varepsilon=\frac12 \left( \scal{[U,V]}{W}_\varepsilon -\scal{U}{[V,W]}_\varepsilon-\scal{V}{[U,W]}_\varepsilon \right) \, ,\]
whence
\begin{align*}
\nabla^\varepsilon_{T_\varepsilon}T_\varepsilon=0 \, ,  \quad &\nabla^\varepsilon_{X_i}X_j = \nabla^{\varepsilon}_{Y_i}Y_j = 0 \quad \textrm{for all } i,j=1,\ldots,n, \quad \textrm{and} \quad
\nabla^{\varepsilon}_{X_i}Y_j = \nabla^{\varepsilon}_{Y_i}X_j = 0, \quad \textrm{for all } i\neq j,\\
\nabla^{\varepsilon}_{X_i}Y_i = - \nabla^{\varepsilon}_{Y_i}X_i &= \dfrac{T_{\varepsilon}}{2\varepsilon}, \quad \nabla^{\varepsilon}_{X_i}T_{\varepsilon} = \nabla^{\varepsilon}_{T_{\varepsilon}}X_i = -\dfrac{Y_i}{2\varepsilon}, \quad \textrm{and} \quad 
\nabla^{\varepsilon}_{Y_i}T_{\varepsilon} = \nabla^{\varepsilon}_{T_{\varepsilon}}Y_i = \dfrac{X_i}{2\varepsilon}, \quad \textrm{for all } i=1,\ldots,n.
\end{align*}
Notice in particular that for any fixed $j =1, \ldots, 2n$ and for $E_j \in \{X_1,\ldots, X_n, Y_1, \ldots, Y_n\} $ we have
\begin{equation}
\label{constant-square}
 \sum_{i=1}^{2n+1}| \nabla^\varepsilon_{E_i} E_j|_\varepsilon^2=\frac{1}{2 \varepsilon^2}  \qquad \hbox{whereas} \qquad \sum_{i=1}^{2n+1}| \nabla^\varepsilon_{E_i}T_\varepsilon|_\varepsilon^2=\frac{n}{2\varepsilon^2} \, .	
\end{equation}
On the other hand, for $i, j , k \in \{1, \ldots , 2n+1\}$ we have also the identities
\begin{equation}
\label{orthogonal-derivatives}
\langle   \nabla^\varepsilon_{E_i} E_j,  E_k\rangle_\varepsilon=-\langle   E_j,  \nabla^\varepsilon_{E_i} E_k\rangle_\varepsilon \qquad \hbox{and} \qquad \langle   \nabla^\varepsilon_{E_i} E_j,  \nabla^\varepsilon_{E_i} E_k\rangle_\varepsilon=0  \quad \hbox{for} \quad j\neq k \, .	
\end{equation}

For a given smooth function $u:\mathbb{H}^n \to \R$ its Riemannian
gradient is the vector field given by
$$\nabla_{\eps}u = \sum_{i=1}^{n}\left[(X_i u)X_i + (Y_i u)Y_i\right] + (\eT u) \eT \, ,$$
whereas for $\varepsilon=0$ its horizontal gradient is the vector field given by $\nabla_0 u=  \sum_{i=1}^{n}\left[(X_i u)X_i + (Y_i u)Y_i\right] $. 

On the other hand, for any smooth vector field $Z_\varphi=\sum_{i=1}^n \left[ \varphi_i X_i + \varphi_{n+i} Y_i \right]+ \varphi_{2n+1} T_\varepsilon$ associated to a smooth map $\varphi:\mathbb{R}^{2n+1} \to \mathbb{R}^{2n+1}$  its Riemannian divergence is given by
$${\rm div}_{\eps} Z_\varphi = \sum_{i=1}^{n}\left[(X_i \varphi_i + Y_i \varphi_{n+i}\right] + \eT  \varphi_{2n+1} \, ,$$
whereas for $\varepsilon=0$ its horizontal divergence is given by ${\rm div}_{0} Z_\varphi = \sum_{i=1}^{n}\left[(X_i \varphi_i + Y_i \varphi_{n+i}\right] $ .

In the region where $\nabla_0 u \neq 0$ each level set $\{ u=s\}$, $s \in \mathbb{R}$, defines a smooth surface with horizontal normal $\frac{\nabla_0 u}{|\nabla_0 u|_0}$ and {\em horizontal mean curvature} 
\begin{equation}
\label{h-meancurvature}
\mathcal{H}_H:={\rm div}_0 \frac{\nabla_0 u}{|\nabla_0 u|_0}= \sum_{i=1}^n \left [X_i  \frac{X_i u}{|\nabla_0 u|_0}+  Y_i  \frac{Y_i u}{|\nabla_0 u|_0} \right] \, .
\end{equation}
Analogous formula holds for the mean curvature $\mathcal{H}_\varepsilon$ in the region where $\nabla_\varepsilon u \neq 0$, showing in particular that $\mathcal{H}_\varepsilon(g) \to \mathcal{H}_H(g)$ whenever $\nabla_0 u (g) \neq 0$.

In terms of $| \cdot |_\varepsilon$ for $\varepsilon>0$ the Riemannian distance $d_\varepsilon$ between points $g, g' \in \He^n$ is defined by minimizing $\ell(\gamma)=\int_0^1 |\gamma'(s)|_\varepsilon \,ds$ among absolutely continuous paths $\gamma: [0,1] \to \He^n$ joining $g$ and $g'$, with the further restriction when defining the sub-Riemannian distance $d_0$ that $\gamma'$ is a.e. horizontal. Note that $d_\varepsilon \leq d_0$ and indeed $d_\varepsilon \uparrow d_0$ as $\varepsilon \to 0$. For $\varepsilon>0$ we denote by $B^\varepsilon_r(g_0)$ the metric balls $B^\varepsilon_r(g_0)= \{ g \in \He^n \,| \, \, d_\varepsilon(g_0,g)<r\}$, or sometimes simply $B_r^\varepsilon$ the balls with $g_0=0$. 
Note that $B^\varepsilon_r(g_0)=L_{g_0}(B^\varepsilon_r(0))$ and $\delta_\lambda(B^\varepsilon_r(g_0))=B^{\varepsilon \lambda}_{r \lambda}(\delta_\lambda(g_0))$ for any $g_0 \in \He^n$ and for any $r>0$ and $\lambda>0$.

A major role in the paper is played by the {\it Kor\'{a}nyi norm}. For any point $g = (x, y, t) \in \He^n$ with $x = (x_1, \ldots, x_n)$ and $y = (y_1, \ldots, y_n)$,
its Kor\'{a}nyi norm is given by
$$
\|g\| := \left( (|x|^2 + |y|^2)^2 + 16t^2 \right)^{1/4} =\left( |z|^4 + 16t^2 \right)^{1/4}\, ,
$$
where $|\cdot|$ is the Euclidean norm in $\mathbb{R}^n$ and to simplify the notation we denote from now on $(|x|^2+|y|^2)$ simply by $|z|^2$, with $z=(x,y) \in \mathbb{R}^{2n}$.

Given a point $g_0 \in \He^n$ and number $r>0$, we denote by
$B_r(g_0)$ the {\it open} Kor\'{a}nyi ball of center $g_0$ and radius $r$.
More precisely, for any $g_0 \in \mathbb{H}^n$
$$B_r(g_0) := \left\{ g \in \He^n: \|g_{0}^{-1} \ast g\| < r\right\} \, ,$$
so that $L_{\hat{g}} (B_r(g_0))=B_r(\hat{g}*g_0)$ and $\delta_\lambda(B_r(g_0))=B_{r \lambda}(\delta_\lambda(g_0))$ for any $\hat{g} \in \He^n$ and any $\lambda>0$. As it is well known, $d_0(\cdot, \cdot)$ and $\| (\, \cdot \, )^{-1}  \cdot  \|$ give equivalent distances and both induce the Euclidean topology on $\He^n$, so that only Kor\'{a}nyi balls will be used for $\varepsilon=0$.

Here we recall that for an open set with smooth boundary $\Omega \subseteq \mathbb{H}^n$ a point $p \in \partial \Omega$ is characteristic if $T_p \Omega=span \{ X_1(p) , \ldots, X_n(p) , Y_1(p), \ldots, Y_n(p)\}$. In case of a 
 Kor\'{a}ny ball $B_r(0)$ it is straightforward to check that the characteristic points are precisely $g=(0,0,\pm r^2/4)$. 
  
\begin{definition}
\label{exterior-condition}
We say that an open set $\Omega \subseteq \He^n$ satisfy the exterior uniform gauge-ball condition ($\mathbf{HP_\Omega}$) with parameter $R_0>0$ if  there exists  $R_0>0$ such that for any $g \in \partial \Omega$ there exists $g_0\in \He^n$ such that $B_{R_0}(g_0) \subset \Omega^{c}$ and $g \in \partial \Omega \cap \partial B_{R_0}(g_0)$, where $B_r(g_0)$ is a gauge-ball associated with the Kor\'anyi norm $\| \cdot \|$.
\end{definition}
Clearly for open sets with at least $C^2$-smooth boundary such an assumption amounts to require a quantitative one-sided flatness of $\partial \Omega$ at any of its {\em characteristic point} and it is clearly satisfied by the complement of any gauge-ball but also by gauge balls themselves (for the latter observe that for any $g \in \partial B_r(\bar{g})$ under the choice $g_0= g*{\bar{g}}^{-1}*g$ and $R_0=r$ it is straightforward to check that the ball $B_{R_0}(g_0)$ has the desired properties because of Euclidean convexity of gauge balls).

\medskip
Given the Riemannian metrics $|\cdot|_\varepsilon^2$, $\varepsilon>0$, the corresponding volume measures $ d vol_\varepsilon(\cdot)$ are left-invariant and  $ d vol_\varepsilon(\cdot)= \varepsilon^{-1} \mathcal{L}^{2n+1} $, i.e., they coincide with the Lebesgue measure up to a constant factor.  Note that, besides translation invariance with respect to the group law, the Lebesgue measure under dilation satisfies the equality $  \mathcal{L}^{2n+1} (\delta_\lambda(E))=\lambda^Q \mathcal{L}^{2n+1} (E)$ for any measurable subset $E \subseteq \He^n$ and for any $\lambda>0$. For bounded open set $\widetilde{\Omega} \subset \mathbb{H}^n$ with $C^1$-smooth boundary and $\varepsilon > 0$ we will also consider the notion of Riemannian perimeter of $\widetilde{\Omega}$ according to the standard definition, by duality with vector fields $Z_\varphi=\sum_{i=1}^n \left[ \varphi_i X_i + \varphi_{n+i} Y_i \right]+ \varphi_{2n+1} T_\varepsilon$ associated to a smooth map $\varphi:\He^n \to \mathbb{R}^{2n+1}$, given by 
\[ Per_\varepsilon (\widetilde{\Omega})= \sup \left\{  \int_{\widetilde{\Omega}}  {\rm div}_{\eps} Z_\varphi \, d vol_\varepsilon \, \, ; \quad \varphi \in C_0^\infty(\He^n ; \mathbb{R}^{2n+1} ) \, , \, |\varphi| \leq 1 \right\} \, . \]
\noindent As it is well known, due to the Riemannian divergence theorem we have $Per_\varepsilon(\widetilde{\Omega})=\mathcal{H}^{2n}_\varepsilon(\partial \widetilde{\Omega})$, where here $\mathcal{H}^{2n}_\varepsilon( \cdot )$ stands for the Hausdorff measure associated to $d_\varepsilon$, which is also equal to the area computed with respect to the Riemannian volume form coming from the metric induced on $\partial \widetilde{\Omega}$. 
In the same way for $\varepsilon=0$ following, e.g., \cite{CapDanGar2}, \cite{MonSeCa} and \cite{FraSeSeCa}, we will consider the horizontal perimeter $Per_0(\widetilde{\Omega})$ computed with a similar formula involving the horizontal divergence, horizontal vector fields (i.e., those with $\varphi_{2n+1} \equiv 0$) and where integration is with respect to $\mathcal{L}^{2n+1}$. The useful convergence property $\varepsilon Per_\varepsilon(\widetilde{\Omega}) \to Per_0(\widetilde{\Omega}) $ as $\varepsilon \to 0$ will be be discussed in Lemma \ref{perimeter} together with a representation formula for any $\varepsilon \geq 0$ in the spirit of \cite[Section 3]{CapDanGar2} and \cite[Section 5]{MonSeCa}. 

Next, we recall three basic properties of the spaces $(\He^n, d_\varepsilon, d vol_\varepsilon)$ which will be relevant in the sequel. Measures $d vol_\varepsilon(\cdot)$ are {\em uniformly doubling}, i.e., according to \cite[Theorem 3.6]{DoMaRi}, there exists constant $C_D \geq 1$ depending only on $Q$ such that for every $\varepsilon >0$ and $\bar{g} \in  \He^n$ we have
\begin{equation}
\label{doublingproperty}
	|B^\varepsilon_{2r}(\bar{g})| \leq C_D |B^\varepsilon_{r}(\bar{g})| \quad \hbox{ for any } r>0 \, .
\end{equation}
It follows from the doubling property that $ dvol_\varepsilon(\cdot)$ support weak $(1,1)$-Poincar\`e inequality with uniform constants. More precisely, according to \cite[Theorem 4.2]{DoMaRi}, there exist constants $C_P \geq 1$ depending only on $Q$ such that for every $\varepsilon >0$, $\bar{g} \in  \He^n$ and $r>0$ we have
\begin{equation}
\label{unifpoincare}
\frac{1}{|B^\varepsilon_r(\bar{g})|}\int_{B^\varepsilon_r(\bar{g})} |U- U_{B^\varepsilon_r(\bar{g})}| \leq C_P r  \frac{1}{|B^\varepsilon_{3r}(\bar{g})|} \int_{B^\varepsilon_{3r}(\bar{g})} | \nabla_\varepsilon U |_\varepsilon  \, ,  
\end{equation}
where $U \in W^{1,1}_\varepsilon( B^\varepsilon_{3r}(\bar{g}))$ and $U_{B^\varepsilon_r(\bar{g})}$ is the average of $U$ over $B^\varepsilon_r(\bar{g})$.

Next, it is an immediate consequence of the horizontal Sobolev inequality for $p=1$ (i.e., the continuous embedding $H\dot{W}^{1,1} (\mathbb{H}^n)\subset L^{1^*}(\mathbb{H}^n)$ proved in \cite[Theorem 1.1]{CapDanGar2}) that for any $1\leq p <Q$ there exists a constant $C>0$ (depending only on $Q$) such that for $p^*=Qp/(Q-p)$ and for any $\varepsilon \geq 0$ we have 
\begin{equation}
\label{unifsobolev}
	\left(\int_{\He^n} |U(g)|^{p^*} \right)^{1/p^*} \leq C p^* \left( \int_{\He^n} | \nabla_\varepsilon U(g)|^p_\varepsilon  \right)^{1/p}\, \quad \hbox{for any} \quad U \in \dot{W}^{1,p}_\varepsilon (\He^n) \, . 
\end{equation}
It is a remarkable consequence of \eqref{doublingproperty} and \eqref{unifpoincare} that Sobolev inequality also holds in a localized version uniformly on $\varepsilon >0$ (see \cite[Theorem 4.6]{DoMaRi}, see also \cite[Lemma 7.4]{CapoCitti3}). Thus, 
  for any $1<p<Q$ there exists a constant $C_S$ (depending only on $p$ and $Q$) such that for every $\varepsilon >0$, $\bar{g} \in  \He^n$ and $r>0$ 
  \begin{equation}
\label{uniflocalsobolev}
	\left(\frac{1}{|B^\varepsilon_r(\bar{g})|}\int_{B^\varepsilon_r(\bar{g})} |U(g)|^{p^*} \right)^{1/p^*} \leq C_S \, r  \left( \frac{1}{|B^\varepsilon_r(\bar{g})|}\int_{B^\varepsilon_r(\bar{g})} | \nabla_\varepsilon U|^p_\varepsilon  \right)^{1/p}\, \, ,
\end{equation}
for any $U \in \dot{W}^{1,p}_\varepsilon (\He^n)$
with support in some ball $B^\varepsilon_r(\bar{g}) \subseteq \He^n$. 

The proofs of \eqref{doublingproperty}, \eqref{unifpoincare}, and \eqref{uniflocalsobolev} in \cite{DoMaRi} rely on the equivalence  (uniformly on $\varepsilon >0$, see \cite[Corollary 3.9]{DoMaRi} ) of the distance $d_\varepsilon$ with the $\varepsilon$-gauge $\mathcal{N}_\varepsilon$, where $\mathcal{N}_\varepsilon(z,t)=|z|+{\rm min} \left\{ \frac{|t|}{\varepsilon}, \sqrt{|t|} \right\}$ (see also \cite{CapoCitti3} for an alternative  definition of equivalent $\varepsilon$-gauge). Such equivalence makes more transparent, at least on the subgroup $\mathbb{R} \simeq \{ (0,0,t) \in \He^n, \, \, t \in \mathbb{R}\}$, the Euclidean character of $d_\varepsilon$ at distances of order below $\varepsilon$ (i.e., for $|t_2-t_1|=O(\varepsilon^2)$) and its sub-Riemannian character otherwise. Here we adopt still another  definition of a (uniformly) equivalent $\varepsilon$-gauge, namely $\| \cdot \|_\varepsilon$, where 
\begin{equation}
\label{eps-gauge}
\| g\|_\varepsilon:=\left( |z|^4 + 16  \left({\rm min} \left\{ \frac{|t|}{\varepsilon}, \sqrt{|t|} \right\}\right)^4 \right)^{1/4}={\rm min} \left\{  \left( |z|^4 + 16   \frac{t^4}{\varepsilon^4} \right)^{1/4} ,  \|g\| \right\} \, ,
\end{equation}
so that $\| \cdot\|_\varepsilon \to \| \cdot \|$ pointwise on $\He^n$ as $\varepsilon \to 0$. Then, straightforward computations in combination with \cite[Corollary 3.7 and Corollary 3.9]{DoMaRi} lead to the following result that collects relevant properties the $\varepsilon$-gauge defined just above and their consequences.
\begin{lemma}
\label{lemma:eps-gauge}
Let $\varepsilon>0$, $\| \cdot \|_\varepsilon$ as in \eqref{eps-gauge} and $\mathcal{B}^\varepsilon_r(\bar{g})=\{ g \in \He^n \hbox{ s.t. } \| \bar{g}^{-1}* g\|_\varepsilon<r \}$ the corresponding balls. Then there exists $C>  1$ depending only on $Q$ such that the following five statements hold.
\begin{enumerate}
\item $\| g\|_\varepsilon \leq \| g\|$ for any $g \in \He^n$, whence $B_r(\bar{g}) \subseteq \mathcal{B}^\varepsilon_{r}(\bar{g})  $ for any $\bar{g} \in \He^n$ and for any $r>0$. Moreover, $\|g \|_\varepsilon \geq \frac12 \|g\|$ whenever $g \in \He^n$ satisfies $\|g\|_\varepsilon \geq \varepsilon$.
\item For any $\bar{g} \in \He^n$ we have $C^{-1} r^Q\leq  |B_r(\bar{g}) | \leq | B^\varepsilon_r(\bar{g})|$ for any $r>0$ and in turn $C^{-1} r^Q\leq   | B^\varepsilon_{4r}(\bar{g})\setminus \overline{B_{2r}(\bar{g})}|$ for any $r>0$. Furthermore  for any $\bar{g}\in \He^n$ we have $| B^\varepsilon_r(\bar{g})| \leq C r^Q$ for any $r\geq \varepsilon$.
\item $\displaystyle{ C^{-1} \| g \|_\varepsilon \leq d_\varepsilon (0,g) \leq C \| g\|_\varepsilon}$ for any $g \in \He^n$, i.e., $d_\varepsilon$ and $\| \cdot\|_\varepsilon$  are uniformly equivalent and in particular $\mathcal{B}^\varepsilon_{C^{-1}r}(\bar{g}) \subseteq B^\varepsilon_r(\bar{g}) \subseteq \mathcal{B}^\varepsilon_{Cr}(\bar{g})$ for any $\bar{g} \in \He^n$ and for any $r>0$.
\item $\| \cdot\|_\varepsilon $ are piecewise $C^1$ on $\He^n$ and  $\displaystyle{ |\nabla_\varepsilon \| \cdot \|_\varepsilon|_\varepsilon^2 }+\left| T \| \cdot \|^2_\varepsilon \right| \leq C$ on $\He^n$; in particular each $\| \cdot\|_\varepsilon$ is $d_\varepsilon$-Lipschitz. 
\item If $0<\rho_1<\rho_2\leq 2 \rho_1$ and $h: [0,\infty) \to [0,1]$ is a Lipschitz and piecewise linear function, with $\{ h \equiv 1\}=[0,\rho_1]$, $\{ h \equiv 0\}=[\frac{\rho_1+\rho_2}2, \infty)$, and $h$ is linear otherwise, then for any $\bar{g} \in \He^n$ the function $\zeta (g)=h (\| \bar{g}^{-1}*g \|_\varepsilon)$ is piecewise $C^1$ on $\He^n$ with $\zeta \equiv 1$ on $\mathcal{B}^\varepsilon_{\rho_1}(\bar{g})$, $spt \,\zeta \subset \mathcal{B}^\varepsilon_{\rho_2}(\bar{g})$ and
\begin{equation}
\label{cut-off-bounds}
| \nabla_\varepsilon  \zeta |_\varepsilon^2 +\left| T  \zeta \right| \leq \frac{4C}{(\rho_2-\rho_1)^2} \, .
\end{equation}
\end{enumerate}
\end{lemma}
The last result of this section provides a uniform density lower bound for
metric balls $B^\varepsilon_r(\bar{g})$, a fact which will be relevant in connection with the proof of the Harnack inequality.

\begin{lemma}
\label{uniform-density}
Let $\varepsilon>0$ and $B^\varepsilon_r(\bar{g}) \subset \He^n$ a fixed ball. There exists $C\geq 1$ depending only on $Q$ such that
\begin{equation}
\label{unif-density-bound}
|B^\varepsilon_\rho(\hat{g})| \leq C |B^\varepsilon_\rho(\hat{g}) \cap B^\varepsilon_r(\bar{g})| \qquad \hbox{for any} \quad \hat{g} \in B^\varepsilon_r(\bar{g}) \quad \hbox{and} \quad \rho\in (0,2r] \, .
\end{equation}
\end{lemma}
\begin{proof}
Clearly we may assume $\hat{g}\neq \bar{g}$, otherwise the claim holds trivially for $\rho \leq r$ with $C=1$ and in view of \eqref{doublingproperty} for $\rho \in (r,2r]$. For fixed $\varepsilon>0$ the space $(\He^n,d_\varepsilon)$ is a locally compact Riemannian manifold which is complete as a metric space. It follows from the Hopf-Rinow Theorem that for any $  \hat{g} \in B^\varepsilon_r(\bar{g}) $, $\hat{g} \neq \bar{g}$, there exists a (strictly positive) constant speed geodesic $\Gamma: [0,1] \to \He^n$ such that $\Gamma(0)=\bar{g}$, $\Gamma(1)=\hat{g}$ and $d_\varepsilon(\Gamma(t),\Gamma(t^\prime))=|t-t^\prime| d_\varepsilon(\bar{g},\hat{g})$ for any $t,t^\prime \in [0,1]$.
Note that the last identity obviously yields $\Gamma(t) \in B^\varepsilon_r(\bar{g})$ for any $t \in [0,1]$.

Suppose first that $\rho \in [\frac14 r,2r]$, then we also have $\Gamma(t) \in B^\varepsilon_\rho(\hat{g})$ for any $t \in [\frac34 ,1]$. In addition, $B^\varepsilon_\rho(\hat{g}) \subset B^\varepsilon_{3r}(\bar{g})$ by triangle inequality. Since $\tilde{g}:=\Gamma(7/8) \in B^\varepsilon_\rho(\hat{g}) \cap B^\varepsilon_r(\bar{g})$ and $d_\varepsilon(\tilde{g},\bar{g})=7d_\varepsilon(\tilde{g},\hat{g})=\frac78 d_\varepsilon(\bar{g},\hat{g})$, then $B^\varepsilon_{r/{16}}(\tilde{g}) \subset B^\varepsilon_\rho(\hat{g}) \cap B^\varepsilon_r(\bar{g}) $. Indeed  for any $g \in B^\varepsilon_{r/{16}}(\tilde{g})$ we have $d_\varepsilon(g,\bar{g}) \leq d_\varepsilon(g, \tilde{g})+d_\varepsilon(\tilde{g},\bar{g})<r/16+\frac78 r<r$ and, analogously,  $d_\varepsilon(g,\hat{g}) \leq d_\varepsilon(g, \tilde{g})+d_\varepsilon(\tilde{g},\hat{g})<r/16+r/8< \rho$. 

Thus, using the left invariance of the measure together with the uniform doubling property \eqref{doublingproperty} combined with the previous inclusions we have
$
|B^\varepsilon_\rho(\hat{g})|\leq |B^\varepsilon_{3r}(\bar{g})|=|B^\varepsilon_{3r}(\tilde{g})| \leq C |B^\varepsilon_{r/16}(\tilde{g})| \leq C  |B^\varepsilon_\rho(\hat{g}) \cap B^\varepsilon_r(\bar{g})| \, ,
$
which is the desired conclusion.

Finally, suppose instead that $0<\rho< r/4$. Clearly we may assume $d_\varepsilon(\bar{g},\hat{g}) \geq \frac34 r$, because otherwise $B_\rho(\hat{g}) \subset B_r(\bar{g})$ and \eqref{unif-density-bound} trivially holds with $C=1$. Then by continuity there exists $t_1 \in (0,1)$ such that $d_\varepsilon(\Gamma(t_1),\hat{g})=\rho$ and $d_\varepsilon(\Gamma(t),\hat{g})<\rho$ for any $t \in (t_1,1]$. Now we set $t_2=\frac12 (t_1+1)$ and $\tilde{g}=\Gamma(t_2) \in B^\varepsilon_\rho(\hat{g})$, so that we clearly have $B^\varepsilon_\rho(\hat{g}) \subset B^\varepsilon_{2\rho}(\tilde{g})$. Note that $d_\varepsilon(\tilde{g},\bar{g})=d_\varepsilon(\bar{g},\hat{g})-d_\varepsilon(\tilde{g},\hat{g})<r-\rho/2$ and therefore $B^\varepsilon_{\rho/{4}}(\tilde{g}) \subset B^\varepsilon_\rho(\hat{g}) \cap B^\varepsilon_r(\bar{g}) $. Indeed  for any $g \in B^\varepsilon_{\rho/4}(\tilde{g})$ we have $d_\varepsilon(g,\bar{g}) \leq d_\varepsilon(g, \tilde{g})+d_\varepsilon(\tilde{g},\bar{g})<\rho/4+r-\rho/2 <r$ and, analogously,  $d_\varepsilon(g,\hat{g}) \leq d_\varepsilon(g, \tilde{g})+d_\varepsilon(\tilde{g},\hat{g})<\rho/4+\rho/2< \rho$. 

Hence, using as above the left invariance of the measure together with the uniform doubling property \eqref{doublingproperty} combined with the previous inclusions we have
$ |B^\varepsilon_\rho(\hat{g})|\leq |B^\varepsilon_{2\rho}(\tilde{g})|\leq C |B^\varepsilon_{\rho/4}(\tilde{g})| \leq C  |B^\varepsilon_\rho(\hat{g}) \cap B^\varepsilon_r(\bar{g})| \, , $
which is the desired conclusion also in this second case. Note that in both cases we proved inequality \eqref{unif-density-bound} with a constant $C\geq 1$ depending only on the one in \eqref{doublingproperty}, hence only on $Q$ as claimed. 
\end{proof}

\medskip

\section{Construction of barriers}
In this section, in the range  $1<p<Q$ we provide a family of barriers which will be used for establishing the boundary gradient estimate for $p$-capacitary potentials. These functions are negative powers of the Kor\'anyi norm obtained as suitable deformation of those giving singular $p$-harmonic functions w.r.to the sub-Riemannian metric. As a result of a careful choice of the exponent, they will be singular subsolutions for the sub-Riemannian $p$-Laplace operator in the whole $\mathbb{H}^n$ and for the Riemannian $p$-Laplace operator in the complement of any fixed Kor\'anyi ball.

We first introduce a polynomial function that will be used frequently in the computations. For any $g=(x,y,t) \in \He^n$ we set
\begin{equation}\label{N}
N(x,y,t):= \| g\|^4= \left(|x|^2 + |y|^2\right)^2 + 16 t^2 = |z|^4 +16t^2 \, .
\end{equation}

We list below some useful and straightforward computations: 
\begin{align}
\label{XiN} X_i N &= 4x_i |z|^2  - 16 y_i t,\\
\label{YiN} Y_i N &= 4y_i |z|^2 + 16 x_i t,\\
\label{TN} TN &= 32 t.
\end{align}
\begin{align}
\label{XjXiN} X_j X_i N &= 4 \delta_{ij} |z|^2 + 8x_i x_j + 8 y_i y_j,\\
\label{YjYiN} Y_j Y_i N &= 4 \delta_{ij} |z|^2 + 8x_i x_j + 8 y_i y_j,\\
\label{TTN} TTN &= 32,\\
\label{YjXiN} Y_j X_i N &= 8x_i y_j -16\delta_{ij} t - 8x_j y_i,\\
\label{X_jY_iN} X_j Y_i N &= 8x_j y_i +16\delta_{ij} t - 8x_i y_j,\\
\label{XiTN} X_i T N &= TX_i N = -16 y_i,\\
\label{YiTN} Y_iT N &= TY_i N = 16 x_i.
\end{align}
We note that, by \eqref{XiN} and \eqref{YiN}, it follows that
\begin{equation}
\label{hnormN}|\nabla_0 N|_0^2 = \sum_{i=1}^{n} (X_i N)^2 + (Y_i N)^2 = 16 |z|^2 N ,
\end{equation}
\noindent and
\begin{equation}
\label{xXyY} \sum_{i=1}^n x_i X_iN+y_i Y_i N = 4|z|^4.
\end{equation}
Consequently, for $\varepsilon \in (0,1]$ we also have $(T_\varepsilon N)^2 \leq 64 \varepsilon^2 N$ and  
\begin{equation}
\label{nablaepsN}
|\nabla_\varepsilon N|_\varepsilon^2= |\nabla_0 N|_0^2+\varepsilon^2 (TN)^2=16|z|^2N+64\varepsilon^2 16t^2 \leq 16 N^{3/2}+64 \varepsilon^2N \, .
\end{equation}
Finally, by \eqref{XjXiN} and \eqref{YjYiN}, with $i=j$, we have
\begin{equation}\label{HLapN}
\Delta_0 N = {\rm div}_0 \left(  \nabla_0 N \right)= 8 (2+n) |z|^2. 
\end{equation}
\medskip

For every $\alpha \in \R$, we define the function $\Phi_{\alpha}:\He^n\setminus \{ 0\} \to \R$ as
\begin{equation}\label{def:Phi}
\Phi_{\alpha}(x,y,t):= N(x,y,t)^{\alpha} = \left((|x|^2 + |y|^2)^2 + 16t^2\right)^{\alpha}, \qquad (x,y,t) \in \He^n\setminus \{ 0\} \,.
\end{equation}
Given two vector fields $Z_1$ and $Z_2$, we have 
$$Z_i \Phi_{\alpha} = \alpha N^{\alpha -1} Z_i N, \quad \mbox{i=1,2},$$
\noindent and
$$Z_j Z_i \Phi_{\alpha} = \alpha (\alpha-1)N^{\alpha-2}(Z_j N)(Z_i N) + \alpha N^{\alpha-1}Z_j Z_i N.$$ 
Therefore,
\begin{align}
\label{XiPhi} X_i \Phi_{\alpha} &= \alpha N^{\alpha -1} X_i N = \alpha N^{\alpha-1} (4x_i |z|^2 - 16y_i t),\\
\label{YiPhi} Y_i \Phi_{\alpha} &= \alpha N^{\alpha -1} Y_i N = \alpha N^{\alpha-1} (4y_i |z|^2 + 16x_i t),\\
\label{TPhi} \eT \Phi_{\alpha} &= \alpha N^{\alpha -1} \eT N = \alpha N^{\alpha-1} 32 \eps t,\\
\eT X_i \Phi_{\alpha} &= \eps \left( \alpha (\alpha-1) N^{\alpha-2}(128 t (x_i |z|^2- 4y_i t)) - 16 \alpha N^{\alpha-1}y_i \right),\\
\eT Y_i \Phi_{\alpha} &= \eps \left( \alpha (\alpha-1) N^{\alpha-2}(128 t (y_i |z|^2+ 4x_i t)) + 16 \alpha N^{\alpha-1}x_i \right),\\
\label{TTPhi} \eT \eT \Phi_{\alpha} &= \eps^2 \left( \alpha (\alpha-1)N^{\alpha-2} (32t)^2 + 32 \alpha N^{\alpha-1}\right).
\end{align}

\medskip

The following lemma provides smooth explicit sub-Riemannian $p$-subharmonic functions in $\He^n \setminus\{ 0\}$.
 
\begin{lemma}\label{p-Harm}
Let $1<p<Q$, $\alpha <0$ and $\Phi_\alpha$ as defined in \eqref{def:Phi}. Then
\begin{equation}
\label{plapPhialpha}
	{\rm div}_0\left(|\nabla_0 \Phi_{\alpha}|_0^{p-2}\nabla_0 \Phi_{\alpha} \right) = (4 |\alpha|)^{p-2}|z|^p N^{\beta} 4\alpha\left(4\alpha(p-1) -(p-Q)  \right), \qquad \beta = \dfrac{2\alpha(p-1)-p}{2}<0 \,.
\end{equation}

As a consequence, in $\mathbb{H}^n \setminus \{ 0\}$

$${\rm div}_0\left(|\nabla_0 \Phi_{\alpha}|_0^{p-2}\nabla_0 \Phi_{\alpha} \right) \geq 0
\quad \Longleftrightarrow  \quad \alpha \leq \dfrac{p-Q}{4(p-1)}<0 .$$
\end{lemma}

\begin{proof}
To simplify the notation, we will denote $\Phi_{\alpha}$ simply by $\Phi$.
Firstly, notice that
\begin{equation}
\label{dechplap}
\begin{aligned}
{\rm div}_0\left(|\nabla_0 \Phi|_0^{p-2}\nabla_0 \Phi \right) &= \scal{\nabla_0\left( |\nabla_0 \Phi|^{p-2}\right)}{\nabla_0 \Phi}_0 + |\nabla_0 \Phi|^{p-2} \Delta_0 \Phi \\
&=|\nabla_0 \Phi|^{p-4} \left( \frac{p-2}2\scal{\nabla_0\left( |\nabla_0 \Phi|^{2}\right)}{\nabla_0 \Phi}_0 + |\nabla_0 \Phi|^{2} \Delta_0 \Phi \right) \, .
\end{aligned}
\end{equation}
It follows from \eqref{XiPhi} and \eqref{YiPhi} that
\begin{equation}
\label{nabla0Phi2}
\nabla_0 \Phi = \alpha N^{\alpha -1} \nabla_0 N \quad \textrm{and } \quad |\nabla_0 \Phi|_0^2 = 16 \alpha^2 |z|^2 N^{2\alpha -1} , 
\end{equation}
hence
\begin{equation}
\label{nablap-4}
|\nabla_0 \Phi|_0^{p-4} = (4 |\alpha|)^{p-4} |z|^{p-4} N^{(2\alpha-1)(p-4)/2} \, .
\end{equation}

Since for $i=1,\ldots,n$, 
$$X_i\left(|\nabla_0 \Phi|_0^2 \right) = 16 \alpha^2 \left( 2x_i N^{2\alpha -1} + (2\alpha -1)|z|^2 N^{2\alpha -2} X_i N\right)$$
and
$$Y_i\left(|\nabla_0 \Phi|_0^2 \right) = 16 \alpha^2 \left( 2y_i N^{2\alpha -1} + (2\alpha -1)|z|^2 N^{2\alpha -2} Y_i N\right) \, ,$$ 


\noindent combining the previous identities with \eqref{XiN}, \eqref{YiN} and \eqref{hnormN} we obtain

\begin{equation}
\label{firstpdec}
\frac{p-2}2\scal{\nabla_0\left( |\nabla_0 \Phi|_0^{2}\right)}{\nabla_0 \Phi}_0 = (4 \alpha)^3 (4\alpha-1) (p-2) |z|^4  N^{3\alpha -2} \, .
\end{equation}


Let us now consider the second term in the r.h. side of \eqref{dechplap}. By \eqref{hnormN} and \eqref{HLapN} we have 
\begin{equation}
\label{Delta0Phi}
\begin{aligned}
\Delta_0 \Phi &= \alpha (\alpha-1)N^{\alpha-2}|\nabla_0 N|_0^2 + \alpha N^{\alpha-1} \Delta_0 N \\
&= 16 \alpha (\alpha-1)|z|^2 N^{\alpha-1} + 8(2+n) \alpha |z|^2 N^{\alpha -1} \\
&= 8\alpha (2\alpha +n) |z|^2 \, N^{\alpha-1} \, , 
\end{aligned}
\end{equation}
whence
\begin{equation}
\label{secondpdec}
|\nabla_0 \Phi|_0^{2}\Delta_0 \Phi = (4 \alpha)^{3}2  (2\alpha +n) |z|^4 N^{3\alpha-2}.
\end{equation}
Finally, combining \eqref{firstpdec} and \eqref{secondpdec} with \eqref{dechplap} and \eqref{nablap-4} we see that
$$  {\rm div}_0\left(|\nabla_0 \Phi|_0^{p-2}\nabla_0 \Phi \right) = (4 |\alpha|)^{p-2}|z|^p N^{\beta} \left(4\alpha(4\alpha -1)(p-2) + 8\alpha(2\alpha +n) \right),$$
\noindent where
$$\beta := \dfrac{(2\alpha-1)(p-4)}{2} + 3\alpha -2 = \dfrac{2\alpha(p-1)-p}{2} \, ,$$
whence \eqref{plapPhialpha} holds and the conclusion follows.
\end{proof}


We now consider the function $\Phi_{\alpha}$ as in \eqref{def:Phi} for $\alpha \leq \tfrac{p-Q}{4(p-1)}$
and for any given point $g_0 \in \mathbb{H}^n$ and $R_0>0$ we consider the auxiliary function $\Phi^{(g_0, R_0)}_{\alpha} : \He^n \setminus \{ g_0\} \to \R$ given by 
\begin{equation}\label{PhiR0}
\Phi^{(g_0, R_0)}_{\alpha}(x,y,t):= \dfrac{\Phi_{\alpha} \circ L_{g_0^{-1}}(x,y,t)}{R_{0}^{4\alpha}} \, .
\end{equation}

Some elementary properties of $\Phi^{(g_0, R_0)}_{\alpha}$ are given in the following auxiliary lemma.
\begin{lemma}
\label{propertiesPhialpha}
For $p>1$ and $K>0$ let $\alpha=-\frac{K}{p-1}<0$ and $\Phi^{(g_0, R_0)}_{\alpha}$ as in \eqref{PhiR0}. Then $\Phi^{(g_0, R_0)}_{\alpha} = 1$ on $\partial B_{R_0}(g_0)$ and for any $\varepsilon \in (0,1]$ we have
\begin{equation}
\label{normderPhir0}
\| |\nabla_0  \Phi^{(g_0, R_0)}_{\alpha}|_0  \|_{L^\infty (\partial B_{R_0}(g_0))} \leq \frac{4 K}{R_0(p-1)} \, \qquad \hbox{and} \quad \| T  \Phi^{(g_0, R_0)}_{\alpha}  \|_{L^\infty (\partial B_{R_0}(g_0))} \leq \frac{8 K }{R_0^2(p-1)} .
\end{equation}
Moreover, for $\tilde{\Phi}^{(g_0, R_0)}_{\alpha} = \min \{1 ,\Phi^{(g_0, R_0)}_{\alpha}\}$ we have
$ \tilde{\Phi}^{(g_0, R_0)}_{\alpha} \in \dot{W}_\varepsilon^{1,p}(\He^n) $ whenever $K>  \frac{(Q-p)(p-1)}{4p}$ and the Sobolev norms are uniformly bounded for $\varepsilon \in (0, 1]$.
\end{lemma}  
\begin{proof}
Up to translations we may assume $g_0=0$. Next for $\Phi=\Phi^{(0, R_0)}_{\alpha}$  inequality \eqref{nablaepsN} together with \eqref{XiPhi} -\eqref{TPhi} yield
\begin{equation}
\label{boundnablaPhi} | \nabla_\varepsilon \Phi |_\varepsilon^2= \left|\nabla_\varepsilon \left( \frac{N}{R_0^4}\right)^\alpha\right|_\varepsilon^2=\alpha^2 \Phi^2 \frac{|\nabla_0 N|_0^2+\varepsilon^2 (TN)^2
}{N^2} \leq \alpha^2 \Phi^2 (16N^{-1/2} + 64 \varepsilon^2 N^{-1}) \, .
\end{equation}
Evaluating on $\partial B_{R_0}(0)$ inequality \eqref{boundnablaPhi} easily yields \eqref{normderPhir0}. 

Concerning integrability, assuming up to translations $g_0=0$ and denoting $A_j= \{ 2^j R_0 <\| \cdot\| \leq 2^{j+1} R_0 \}$, for any $j \geq 0$, by change of variables we have  $| A_j | \simeq (2^Q)^j$ and in turn
\[ \int_{\| \cdot\| >R_0} \Phi_\alpha^{p^*}  = R_0^{-4\alpha p^*} \sum_{j=0}^\infty \int_{A_j} \left(N^{\alpha}\right)^{p^*}\simeq \sum_{j=0}^\infty (2^{4\alpha p^*})^j |A_j| \simeq \sum_{j=0}^\infty (2^{4\alpha p^*+Q})^j   \, .\]
As a consequence $\tilde{\Phi}^{(g_0, R_0)}_{\alpha} \in L^{p^*}$ if and only if $4\alpha p^*+Q<0$, which holds for $K>0$ precisely as in the statement above.

On the other hand, applying again \eqref{boundnablaPhi} we have $| \nabla_\varepsilon \Phi |_\varepsilon \lesssim \Phi N^{-1/4}$ in $\He^n\setminus B_{R_0}$ uniformly on $\varepsilon \in (0,1]$, whence  
\[ \int_{\| \cdot\| >R_0} | \nabla_\varepsilon \Phi |_\varepsilon^p  \lesssim  \sum_{j=0}^\infty \int_{A_j} \left(N^{\alpha}\right)^{p} N^{-p/4}\simeq \sum_{j=0}^\infty (2^{(4\alpha-1) p})^j |A_j| \simeq \sum_{j=0}^\infty (2^{(4\alpha-1) p+Q})^j   \,  ,\]
so that the integral is finite and uniformly bounded on $\varepsilon \in (0,1]$ whenever $(4\alpha-1) p+Q<0$, which is precisely the same choice of $K>0$ in the statement above.
\end{proof}
Aiming to use functions in \eqref{PhiR0} as barriers in the next section for $\varepsilon>0$, in the following preliminary result we show that they are subsolutions in the complement of Kor\`anyi balls.

\begin{proposition}\label{SubBarrier}
For any $R_0>0$, $1<p<Q$ and $\varepsilon \in (0,1]$ let $K\geq \frac{Q-p}4 + \varepsilon^4 \left(\frac{2}{R_0^2} +  \frac{8 Q}{R_0^4} \right) $. Then for any $g_0 \in \mathbb{H}^n$ the function $\Phi^{(g_0, R_0)}_{\alpha}$ defined in \eqref{PhiR0} with $\alpha (p-1) = -K$ satisfies
\begin{equation}\label{goal}
\mathrm{div}_{\eps}\left(  |\nabla_{\eps} \Phi^{(g_0, R_0)}_{\alpha}|_{\eps}^{(p-2)}\nabla_{\eps}\Phi^{(g_0, R_0)}_{\alpha} \right) \geq 0 \quad \textrm{in } B^{c}_{R_0}(g_0) \, .
\end{equation}
\end{proposition}
\begin{proof}
Without loss of generality, we can assume up to a translation that $g_0 = 0$. To simplify further the
readability, we set $\Phi=\Phi_{\alpha}$ and we define 
\begin{equation}\label{def:Phi'}
  \Phi':= \Phi^{(0, R_0)}_{\alpha} = R_{0}^{-4\alpha} \Phi \, .
  \end{equation}
We further observe that
\begin{equation}
\label{decpepslap}
\begin{aligned}
\mathrm{div}_{\eps}\left(  |\nabla_{\eps} \Phi'|_{\eps}^{(p-2)}\nabla_{\eps}\Phi' \right) &= R_{0}^{-4\alpha (p-1)} \mathrm{div}_{\eps}\left( \left( |\nabla_{\eps} \Phi|_{\eps}^{2}\right)^{(p-2)/2}\nabla_{\eps}\Phi \right) \\
&= R_{0}^{-4\alpha (p-1)} \left[\left(|\egrad \Phi|_\varepsilon^2 \right)^{\tfrac{p-2}{2}} \Delta_{\eps}\Phi + \escal{\egrad \left(|\egrad \Phi|_\varepsilon^{2} \right)^{\tfrac{p-2}{2}}}{\egrad \Phi}\right]\\
&=: R_{0}^{-4\alpha (p-1)} \left( (i) + (ii) \right).
\end{aligned}
\end{equation}
Now, for every smooth vector field $W$, it holds that
\begin{equation}
\label{Waction}
\begin{aligned} 
W &\left(|\egrad \Phi|_\varepsilon^2 \right)^{\tfrac{p-2}{2}} = \tfrac{p-2}{2} \left(|\egrad \Phi|_\varepsilon^2 \right)^{\tfrac{p-4}{2}} W (|\egrad \Phi|_\varepsilon^{2})\\
&= \tfrac{p-2}{2} \left(|\egrad \Phi|_\varepsilon^2 \right)^{\tfrac{p-4}{2}} W \left( \sum_{i=1}^{n}\left[(X_i \Phi)^2 + (Y_i \Phi)^2\right] + (\eT \Phi)^2 \right) \\
&= (p-2) \left(|\egrad \Phi|_\varepsilon^2 \right)^{\tfrac{p-4}{2}} \left[ \sum_{i=1}^{n} \left( X_i \Phi_\alpha WX_i \Phi + Y_i \Phi_\alpha WY_i \Phi \right) + \eT \Phi_\alpha W\eT \Phi \right].
\end{aligned}
\end{equation}
Applying \eqref{Waction} with $W=\nabla_\varepsilon \Phi$ in $B_{R_0}(0)^c$ yields 
\begin{equation*}
\begin{aligned}
(ii) &= (p-2) \left(|\egrad \Phi|_\varepsilon^2 \right)^{\tfrac{p-4}{2}} \sum_{j=1}^{n}\left[(X_j \Phi)\left(\sum_{i=1}^{n} (X_i\Phi X_j X_i \Phi + Y_i \Phi X_j Y_i \Phi) + \eT \Phi X_j \eT \Phi \right)\right]\\
&+ (p-2) \left(|\egrad \Phi|_\varepsilon^2 \right)^{\tfrac{p-4}{2}}\sum_{j=1}^{n}\left[(Y_j \Phi) \left(\sum_{i=1}^{n}(X_i \Phi Y_j X_i \Phi + Y_i \Phi Y_j Y_i \Phi )+ \eT \Phi Y_j \eT \Phi\right)\right]\\
&+ (p-2) \left(|\egrad \Phi|_\varepsilon^2 \right)^{\tfrac{p-4}{2}} \eT \Phi \left( \sum_{i=1}^{n}(X_i \Phi \eT X_i \Phi + Y_i \Phi \eT Y_i \Phi) + \eT \Phi \eT \eT \Phi \right).
\end{aligned}
\end{equation*}
Commuting $X_i$ and $Y_i$ with $T_\varepsilon=\varepsilon T$ and reordering with respect to the powers of $\eps$ we obtain
\begin{equation}
\label{(ii)}
\begin{aligned}
(ii)&= (p-2)\left(|\egrad \Phi|_\varepsilon^2 \right)^{\tfrac{p-4}{2}} \sum_{j=1}^{n}\left[(X_j \Phi)\sum_{i=1}^{n} (X_i\Phi X_j X_i \Phi + Y_i \Phi X_j Y_i \Phi) \right]\\
&+ (p-2) \left(|\egrad \Phi|_\varepsilon^2 \right)^{\tfrac{p-4}{2}}\sum_{j=1}^{n}\left[ (Y_j \Phi) \sum_{i=1}^{n}(X_i \Phi Y_j X_i \Phi + Y_i \Phi Y_j Y_i \Phi ) \right]\\
&+2\eps^2 (p-2)\left(|\egrad \Phi|_\varepsilon^2 \right)^{\tfrac{p-4}{2}} \sum_{j=1}^{n}\left(X_j \Phi T\Phi  X_j T\Phi + Y_j \Phi T\Phi Y_j T\Phi \right) \\
&+ \eps^4 (p-2) \left(|\egrad \Phi|_\varepsilon^2 \right)^{\tfrac{p-4}{2}}  (T\Phi )^2 TT\Phi \\
&=  (p-2)\left(|\egrad \Phi|_\varepsilon^2 \right)^{\tfrac{p-4}{2}} \left[  \frac12\scal{\nabla_0 |\nabla_0 \Phi|^2_0}{\nabla_0 \Phi}_0 +\varepsilon^2 T\Phi T \left( |\nabla_0 \Phi|^2_0\right)+ \varepsilon^4(T\Phi)^2 TT\Phi \right] \,.
\end{aligned}
\end{equation}
On the other hand, using again the identity $T_\varepsilon= \varepsilon T$ we also have 
\begin{equation}
\label{(i)}
\begin{aligned}
(i) &= \left(|\egrad \Phi|_\varepsilon^2 \right)^{\tfrac{p-4}{2}} \left( |\nabla_0 \Phi|_0^{2} + \eps^2 (T\Phi)^2 \right) (\Delta_0 \Phi + \eps^2 TT\Phi)\\
&= \left(|\egrad \Phi|_\varepsilon^2 \right)^{\tfrac{p-4}{2}} \left[ |\nabla_0 \Phi|_0^{2}\Delta_0 \Phi+ \varepsilon^2 \left( |\nabla_0 \Phi|_0^{2}TT\Phi+ (T\Phi)^2 \Delta_0 \Phi\right) +\varepsilon^4 (T\Phi)^2TT\Phi \right]\,.
\end{aligned}
\end{equation}

Thus, \eqref{decpepslap}, \eqref{(i)} and \eqref{(ii)} yield 
\begin{equation}
\label{decpepslap2}
\mathrm{div}_{\eps} \left( |\nabla_{\eps} \Phi|_{\eps}^{(p-2)}\nabla_{\eps}\Phi \right) =(i)+(ii)= \left(|\nabla_{\eps}\Phi|_{\eps}^2 \right)^{\tfrac{p-4}{2}}\left[ (I) + (II) + (III)  \right],
\end{equation}
\noindent where
\begin{equation}
\label{1-2-3}
\begin{aligned}
(I):&= |\nabla_0 \Phi|_0^{2}\Delta_0 \Phi+\frac{(p-2)}2   \scal{\nabla_0 |\nabla_0 \Phi|^2_0}{\nabla_0 \Phi}_0 \, ,\\
(II):&= \varepsilon^2 \left( |\nabla_\varepsilon \Phi|_0^{2}TT\Phi+ (T\Phi)^2 \Delta_0 \Phi    +(p-2)T\Phi T \left( |\nabla_0 \Phi|^2_0\right) \right) \, , \\
(III):&= \eps^4 (p-1)  (T\Phi)^2 (TT\Phi) \, .
\end{aligned}
\end{equation}

\noindent Concerning the first term in \eqref{1-2-3} identities \eqref{firstpdec} and \eqref{secondpdec} give
\begin{equation}
\label{ineq(1)}
(I)= (4\alpha)^3 |z|^4 N^{3\alpha-2} (4\alpha (p-1)+Q-p)
>0 \, ,
\end{equation}
\noindent for any $K>\frac{Q-p}4$.

Next for $\alpha<0$ and restricting to $\He^n\setminus B_{R_0}$ we bound separately from below the terms $(II)$ and $(III)$ in \eqref{1-2-3}.
To this end, notice that on $\He^n\setminus B_{R_0}(0)$ we have $N \geq R_0^4$ an therefore 
\begin{equation}\label{StimaNR0}
-N^{\gamma} \geq -R_0^{4\gamma}, \quad \textrm{for every } \gamma<0.
\end{equation}

\noindent Concerning $(II)$, using \eqref{TPhi}, \eqref{TTPhi}, \eqref{nabla0Phi2} and \eqref{Delta0Phi} we compute all the terms and obtain
\begin{align*}
(T\Phi)^2 \Delta_0 \Phi &= 16 (8\alpha)^3 (2\alpha +n) N^{3\alpha -3}\, |z|^2 \, t^2,\\
|\nabla_0 \Phi|^2 TT\Phi &= (8 \alpha)^3  N^{3\alpha -3}\, |z|^2 (|z|^4 + (2\alpha-1)16t^2) \, ,\\
(p-2)T\Phi T \left( |\nabla_0 \Phi|^2_0\right)&= 32  (p-2) (8\alpha)^3 (2\alpha -1)N^{3\alpha -3} \, |z|^2 \, t^2 \, ,
\end{align*}  
whence
\begin{equation}
\label{eq:II}
\begin{aligned}(II)&=\varepsilon^2 (8 \alpha)^3 N^{3\alpha-3}|z|^2 \left[ 16 (2\alpha+n) t^2+ |z|^4+16(2\alpha-1)t^2 +32(p-2)(2\alpha-1) t^2\right]\\
&=\varepsilon^2 (8 \alpha)^3 N^{3\alpha-3}|z|^2 \left[ |z|^4+ 8t^2(Q+4(2\alpha-1)(p-1))\right] \, .
\end{aligned}
\end{equation}
Note that for $p>1$ we have $Q+4(2\alpha-1)(p-1)<0$ by our choice of $\alpha$, hence combining on $\He^n\setminus B_{R_0}$ the last identity with \eqref{N} and \eqref{StimaNR0} we obtain uniformly w.r.to $\varepsilon \in (0,1]$
\begin{equation}
\label{ineq(2)}
(II) \geq  \varepsilon^2 (8 \alpha)^3 |z|^6 N^{3\alpha-3}  \geq \varepsilon^4 (4 \alpha)^3 |z|^4 N^{3\alpha-2} 8 N^{-1/2} \geq \varepsilon^4 (4 \alpha)^3 |z|^4 N^{3\alpha-2} 8 R_0^{-2}\, .
\end{equation} 

\noindent On the other hand, applying \eqref{TPhi} and \eqref{TTPhi} we get
\begin{equation}
\label{eq:III}
\begin{aligned}(III) &= \eps^4 (p-1) (T\Phi)^2 (TT\Phi) = \eps^4 (p-1) (32 \alpha N^{\alpha-1}t)^2 \left[32 \alpha N^{\alpha-2} (N+(\alpha-1)32t^2)\right]\\
&= \eps^4 (32 \alpha)^3 (p-1) N^{3\alpha -4} t^2 \left[ |z|^4 + \underbrace{16 (2\alpha -1)t^2}_{\leq 0} \right], 
\end{aligned}
\end{equation}
therefore on $\He^n\setminus B_{R_0}$ for all $\varepsilon \in (0,1]$ by \eqref{N} and \eqref{StimaNR0} we obtain
\begin{equation}
\label{ineq(3)}
(III)\geq \varepsilon^4 8^3 (4\alpha)^3 (p-1) N^{3\alpha-4} t^2 |z|^4\geq \varepsilon^4 32 (4\alpha)^3 (p-1)  |z|^4 N^{3\alpha-2}   N^{-1}\geq  \varepsilon^4 (4\alpha)^3  |z|^4 N^{3\alpha-2} 32 Q   R_0^{-4}\,.
\end{equation}

Combining \eqref{decpepslap2}, \eqref{ineq(1)}, \eqref{ineq(2)}, and \eqref{ineq(3)} in the whole $\He^n\setminus B_{R_0}$ we infer

 \[ \mathrm{div}_{\eps} \left( |\nabla_{\eps} \Phi|_{\eps}^{(p-2)}\nabla_{\eps}\Phi \right) \geq(4\alpha^3) |z|^4 N^{3\alpha-2} \left(|\nabla_{\eps}\Phi|_{\eps}^2 \right)^{\tfrac{p-4}{2}}\left[ -4K+Q-p + \varepsilon^4 \left(\frac{8}{R_0^2} +  \frac{32 Q}{R_0^4} \right) \right] \, ,\]
 whence the conclusion follows from the choice of $K$ as $\alpha<0$.
 \end{proof}
 \begin{remark}
 \label{expchoice}
{\rm  Note that in order to construct subsolutions in a Riemannian case one has to adjust the exponent comparing to the one in the sub-Riemannian case treated in Lemma \ref{p-Harm}. However, the difference of the exponents corresponding to the two choices can be made smaller and smaller as $\varepsilon$ gets smaller or $R_0>0$ gets larger. We will exploit this fact in the proof of Proposition \ref{univ-pt-lbound-infinity}. On the other hand, the construction of Riemannian supersolutions via negative powers of the Kor\'{a}nyi norm seems to be impossible. Indeed, considering \eqref{ineq(1)}, \eqref{eq:II} and \eqref{eq:III} restricted to the vertical axis (i.e. for $|z|=0$) clearly gives the opposite of the desired sign. A similar behaviour will be faced even considering the $\varepsilon$-gauge defined in \eqref{eps-gauge}, as they agree with the Kor\'{a}nyi norm for $|t|\geq \varepsilon$.} 
 \end{remark}

\section{Two uniform estimates in multiscale geometries.}

In this section we present two results concerning $p$-harmonic functions on the Riemannian Heisenberg group $(\He^n, d_\varepsilon)$ which for fixed $p \in (1,Q)$ hold uniformly w.r.to the parameter $\varepsilon \in (0,1]$, namely, an $L^p-L^\infty$ estimates for their gradient in Proposition \ref{uniform-capogna-citti} and a local Harnack inequality in Proposition \ref{garofalo}. 
As it will be clear in the next sections, the these two results will be essential in Proposition \ref{domain-convergence} in order to establish decay properties at infinity for the gradient of Riemannian $p$-capacitary potentials in exterior domains and in turn to prove the gradient bounds announced in Theorem \ref{Thm:differential-harnack} and Corollary \ref{Cor:gradbound}. 

Note that the $L^p-L^\infty$ estimates for their gradient discussed here is the counterpart of the one obtained in \cite{Zhong} for the horizontal gradient of $p$-harmonic functions on the sub-Riemannian Heisenberg group. The result presented here is essentially contained in \cite{CapoCitti4}, although not explicitly stated there. Indeed, in \cite{CapoCitti4} the key estimate concerning a weighted Caccioppoli inequality for the gradient is obtained and here we fill the missing details to derive the claim through uniform local Sobolev inequality from \cite{DoMaRi} and Moser iteration.
Similarly, the uniform local Harnack inequality discussed here is the Riemannian counterpart of the one proved in \cite{CapDanGar} for the sub-Riemannian case (see also \cite{CapoCitti3} for analogous results concerning the Riemannian approximation for parabolic equations modeled on the $p$-Laplacian for $p\geq 2$). Here, following closely the argument in \cite{CapDanGar} for the sub-Riemannian case, we detail how to obtain the Harnack inequality for $p$-harmonic functions in the Riemannian case uniformly on $\varepsilon \in (0,1]$ and with a quantitative dependence on $p$ , relying on the main results in \cite{DoMaRi} (see also \cite{CapoCitti3}), i.e., the uniform doubling property and the uniform Poincar\'e and Sobolev inequalities reviewed in Section 2 in combination with Moser iteration and John-Nirenberg inequality.

Since we were not able to find a reference in the literature and although the arguments are essentially standard and well-known to the experts, for both the aforementioned results we sketch in the sequel the proofs mainly for the reader's convenience, so that uniformity w.r.to $\varepsilon \in (0,1]$ will be transparent and the dependence on $p$ will be made explicit when needed.

\begin{proposition}
\label{uniform-capogna-citti}
Let  $1<p<Q$, $\varepsilon \in (0,1]$, and let $\Omega \subset \He^n$ be an open set. There exist constants $C>0$ and $\theta> 1$ depending only on $p$ and $Q$ such that for every $\varepsilon \in (0,1]$  and every $v=v^\varepsilon_p \in \dot{W}
^{1,p}_{ \varepsilon,loc} (\Omega)$ which is weakly $p$-harmonic in $\Omega$ we have
\begin{equation}
\label{capogna-citti-bound} 
\| |\nabla_\varepsilon v|_\varepsilon  \|_{L^\infty(B_{\theta^{-1}r}(\bar{g}))} \leq C  \left( \frac{1}{|B_{ r}^\varepsilon(\bar{g})|}  \int_{ B_{r}^\varepsilon(\bar{g})} |\nabla_\varepsilon v|_\varepsilon^p \right)^{1/p}
	\end{equation}
whenever $B_{2  r}^\varepsilon (\bar{g}) \subseteq  \Omega$, for some $\bar{g} \in \Omega$ and $r >0$. 
\end{proposition}
\begin{proof}
The proof relies essentially on the uniform Caccioppoli inequality for the gradient of $p$-harmonic functions proved in \cite[Theorem 5.3]{CapoCitti4} combined with the with the Moser iteration technique. However, instead of working with the metric balls $B^\varepsilon_r$, here we prefer to perform the iteration using the balls $\mathcal{B}^\varepsilon_r$  associated to the $\varepsilon$-gauge $\| \cdot \|_\varepsilon$, so that well-behaved cut-off functions are available in view of Lemma \ref{lemma:eps-gauge}-{\it (5)}. 

 For fixed $\varepsilon \in (0,1]$ and $B^\varepsilon_{2r}(\bar{g}) \subseteq \Omega$ we consider the auxiliary domain $\Omega^\prime:=\mathcal{B}^\varepsilon_{2\bar{C}^{-1}r}(\bar{g}) \subseteq B^\varepsilon_{2r}(\bar{g})$, i.e., a ball associated to the $\varepsilon$-gauge $\| \cdot \|_\varepsilon$ discussed in Lemma \ref{lemma:eps-gauge} for a constant $\bar{C}>1$ depending only on $Q$ obtained there. Thus, choosing $\theta=2 \bar{C}^2>1$ and setting for brevity $\bar{r}=(2\bar{C})^{-1}r$, Lemma \ref{lemma:eps-gauge}-{\it (3)} yields $B^\varepsilon_{\theta^{-1} r}(\bar{g}) \subseteq \mathcal{B}^\varepsilon_{\bar{r}}(\bar{g})  \subseteq \mathcal{B}^\varepsilon_{2\bar{r}}(\bar{g})\subseteq B^\varepsilon_{r}(\bar{g}) $ and the lemma is proved once we show that 
 \begin{equation}
\label{capogna-citti-bound-bis} 
\| |\nabla_\varepsilon v|_\varepsilon  \|_{L^\infty(\mathcal{B}^\varepsilon_{\bar{r}}(\bar{g}))} \leq C  \left( \frac{1}{|\mathcal{B}_{ 2\bar{r}}^\varepsilon(\bar{g})|}  \int_{ \mathcal{B}_{2\bar{r}}^\varepsilon(\bar{g})} |\nabla_\varepsilon v|_\varepsilon^p \right)^{1/p}\, ,
\end{equation}
\noindent for some constant $C>0$ depending only on $p$ and $Q$ (note that the volumes of the balls above are all comparable uniformly on $\varepsilon \in (0,1]$ and $r>0$ because of \eqref{doublingproperty}).

In order to prove \eqref{capogna-citti-bound-bis} we rely on the approach in \cite{CapoCitti4} and proceed by a further regularization of \eqref{EqForVpeps} and  \eqref{p-dirichlet}, thus for any $\sigma>0$ we consider the Euler-Lagrange equations
\begin{equation}\label{EqFor-sigmaVpeps}
\left\{ \begin{array}{ll}
            {\rm div}_ \varepsilon \left(  \left( \sigma+| \nabla_\varepsilon w |^2_\varepsilon \right)^{(p-2)/2} \nabla_\varepsilon w \right) = 0 & \textrm{in } \Omega^\prime,\\
						w= v & \textrm{on } \partial \Omega^\prime \, ,
        \end{array}\right.
\end{equation}
corresponding to the minimization of the energy functional
\begin{equation}
\label{sigma-p-energy}
E^{\varepsilon,\sigma}_p (w)=\frac{1}{p} \int_{\Omega^\prime} \left( \sigma+| \nabla_\varepsilon w|^2_\varepsilon \right)^{p/2} \, ,
\end{equation}
among Sobolev functions $W^{1,p}_{v,\varepsilon}(\Omega^\prime)= \left\{  w \in W^{1,p}_\varepsilon (\Omega^\prime) \quad \hbox{s.t.} \, \, w =v  \, \, \hbox{a.e. on} \, \, \partial \Omega^\prime \right\}$.
Note that by definition of $\| \cdot \|_\varepsilon$ the domain $\Omega^\prime$ has piecewise smooth boundary, hence $W^{1,p}_{v,\varepsilon}(\Omega^\prime)$ is a well defined closed convex subset of $W^{1,p}_\varepsilon (\Omega^\prime)$ in view of standard trace theory for Sobolev functions.

 By the direct method in the Calculus of Variations equations \eqref{EqFor-sigmaVpeps} do have unique solutions $w^\sigma=w^{\varepsilon, \sigma}_p$ which are minimizers of \eqref{sigma-p-energy} in the class above (here uniqueness follows from the strict convexity of the energy functionals). Thus, $E^{\varepsilon,\sigma}_p (w^\sigma) \leq E^{\varepsilon,\sigma}_p (v^\varepsilon_p) \leq C$ uniformly on $\sigma$, so that $w^\sigma \to w^* $ as $\sigma \to 0$ (possibly up to subsequences) weakly in $W^{1,p}_\varepsilon$ and strongly in $L^p$ , in particular $w^* \in W^{1,p}_{v,\varepsilon}$ by weak continuity of the trace operator. Since $E^{\varepsilon,\sigma}_p (v^\varepsilon_p) \to E^\varepsilon_p (v^\varepsilon_p)$ as $\sigma \to 0$, by weak lower semicontinuity of $E_p^\varepsilon(\cdot)$ and monotonicity of the map $\sigma\to E^{\varepsilon,\sigma}_p (w)$ for any fixed $w$ we conclude that $E^\varepsilon_p (w^*) \leq E^\varepsilon_p (v^\varepsilon_p)$, so that $w^*=v^\varepsilon_p$ because the latter is the unique minimizer  of $E^\varepsilon$ in the class above again in view of the strict convexity of $E_p^\varepsilon(\cdot)$. Thus,  $E^{\varepsilon}_p (w^\sigma) \to E^\varepsilon_p (v^\varepsilon_p)$ as $\sigma \to 0$, so that $w^\sigma \to v^\varepsilon_p$ as $\sigma \to 0$ strongly in $W^{1,p}_\varepsilon$ and (possibly up to subsequences)  a.e. convergence of the gradients holds. 
 Taking into account the previous regularization scheme and its convergence properties as $\sigma \to 0$, inequality \eqref{capogna-citti-bound-bis} follows once we prove that
  \begin{equation}
\label{capogna-citti-bound-ter} 
\| |\nabla_\varepsilon w^\sigma|_\varepsilon  \|_{L^\infty(\mathcal{B}^\varepsilon_{\bar{r}}(\bar{g}))} \leq C  \left( \frac{1}{|\mathcal{B}_{ 2\bar{r}}^\varepsilon(\bar{g})|}  \int_{ \mathcal{B}_{2\bar{r}}^\varepsilon(\bar{g})} \left( \sigma+|\nabla_\varepsilon w^\sigma |^2_\varepsilon\right)^{p/2} \right)^{1/p}\, ,
\end{equation}
\noindent for some constant $C>0$ depending only on $p$ and $Q$ (in particular, independent of $\sigma$). 
 
In order to obtain \eqref{capogna-citti-bound-ter} we rely on \cite[Theorem 5.3]{CapoCitti4}, so that for any $\beta \geq 2$ and any $\eta \in C^\infty_0(\Omega^\prime)$ we have
\begin{equation}
\label{gradient-caccioppoli}
\int_{\Omega^\prime} \eta^2 \left( \sigma+|\nabla_\varepsilon w^\sigma |^2_\varepsilon\right)^{\frac{p-2+\beta}2} |\overline{D}^2_\varepsilon w^\sigma |^2  \leq C \beta^{10} \left( \| | \nabla_\varepsilon \eta |_\varepsilon  \|_{L^\infty}^2 +\| \eta T \eta  \|_{L^\infty}\right)    \int_{spt \, \eta } \left( \sigma+|\nabla_\varepsilon w^\sigma |^2_\varepsilon\right)^{\frac{p+\beta}2} \, ,
\end{equation}
where $C>0$ depends only on $p$ and $Q$ and $|\overline{D}^2_\varepsilon w^\sigma |^2$ is the sum of the squared second derivatives w.r. to the orthonormal frame $\{ X_1, \ldots, X_n, Y_1 , \ldots, Y_n ,T_\varepsilon \}$.

Now notice that $\Omega^\prime= \mathcal{B}_{ 4\bar{r}}^\varepsilon(\bar{g})$ and choosing $\bar{r}\leq \rho_1<\rho_2 \leq 2 \bar{r}\leq 2 \rho_1  \leq 4\bar{r}$ we have $\mathcal{B}_{ \rho_1}^\varepsilon(\bar{g}) \subset \mathcal{B}_{ \rho_2}^\varepsilon(\bar{g}) \subseteq \Omega^\prime$, so that by a simple limiting argument inequality \eqref{gradient-caccioppoli} still holds for $\eta=\zeta$, the Lipschitz cut-off function given in Lemma \ref{lemma:eps-gauge}-{\it (5)}. Thus, combining \eqref{gradient-caccioppoli} with the elementary inequality $| \nabla_\varepsilon \left( \sigma+  | \nabla_\varepsilon w^\sigma|^2_\varepsilon \right) |^2_\varepsilon \leq 4 \left( \sigma+| \nabla_\varepsilon w^\sigma|^2_\varepsilon  \right)|\overline{D}^2_\varepsilon w^\sigma |^2 $ we obtain for any $\beta \geq 2$

\[
\int_{\mathcal{B}_{ \rho_2}^\varepsilon(\bar{g}) } \zeta^2\left| \nabla_\varepsilon  \left(\sigma+|\nabla_\varepsilon w^\sigma |^2_\varepsilon\right)^{\frac{p+\beta}4} \right|_\varepsilon^2 \leq \frac{C (1+\beta)^{12}}{(\rho_2-\rho_1)^2}   \int_{ \mathcal{B}_{ \rho_2}^\varepsilon(\bar{g})} \left( \sigma+|\nabla_\varepsilon w^\sigma |^2_\varepsilon\right)^{\frac{p+\beta}2} \, .
\]
On the other hand still by Lemma \ref{lemma:eps-gauge}-{\it (5)} we have for any $\beta \geq 2$
\[ \int_{ \mathcal{B}_{ \rho_2}^\varepsilon(\bar{g})} |\nabla_\varepsilon \zeta|_\varepsilon^2 \left( \sigma+|\nabla_\varepsilon w^\sigma |^2_\varepsilon\right)^{\frac{p+\beta}2} \leq \frac{C }{(\rho_2-\rho_1)^2}   \int_{ \mathcal{B}_{ \rho_2}^\varepsilon(\bar{g})} \left( \sigma+|\nabla_\varepsilon w^\sigma |^2_\varepsilon\right)^{\frac{p+\beta}2} \, ,
\]
 which combined with the previous inequality and Leibnitz's rule easily yields for any $\beta \geq 2$
\begin{equation}
\label{gradient-caccioppoli-2}
\int_{\mathcal{B}_{ \rho_2}^\varepsilon(\bar{g}) } \left| \nabla_\varepsilon  \left\{ \zeta \left(\sigma+|\nabla_\varepsilon w^\sigma |^2_\varepsilon\right)^{\frac{p+\beta}4} \right\} \right|_\varepsilon^2 \leq \frac{C (1+\beta)^{12}}{(\rho_2-\rho_1)^2}   \int_{ \mathcal{B}_{ \rho_2}^\varepsilon(\bar{g})} \left( \sigma+|\nabla_\varepsilon w^\sigma |^2_\varepsilon\right)^{\frac{p+\beta}2} \, .
\end{equation}

From now on we set for brevity $\Psi_\sigma=\left( \sigma+|\nabla_\varepsilon w^\sigma |^2_\varepsilon\right)^{\frac12} $. Since $ spt \, \zeta \subset\mathcal{B}_{ \rho_2}^\varepsilon(\bar{g}) \subseteq  \Omega^\prime \subseteq B^\varepsilon_{2r}(\bar{g})$,
from the localized Sobolev inequality \eqref{uniflocalsobolev} (with $p=2$ and on the ball $B^\varepsilon_{2r}(\bar{g})$) we infer for any $\beta \geq 2$

  \[
	\left( \frac{1}{|B^\varepsilon_{2r}(\bar{g})| }\int_{B^\varepsilon_{2r}(\bar{g})} \left|  \zeta \Psi_\sigma^{\frac{p+\beta}2}\right|^{\frac{2Q}{Q-2}} \right)^{\frac{Q-2}{2Q}} \leq 2r C \,  \left( \frac{1}{|B^\varepsilon_{2r}(\bar{g})| }\int_{B^\varepsilon_{2r}(\bar{g})}  \left| \nabla_\varepsilon  \left\{ \zeta \Psi_\sigma^{\frac{p+\beta}2} \right\} \right|_\varepsilon^2   \right)^{1/2}\, \, ,
\]
where the constant $C>0$ depends only on $p$ and $Q$. 

Now we observe that $\mathcal{B}_{ \bar{r}}^\varepsilon(\bar{g}) \subset \mathcal{B}^\varepsilon_{\rho_1} (\bar{g}) \subset  spt \, \zeta \subset \mathcal{B}_{ \rho_2}^\varepsilon(\bar{g})  \subseteq \Omega^\prime \subseteq B^\varepsilon_{2r}(\bar{g})=B^\varepsilon_{4 \bar{C} \bar{r}}(\bar{g})$, hence all the balls have uniformly equivalent volumes in view of \eqref{doublingproperty}. Thus, combining the last inequality with \eqref{gradient-caccioppoli-2} and taking into account the properties of $\zeta$, for any $\beta \geq 2$ and $\bar{r}\leq \rho_1<\rho_2 \leq 2 \bar{r}\leq 2 \rho_1  \leq 4\bar{r}$ we finally get
\begin{equation}
\label{reverse-holder}
\left(  \frac{1}{|\mathcal{B}^\varepsilon_{\rho_1}(\bar{g})|}\int_{\mathcal{B}^\varepsilon_{\rho_1}(\bar{g})}  \left( \Psi_\sigma^{p+\beta} \right)^{\frac{Q}{Q-2}} \right)^{\frac{Q-2}{Q}} \leq \frac{C \rho_2^2(1+\beta)^{12}}{(\rho_2-\rho_1)^2}   \frac{1}{|\mathcal{B}^\varepsilon_{\rho_2}(\bar{g})|}\int_{ \mathcal{B}_{ \rho_2}^\varepsilon(\bar{g})} \Psi_\sigma^{p+\beta} \, 
\, ,
\end{equation}
where the constant $C>0$ depends only on $p$ and $Q$.

Next we perform Moser iteration, i.e., for $i \in \mathbb{N}$ we repeatedly use \eqref{reverse-holder} with the choice $(\rho_1,\rho_2,\beta)=(\tilde{\rho}_{i+1}, \tilde{\rho}_i,\beta_i)$, where $\tilde{\rho}_i=\bar{r}(1+2^{-i-1})$ for $i \geq 0$ and $\beta_{i+1}=\frac{Q}{Q-2} \beta_i+\frac{2p}{Q-2}$, $\beta_0=2$. Note that $\{ \beta_i\}$ is positive, increasing and $\beta_i \to +\infty$ as $i \to \infty$, moreover $p+\beta_{i+1}=(p+\beta_i) \frac{Q}{Q-2}$ and $\rho_2 -\rho_1= \tilde{\rho}_i - \tilde{\rho}_{i+1}= \bar{r} 2^{-i-2}$ for any $i \geq 0$. Thus, under the previous choice \eqref{reverse-holder} leads to
\begin{equation}
\label{reverse-holder2}
\left( \frac{1}{|\mathcal{B}^\varepsilon_{\tilde{\rho}_{i+1}}(\bar{g})|}\int_{\mathcal{B}^\varepsilon_{\tilde{\rho}_{i+1}}(\bar{g})}   \Psi_\sigma^{p+\beta_{i+1}}  \right)^{\frac{1}{p+\beta_{i+1}}} \leq \left( 4^i C (p+\beta_i)^{12} \right)^{\frac{1}{p+\beta_i}}  \left( \frac{1}{|\mathcal{B}^\varepsilon_{\tilde{\rho}_i}(\bar{g})|}\int_{ \mathcal{B}_{ \tilde{\rho}_i}^\varepsilon(\bar{g})} \Psi_\sigma^{p+\beta_i}  \right)^{\frac{1}{p+\beta_i}}\, 
\, ,
\end{equation}
whence the pointwise bound $|\nabla_\varepsilon w^\sigma|_\varepsilon  \leq \Psi_\sigma$, the inclusion $\mathcal{B}^\varepsilon_{\bar{r}}(\bar{g}) \subset \mathcal{B}^\varepsilon_{\tilde{\rho}_{i+1}}(\bar{g}) \subset \mathcal{B}^\varepsilon_{2\bar{r}}(\bar{g}) $  with uniform equivalence of volumes, and a simple iteration of \eqref{reverse-holder2} give 
 \begin{equation}
\label{capogna-citti-bound-fake} 
\begin{aligned}
\| |\nabla_\varepsilon w^\sigma|_\varepsilon  \|_{L^\infty(\mathcal{B}^\varepsilon_{\bar{r}}(\bar{g}))} &= \lim_{i \to \infty}   \left( \frac{1}{\mathcal{B}^\varepsilon_{\bar{r}}(\bar{g})}\int_{\mathcal{B}^\varepsilon_{\bar{r}}(\bar{g})}   |\nabla_\varepsilon w^\sigma|_\varepsilon ^{p+\beta_{i+1}}  \right)^{\frac{1}{p+\beta_{i+1}}} \\
&\leq      \hat{C}  \left( \frac{1}{|\mathcal{B}_{ \frac32\bar{r}}^\varepsilon(\bar{g})|}  \int_{ \mathcal{B}_{\frac32\bar{r}}^\varepsilon(\bar{g})} \left(\sigma+|\nabla_\varepsilon w^\sigma |^2_\varepsilon\right)^{\frac{p+2}2} \right)^{\frac{1}{p+2}}\, , 
\end{aligned}
\end{equation}
where $\hat{C}=\Pi_{i=0}^\infty \left( 4^i C (p+\beta_i)^{12} \right)^{\frac{1}{p+\beta_i}} <\infty $ depends only on $p$ and $Q$ and in particular it is independent of $\sigma$ (notice that by its very definition $\beta_i \geq \left(\frac{Q}{Q-2} \right)^i$ for any $i \geq 0$, whence $\sum_{i=0}^\infty \frac{\log (p+\beta_i)}{p+\beta_i}<\infty$ and in turn the convergence of the infinite product follows easily). 

In order to conclude the proof it remains to show how to infer \eqref{capogna-citti-bound-ter} from \eqref{capogna-citti-bound-fake}, i.e., how to decrease the exponent from $p+2$ to $p$ on the r.h.s. of \eqref{capogna-citti-bound-fake} (while inflating slightly the corresponding ball), and to do this we adapt the interpolation argument in \cite[Lemma 3.38]{HKM}.
First, using \eqref{reverse-holder} with $\beta=2$ for any $\bar{r}\leq \rho_1<\rho_2 \leq 2 \bar{r}\leq 2 \rho_1  \leq 4\bar{r}$, so that $\lambda=\rho_1/\rho_2 \in [\frac{\bar{r}}{\rho_2},1) \subseteq [\frac12,1)$, we deduce that the following inequality holds
\begin{equation}
\label{reverse-holder3}
\left(  \frac{1}{|\mathcal{B}^\varepsilon_{\lambda\rho_2}(\bar{g})|}\int_{\mathcal{B}^\varepsilon_{\lambda \rho_2}(\bar{g})}  \left( \Psi_\sigma^{p+2} \right)^{\frac{Q}{Q-2}} \right)^{\frac{Q-2}{Q(p+2)}} \leq \frac{C }{(1-\lambda)^{\frac2{p+2}}}   \left( \frac{1}{|\mathcal{B}^\varepsilon_{\rho_2}(\bar{g})|}\int_{ \mathcal{B}_{ \rho_2}^\varepsilon(\bar{g})} \Psi_\sigma^{p+2} \, \right)^{\frac1{p+2}} 
\, ,
\end{equation}
where the constant $C>0$ depends only on $p$ and $Q$. 

Next we optimize \eqref{reverse-holder3} with respect to  $\lambda\in [\frac{\bar{r}}{\rho_2},1)$ while keeping  $\rho_2 \in [\bar{r},2\bar{r}]$ fixed. Since $p<p+2< (p+2)\frac{Q}{Q-2}$, if we choose $\mu\in (0,1)$ such that 
\begin{equation}
\label{interpexp} \frac1{\mu}-1= \frac{Q}p \, , \quad \hbox{i.e.} \, , \qquad \frac{1}{p+2}= \frac{\mu}{p}+\frac{1-\mu}{(p+2) \frac{Q}{Q-2}} \, , 
\end{equation}
then in view of H\"{o}lder inequality,  \eqref{reverse-holder3} and  $1<p<Q$ we have for fixed $\rho_2 \in [\bar{r},2\bar{r}]$ 
\begin{equation}
\label{def:Phip} 
\Phi(\rho_2;p):= \sup_{ \lambda\in [\frac{\bar{r}}{\rho_2},1)} (1-\lambda)^{\frac{2(1-\mu)}{\mu(p+2)}}   \left( \frac{1}{|\mathcal{B}^\varepsilon_{\lambda \rho_2}(\bar{g})|}\int_{ \mathcal{B}_{ \lambda \rho_2}^\varepsilon(\bar{g})} \Psi_\sigma^{p+2}  \right)^{\frac1{p+2}} \leq C  \left( \frac{1}{|\mathcal{B}^\varepsilon_{\rho_2}(\bar{g})|}\int_{ \mathcal{B}_{ \rho_2}^\varepsilon(\bar{g})} \Psi_\sigma^{p+2} \, \right)^{\frac1{p+2}}  \, ,
 \end{equation}
where the constant $C>0$ depends only on $p$ and $Q$.

For given  $\lambda \in [\frac{\bar{r}}{\rho_2},1) \subseteq [\frac12,1)$
we set $\lambda^\prime=\frac12(1+\lambda) \in (\lambda,1)\subset [\frac12,1)$, so that $\lambda'$ is admissible in \eqref{def:Phip} and $\bar{r}\leq \rho_1= \lambda \rho_2<\lambda^\prime \rho_2 <2 \bar{r}\leq 2 \rho_1 < 4\bar{r}$, so that \eqref{reverse-holder3} also holds with radii $(\lambda \rho_2,\lambda^\prime \rho_2)$ instead of $(\lambda \rho_2, \rho_2)$. Note that $1-\lambda/\lambda^\prime=(1-\lambda^\prime)/\lambda^\prime=\frac{1-\lambda}{2\lambda^\prime}$, so that $1-\lambda/\lambda^\prime$ and $1-\lambda$ are uniformly comparable to $1-\lambda^\prime$ for $\lambda \in [\frac12,1)$, hence combining the previous two inequalities with \eqref{interpexp} we have
the estimate
\begin{equation}
\label{HKM1}
(1-\lambda)^{\frac{2}{\mu(p+2)}}\left(  \frac{1}{|\mathcal{B}^{\varepsilon}_{\lambda\rho_2}(\bar{g})|}\int_{\mathcal{B}^\varepsilon_{\lambda \rho_2}(\bar{g})}  \left( \Psi_\sigma^{p+2} \right)^{\frac{Q}{Q-2}} \right)^{\frac{Q-2}{Q(p+2)}} \leq C \frac{(1-\lambda)^{\frac{2}{\mu(p+2)}}}{(1-\lambda/\lambda^\prime)^{\frac{2}{p+2}}}\left( \frac{1}{|\mathcal{B}^\varepsilon_{\lambda^\prime \rho_2}(\bar{g})|}\int_{ \mathcal{B}_{ \lambda^\prime \rho_2}^\varepsilon(\bar{g})} \Psi_\sigma^{p+2}  \right)^{\frac1{p+2}} \leq \widetilde{C} \Phi(\rho_2;p)\, ,
\end{equation}
where $\widetilde{C}>0$ depends only on $p$ and $Q$ and  $\lambda \in [\frac{\bar{r}}{\rho_2},1)$ is arbitrary. 

Now for any $\delta>0$ we chose  $\tilde{\lambda} \in [\frac{\bar{r}}{\rho_2},1)$  in \eqref{def:Phip} so that in view also of \eqref{interpexp} and H\"{o}lder inequality we have
\[ \Phi(\rho_2;p) <(1-\tilde{\lambda})^{\frac{2(1-\mu)}{\mu(p+2)}}   \left( \frac{1}{|\mathcal{B}^\varepsilon_{\tilde{\lambda} \rho_2}(\bar{g})|}\int_{ \mathcal{B}_{ \tilde{\lambda} \rho_2}^\varepsilon(\bar{g})} \Psi_\sigma^{\mu(p+2)}  \Psi_\sigma^{(1-\mu)(p+2)} \right)^{\frac1{p+2}}+\delta \, , \] 
\[ \leq (1-\tilde{\lambda})^{\frac{2(1-\mu)}{\mu(p+2)}}  \left( \frac{1}{|\mathcal{B}^\varepsilon_{\tilde{\lambda} \rho_2}(\bar{g})|}\int_{ \mathcal{B}_{ \tilde{\lambda} \rho_2}^\varepsilon(\bar{g})} \Psi_\sigma^p  \right)^{\frac{\mu}p}  \left(  \frac{1}{|\mathcal{B}^{\varepsilon}_{\tilde{\lambda}\rho_2}(\bar{g})|}\int_{\mathcal{B}^\varepsilon_{\tilde{\lambda} \rho_2}(\bar{g})}  \left( \Psi_\sigma^{p+2} \right)^{\frac{Q}{Q-2}} \right)^{\frac{Q-2}{Q(p+2)}(1-\mu)} +\delta \, . \]

Next, for $\nu>0$ to be chosen later and $a,b \geq 0$ we use Young's inequality $ab \leq c(\nu,\mu) a^{1/\mu}+\nu b^{1/(1-\mu)}$ to infer from the last inequality that
\[ \Phi(\rho_2;p) < c(\nu, \mu)  \left( \frac{1}{|\mathcal{B}^\varepsilon_{\tilde{\lambda} \rho_2}(\bar{g})|}\int_{ \mathcal{B}_{ \tilde{\lambda} \rho_2}^\varepsilon(\bar{g})} \Psi_\sigma^p  \right)^{\frac1p}+\nu  (1-\tilde{\lambda})^{\frac{2}{\mu(p+2)}}\left(  \frac{1}{|\mathcal{B}^\varepsilon_{\tilde{\lambda}\rho_2}(\bar{g})|}\int_{\mathcal{B}^\varepsilon_{\tilde{\lambda} \rho_2}(\bar{g})}  \left( \Psi_\sigma^{p+2} \right)^{\frac{Q}{Q-2}} \right)^{\frac{Q-2}{Q(p+2)}} +\delta \, ,\]
whence equivalence of volumes between $\mathcal{B}^\varepsilon_{\tilde{\lambda} \rho_2}(\bar{g})$ and $\mathcal{B}^\varepsilon_{ \rho_2}(\bar{g})$ for the first term together with \eqref{HKM1} for the second one yield
\begin{equation}
\label{HKM2}
\Phi(\rho_2;p) < c(\nu, \mu,p,Q)  \left( \frac{1}{|\mathcal{B}^\varepsilon_{ \rho_2}(\bar{g})|}\int_{ \mathcal{B}_{  \rho_2}^\varepsilon(\bar{g})} \Psi_\sigma^p  \right)^{\frac1p}+\nu  \, \widetilde{C} \,\Phi(\rho_2;p) +\delta \, .
\end{equation}
 Finally, choosing $\nu=\frac{1}{2 \widetilde{C}}$ and letting $\delta \to 0$ in \eqref{HKM2} we get $\Phi(\rho_2;p) \leq 2 c(\nu, \mu,p,Q)  \left( \frac{1}{|\mathcal{B}^\varepsilon_{ \rho_2}(\bar{g})|}\int_{ \mathcal{B}_{  \rho_2}^\varepsilon(\bar{g})} \Psi_\sigma^p  \right)^{\frac1p}$, which together with \eqref{def:Phip} for $\rho_1=\frac32 \bar{r}$, $\rho_2=2\bar{r}$ and $\lambda= \frac{\rho_1}{\rho_2} =\frac34>\frac12$ give
 \[  \left( \frac{1}{|\mathcal{B}^\varepsilon_{\frac32 \bar{r}}(\bar{g})|}\int_{ \mathcal{B}_{ \frac32 \bar{r}}^\varepsilon(\bar{g})} \Psi_\sigma^{p+2}  \right)^{\frac1{p+2}} \leq C \Phi(p) \leq C  \left( \frac{1}{|\mathcal{B}^\varepsilon_{ 2 \bar{r}}(\bar{g})|}\int_{ \mathcal{B}_{  2 \bar{r}}^\varepsilon(\bar{g})} \Psi_\sigma^p  \right)^{\frac1p} \, , \]
 for constants $C>0$ depending only on $p$ and $Q$. Combining the last inequality with \eqref{capogna-citti-bound-fake} the estimate \eqref{capogna-citti-bound-ter} follows and the proof is complete.
 
\end{proof}

The second tool needed in the next sections is the following uniform Harnack inequality for $p$-harmonic functions.

\begin{proposition}
\label{garofalo}
Let $1<p<Q$, $\varepsilon>0$ and $\bar{g} \in \He^n$.  For any $ \bar{r}\geq 2 \varepsilon$ let $v=v^\varepsilon_p \in C^1(B^\varepsilon_{ 9\bar{r}}(\bar{g}))$ be a nonnegative $p$-harmonic function w.r. to the metric $|\cdot |^2_\varepsilon$. There exist a constant $C>1$ and $C'>0$ depending only on $Q$ such that
\begin{equation}
\label{harnack}
\underset{B^\varepsilon_r(\bar{g})}{\rm{ess} \sup} \,\,v \leq \left(C p^* \right)^\frac{C'}{p-1} 	\, \underset{B^\varepsilon_r(\bar{g})}{\rm{ess} \inf} \,\,v \qquad \hbox{for any } 0<r \leq \frac12 \bar{r} \, .
\end{equation}
\end{proposition}

\begin{proof}
It is clearly enough to prove the claim assuming $ r=\frac12 \bar{r}$ to get immediately \eqref{harnack} at any smaller scale with the same constant. Moreover, by left invariance of the orthonormal frame $\{X_1, \ldots,X_n,Y_1, \ldots, Y_n, T_\varepsilon\}$ and the related equivariance under left translation of the corresponding $p$-Laplacian, it follows that inequality \eqref{harnack} is translation invariant and it is enough to prove it for $\bar{g}=0$, denoting in the sequel $B^\varepsilon_r=B^\varepsilon_r(0)$ for brevity.  Finally, once $\bar{g}=0$ we recall that for $\widetilde{\varepsilon}:=\frac{2\varepsilon}{\bar{r}} \in(0,1]$ by group dilation we have $\delta_{\bar{r}/2}(B_{18}^{\widetilde{\varepsilon}})=B^\varepsilon_{9 \bar{r}}$ therefore $\tilde{v}=v \circ \delta_{\bar{r}/2}$ is $p$-harmonic w.r. to $| \cdot|_{\widetilde{\varepsilon}}$ in $ B_{18}^{\widetilde{\varepsilon}} $. As a consequence in what follows without loss of generality we will assume throughout the proof that $\varepsilon \in (0,1]$ and $\bar{r}=2=2r$ so that in proving \eqref{harnack} in this (rescaled) ball we will obtain constant independent of $\bar{r}$.

As already announced, we follow closely the strategy in \cite{CapDanGar}, although with some simplifications in treating the two involved $L^\infty-L^p$ bounds through Moser iteration in a unified way and without relying on the Dahlberg and Kenig trick (compare \cite[Lemma 3.29]{CapDanGar}). Moreover, relying on Lemma \ref{uniform-density},
we can apply the abstract John-Nirenberg inequality for BMO functions on homogeneous spaces from \cite{Burger} and infer \eqref{harnack} from the  $L^\infty-L^p$ bounds.


 We first notice that $v>0$ in $B^\varepsilon_{ 9\bar{r}}$ by the strong maximum principle, unless $v \equiv 0$ and \eqref{harnack} holds trivially.
Now we derive a Caccioppoli-type inequality for arbitrary powers of $v$. Thus, for fixed $\varepsilon>0$, $\beta \in \mathbb{R}\setminus \{ -1/p\}$ and for nonnegative $d_\varepsilon$-Lipschitz function $\zeta$ compactly supported in $B^\varepsilon_{ 9\bar{r}}$ to be specified later, we test the equation for $v$ with $\zeta^p v^{1+\beta p}$, hence integration by parts, Cauchy-Schwarz and Young inequalities yield
\[ \int_{B^\varepsilon_{9 \bar{r}}}  \zeta^p | \nabla_\varepsilon v |_\varepsilon^p v^{\beta p} = -\frac{1}{1+\beta p}\int_{B^\varepsilon_{ 9\bar{r}}} | \nabla_\varepsilon v |_\varepsilon^{p-2} v^{1+\beta p} \langle \nabla_\varepsilon v, \nabla_\varepsilon  \zeta^p \rangle_{\varepsilon}  \] \[ \leq  \int_{B^\varepsilon_{ 9\bar{r}}} \left( \zeta^{p-1}| \nabla_\varepsilon v |_\varepsilon^{p-1} v^{\beta (p-1)} \right)  \left(  \frac{p}{|1+\beta p|}v^{1+\beta }  |\nabla_\varepsilon \zeta|_\varepsilon \right) \] 
\[ \leq \left( 1-\frac1p\right) \int_{B^\varepsilon_{ 9\bar{r}}}  \zeta^p | \nabla_\varepsilon v |_\varepsilon^p v^{\beta p}+ \frac1p\int_{B^\varepsilon_{ 9\bar{r}}}   \left(  \frac{p}{|1+\beta p|} \right)^p v^{(1+\beta)p }  |\nabla_\varepsilon \zeta|^p_\varepsilon \, .\]
Thus, for $\beta=-1$ the previous inequality in terms of $w=(1-p)\log v$ becomes
\begin{equation}
\label{caccioppoli-garofalo-1}
\int_{B^\varepsilon_{ 9\bar{r}}}  \zeta^p | \nabla_\varepsilon w |_\varepsilon^p   \leq  Q^p \int_{B^\varepsilon_{ 9\bar{r}}}    |\nabla_\varepsilon \zeta|^p_\varepsilon \, ,
\end{equation}
while for $\beta \neq -1$ in terms of $v^{1+\beta}$ we have
\begin{equation}
\label{caccioppoli-garofalo-1.5}
	\int_{B^\varepsilon_{ 9\bar{r}}}  \zeta^p \left| \nabla_\varepsilon v^{1+\beta} \right|_\varepsilon^p  \leq  \left(  \frac{|1+\beta|p}{|1+\beta p|} \right)^p\int_{B^\varepsilon_{9 \bar{r}}}   \left( v^{(1+\beta) }  |\nabla_\varepsilon \zeta|_\varepsilon \right)^p \, ,
\end{equation}
so that for $\beta=-1+\kappa (1-1/p)$ and $\kappa \not \in \{0, 1\}$ we obtain
\[
\int_{B^\varepsilon_{ 9\bar{r}}}  \zeta^p \left| \nabla_\varepsilon v^{\kappa (1-1/p)} \right|_\varepsilon^p  \leq  \left(  \frac{|\kappa |}{|\kappa  -1|} \right)^p\int_{B^\varepsilon_{9 \bar{r}}}   \left( v^{\kappa (1-1/p)) }  |\nabla_\varepsilon \zeta|_\varepsilon \right)^p \, ,
\]
which combined with Leibniz rule and triangle inequality, together with convexity of $s \to s^p$ and subadditivity of $s \to s^{1/p}$ for $s \geq 0$, gives 
\begin{equation}
\label{caccioppoli-garofalo-2}
\left(\int_{B^\varepsilon_{9 \bar{r}}}   \left| \nabla_\varepsilon \left( \zeta v^{\kappa (1-1/p)}\right) \right|_\varepsilon^p \right)^{1/p} \leq   4 \left[  \frac{1}{|\kappa -1|} +1 \right] \left( \int_{B^\varepsilon_{ 9\bar{r}}}   \left( v^{\kappa (1-1/p) }  |\nabla_\varepsilon \zeta|_\varepsilon \right)^p \right)^{1/p}\, .
\end{equation} 

Next, we specialize inequalities \eqref{caccioppoli-garofalo-1} and \eqref{caccioppoli-garofalo-2} by choosing test functions adapted to the metric $d_\varepsilon$. Instead of resorting to the general theory of spaces of homogeneous type as in, e.g., \cite[Section 7]{CapoCitti3}, here we argue as in Lemma \ref{lemma:eps-gauge}. Notice that $|\nabla_\varepsilon d_\varepsilon(\cdot,\hat{g})|_\varepsilon \leq 1$ on $\He^n$ for any $\hat{g} \in \He^n$ and for each $\varepsilon>0$ because such function is $1$-Lipschitz w.r.to $d_\varepsilon$ just by triangle inequality. Thus,  if $0<\rho_1<\rho_2\leq 2 \rho_1$ and $h: [0,\infty) \to [0,1]$ is a Lipschitz and piecewise linear function, with $\{ h \equiv 1\}=[0,\rho_1]$, $\{ h \equiv 0\}=[\frac{\rho_1+\rho_2}2, \infty)$, and $h$ is linear otherwise, then for any $\hat{g} \in \He^n$ the function $\zeta (g)=h (d_\varepsilon(g,\hat{g}))$ is Lipschitz on $\He^n$, with $\zeta \equiv 1$ on $B^\varepsilon_{\rho_1}(\hat{g})$, $spt \,\zeta \subset B^\varepsilon_{\rho_2}(\hat{g})$ and $| \nabla_\varepsilon  \zeta |_\varepsilon \leq \frac{2}{\rho_2-\rho_1} \, $.

Thus, for $\hat{g} \in B^\varepsilon_{\bar{r}}$ and $0<\rho \leq  2\bar{r}$ we have $B^\varepsilon_{3\rho}(\hat{g})\subset B^\varepsilon_{7\bar{r}}$ and $B^\varepsilon_{4\rho}(\hat{g})\subset B^\varepsilon_{9\bar{r}}$, hence for $(\rho_1,\rho_2)=(3\rho,4\rho)$ in view of the properties of $\zeta$ inequality \eqref{caccioppoli-garofalo-1} together with the doubling property \eqref{doublingproperty}, the uniform Poincar\'{e} inequality \eqref{unifpoincare} and Holder inequality yield for $w=(1-p)\log v$
\begin{equation}
\label{bmo-property-fake}
\frac{1}{|B^\varepsilon_\rho(\hat{g})|}\int_{B^\varepsilon_\rho(\hat{g})} |w-  w_{B^\varepsilon_\rho(\hat{g})}| \leq C  \frac{\rho}{|B^\varepsilon_{3\rho}(\hat{g})|} \int_{B^\varepsilon_{3\rho}(\hat{g})} | \nabla_\varepsilon w |_\varepsilon \leq C \rho \left( \frac{1}{|B^\varepsilon_{4\rho}(\hat{g})|} \int_{B^\varepsilon_{4\rho}(\hat{g})} \zeta^p | \nabla_\varepsilon w |^p_\varepsilon\right)^{1/p} \leq C\, ,  
\end{equation}
where $C>0$ depends only on $Q$. Then, setting $\widetilde{B}^\varepsilon_\rho(\hat{g})= B^\varepsilon_\rho(\hat{g})\cap B^\varepsilon_{\bar{r}} $ and applying inequality \eqref{unif-density-bound} we easily infer from \eqref{bmo-property-fake} and the triangle inequality that for $\hat{g} \in B^\varepsilon_{\bar{r}}$ and $0<\rho \leq  2\bar{r}$ we have
\begin{equation}
\label{bmo-property}
\begin{aligned}
\frac{1}{|\widetilde{B}^\varepsilon_\rho(\hat{g}) |}\int_{\widetilde{B}^\varepsilon_\rho(\hat{g}) } |w-  w_{\widetilde{B}^\varepsilon_\rho(\hat{g})}| 	&\leq C \frac{1}{|B^\varepsilon_\rho(\hat{g})|^2} \int_{B^\varepsilon_\rho(\hat{g})} \int_{B^\varepsilon_\rho(\hat{g})} |w(g)-  w (g^\prime)|\\
&\leq  C\frac{1}{|B^\varepsilon_\rho(\hat{g})|}\int_{B^\varepsilon_\rho(\hat{g})} |w-  w_{B^\varepsilon_\rho(\hat{g})}|  \leq C,
\end{aligned}
\end{equation}
where $C>0$ depends only on $Q$. In particular, the last estimate shows that $w \in BMO(B^\varepsilon_{\bar{r}}, d_\varepsilon, |\cdot|)$ as the ball $\widetilde{B}^\varepsilon_{\rho}(\hat{g}) \subset B^\varepsilon_{\bar{r}}$ is an arbitrary ball in the metric space $(B^\varepsilon_{\bar{r}}, d_\varepsilon)$. 

Note that for all $\varepsilon>0$ the spaces $(B^\varepsilon_{\bar{r}}, d_\varepsilon, |\cdot|)$ are of homogeneous type, the doubling property for them holds uniformly w.r.to $\varepsilon>0$ due to \eqref{doublingproperty} together with \eqref{unif-density-bound}.  Then, it follows from the bound of the BMO norms \eqref{bmo-property} (which does not depend neither on $1<p<Q$ nor on $\varepsilon >0$) and the abstract John-Nirenberg inequality (see, e.g, \cite[Theorem 1]{Burger}) that there exists $\hat{\kappa} \in (0,1/Q)$ depending only on $Q$ (as the $BMO$ norm does) and for each $\bar{\kappa} \in (0,\hat{\kappa})$ there exists a corresponding $C=C(\bar{\kappa})>0$ both dependent only on  $Q$, since $\bar{r}=2$ here, but independent of $\varepsilon$ and $p$,  such that $\int_{B^\varepsilon_{\rho}} exp \left(\bar{\kappa} |w-   w_{B^\varepsilon_\rho}|  \right)\leq C |B^\varepsilon_\rho|$ for any $0<\rho\leq \bar{r}$. Then, the identity $w=(1-p)\log v$ and $e^{\pm \bar{\kappa} w}=v^{\mp \bar{\kappa}(p-1)) }$ together with exponential integrabiliy readily yield for any $0<\rho\leq \bar{r}$ the estimate
\begin{equation}
\label{muckenoupt}
\left( \frac{1}{|B^\varepsilon_\rho|} \int_{B^\varepsilon_\rho}v^{\bar{\kappa}(p-1)}\right)^{\frac1{\bar{\kappa}}} \leq C  \left( \frac{1}{|B^\varepsilon_\rho|} \int_{B^\varepsilon_\rho}v^{-\bar{\kappa}(p-1)}\right)^{-\frac1{\bar{\kappa}}} \, ,
\end{equation}
where $C>0$ depends  only on $Q$ and $\bar{\kappa} \in (0, \hat{k}) \subset (0,1/Q)$ (recall that actually $\bar{r}=2$ here).

In order to proceed in the proof we go back to \eqref{caccioppoli-garofalo-2} and use the Sobolev inequality \eqref{unifsobolev}. For $\hat{g}=0$ we take the cut-off function $\zeta$ as above with $\varepsilon/2 \leq  \bar{r}/4 \leq \rho_1 <\rho_2 \leq \bar{r}$ and we observe that $  B_{ \rho_1}^\varepsilon  \subset spt \, \zeta \subset B_{ \rho_2}^\varepsilon$ with uniformly comparable volumes because of  \eqref{doublingproperty}, moreover Lemma \ref{eps-gauge} yields $| B_{\rho_1}^\varepsilon |\leq C \rho_2^Q$ for a constant $C>0$ depending only on $Q$.
Thus, from Jensen inequality together with \eqref{unifsobolev} and \eqref{doublingproperty} we infer for any $\kappa \in \mathbb{R}\setminus\{ 0, 1\}$ and $\varepsilon/2 \leq \bar{r}/4 \leq \rho_1 <\rho_2 \leq \bar{r}$ that

  \begin{equation}
  \label{revholder-garofalo}
  \begin{aligned}	
  \left( \frac{1}{|B^\varepsilon_{\rho_1}|}   \int_{B^\varepsilon_{ \rho_1}}    \left( v^{(p-1)}\right)^{\frac{\kappa Q}{Q-1}} \right)^{\frac{Q-1}{|\kappa| Q}}   &\leq \left( \frac{1}{|B^\varepsilon_{\rho_1}|}   \int_{B^\varepsilon_{ \rho_1}}    \left( v^{\kappa(p-1)}\right)^{\frac{Q}{Q-p}} \right)^{\frac{Q-p}{|\kappa| Q}} \\
 & \leq \left(
 C p^* \rho_2\right)^{p/|\kappa|} \left[   \frac{1}{|\kappa -1|}+1\right]^{p/|\kappa|}  \left( \frac{1}{|B^\varepsilon_{\rho_2}|}\int_{B^\varepsilon_{\rho_2}}   \left( v^{(p-1) }\right)^\kappa   |\nabla_\varepsilon \zeta|_\varepsilon^p \right)^{1/|\kappa|}
    \\
  & \leq  \left(
 \frac{C p^* \rho_2}{\rho_2-\rho_1}\right)^{Q/|\kappa|} \left[   \frac{1}{|\kappa -1|}+1\right]^{Q/|\kappa|}  \left( \frac{1}{|B^\varepsilon_{\rho_2}|}\int_{B^\varepsilon_{\rho_2}}   \left( v^{(p-1) }\right)^\kappa   \right)^{1/|\kappa|},
 \end{aligned}
\end{equation}
where the constant $C>1$ depends only on $Q$.

Next we perform Moser iteration with exponents $\{ \kappa_i\}_{i \geq 0} \subset (0,1) \cup (1, \infty)$. For integer $j \geq 0$ we set $\kappa_j^*:=\left( \frac{Q}{Q-1}\right)^{-(2j+1)/2} \to 0$ as $j \to \infty$ and we chose $\bar{j}$ so large that by construction $\bar{\kappa}:=\kappa^*_{\bar{j}} <\hat{\kappa}$. Next we define a strictly increasing sequence $\kappa_i:= \left(\frac{Q}{Q-1}\right)^i \bar{\kappa}$, hence $\bar{\kappa}=\kappa_0<\kappa_{\bar{j}}=\sqrt{\frac{Q-1}{Q}}<1< \sqrt{\frac{Q}{Q-1}}=\kappa_{\bar{j}+1}$ and by monotonicity and trivial manipulations we get $|\kappa_i-1| \geq \frac{1}{2Q}$ for any $i \geq 0$.

Thus, for fixed $\varepsilon \leq r = \bar{r}/2$ we repeatedly use \eqref{revholder-garofalo} with the choice $(\rho_1,\rho_2,\kappa)=(\tilde{\rho}_{i+1}, \tilde{\rho}_i,\kappa_i)$, where $\tilde{\rho}_i=\bar{r}(\frac12+2^{-i-1})$ for $i \geq 0$, so that in view of the lower bound $|\kappa_i-1| \geq \frac{1}{2Q}$ inequality\eqref{revholder-garofalo} rewrites for all $i\geq 0$ as
\begin{equation}
\label{revholder-g1}
\left( \frac{1}{|B^\varepsilon_{\tilde{\rho}_{i+1}}|}\int_{B^\varepsilon_{\tilde{\rho}_{i+1}}}   \left(v^{p-1} \right)^{\kappa_{i+1}} \right)^{\frac{1}{|\kappa_{i+1}|}} \leq (C p^*)^{\frac{Q(i+2)}{|\kappa_i|}} \left( \frac{1}{|B^\varepsilon_{\tilde{\rho}_i}|}\int_{ B_{ \tilde{\rho}_i}^\varepsilon} \left(v^{p-1} \right)^{\kappa_i}  \right)^{\frac{1}{|\kappa_i|}}\, 
\, ,
\end{equation}
for a constant $C>1$ depending only on $Q$.

Iterating \eqref{revholder-g1} as in the proof of the previous lemma we easily obtain the inequality 
 \begin{equation}
\label{halfharnack-sup} 
\left(\underset{B^\varepsilon_{r}}{\rm{ess} \sup} \,\,v  \right)^{p-1}= \| v \|^{p-1}_{L^\infty(B^\varepsilon_{r})} =      \left\| v^{p-1} \right\|_{L^\infty(B^\varepsilon_{r})} \leq (C p^*)^{C^\prime}  \left( \frac{1}{|B_{ \bar{r}}^\varepsilon|}  \int_{ B_{\bar{r}}^\varepsilon} v^{\bar{\kappa} (p-1)}  \right)^{\frac{1}{\bar{\kappa}}}\, , 
\end{equation}
where $C>1$ depends on $Q$ and $C^\prime=\sum_{i=0}^\infty Q \frac{i+2}{|\kappa_i|}<\infty$ depends only on $Q$ (as it is the case for $\bar{k}$). 

On the other hand, for fixed $\varepsilon \leq r= \bar{r}/2$ relying again on \eqref{revholder-garofalo} we perform another Moser iteration to show that 
\begin{equation}
\label{halfharnack-inf}
 \left( \underset{B^\varepsilon_{r}}{\rm{ess} \inf} \,\,v \right)^{-(p-1)}=\left\| v^{-(p-1)} \right\|_{L^\infty (B^\varepsilon_{r})} \leq (C p^*)^{C^\prime}  \left( \frac{1}{|B^\varepsilon_{\bar{r}}|} \int_{B^\varepsilon_{\bar{r}}} v^{-\bar{\kappa}(p-1)}\right)^{1/\bar{\kappa}} \, ,
\end{equation}
with $C$ and $C'$ exactly as above, i.e., depending  only on $Q$.
Indeed, if for $i \geq 0$ we set $\tilde{\kappa}_i:=- \kappa_i \searrow -\infty $ as $i \to \infty$ then $\tilde{\kappa}_0=-\bar{\kappa}$ and $\{\tilde{\kappa}_i\} \subset (-\infty, 0)$, hence we can apply \eqref{revholder-garofalo} with the choice $(\rho_1,\rho_2,\kappa)=(\tilde{\rho}_{i+1}, \tilde{\rho}_i,\tilde{\kappa}_i)$ and where still $\tilde{\rho}_i=\bar{r}(\frac12+2^{-i-1})$ for $i \geq 0$. Thus, we get once more \eqref{revholder-g1} for all $i \geq 0$ for $\tilde{\kappa}_i$ instead of $\kappa_i$. Moreover, as in the proof of \eqref{halfharnack-sup}, iterating \eqref{revholder-g1} for these new set of exponents inequality \eqref{halfharnack-inf} follows  easily as $i \to \infty$ . 

Finally, combining inequalities \eqref{halfharnack-sup} and \eqref{halfharnack-inf} with \eqref{muckenoupt} for $\rho=\bar{r}$, the desired inequality \eqref{harnack} follows with $C>1$ and $C'>0$ with the desired dependence on $Q$ only.

\end{proof}


\section{Two-sided control at infinity for $p$-capacitary potentials. }


In this section, we consider $p$-capacitary potentials associated to  exterior domains, the model case being the complement of a Kor\'{a}nyi ball, and we discuss their two-sided behaviour at infinity. On the one hand, we are going to establish a sharp universal pointwise lower bound through an iterative use of the barriers constructed in Section 3. This lower bound will be relevant to obtain for the limit $u$ of solutions to \eqref{EqForUpeps}-\eqref{EqForUpeps-bis}-\eqref{substitution}-\eqref{EqForVpeps}-\eqref{EqForVpeps-bis} as $p \to 1$ a lower bound for the size of their level sets of the form $diam_\varepsilon (\{ u=s\}) \gtrsim e^{s/(Q-1)}$.  

Concerning the upper bound, we first obtain it in the form of integrability at infinity below the Sobolev exponent $p^*$ in the space $L^{\sigma,\infty}$, where $\sigma=\frac{Q(p-1)}{Q-p}$ is the natural exponent dictated by the explicit singular solutions provided in Lemma \ref{p-Harm} , from which higher integrability follows by interpolation. Then, as a byproduct of our analysis performed in Proposition \ref{garofalo}, in Proposition \ref{annular-harnack} we derive a Harnack inequality for $p$-harmonic functions on annular regions which holds uniformly on $\varepsilon \in (0,1]$ and with constants of controlled growth as $p \to 1$. As a consequence, for $p>1$  we provide a power-type decay of $p$-capacitary potentials of the form 
$ v^\varepsilon_p (g)\lesssim \|g\|^{\frac{p-Q}{p-1}}$ as $\| g\| \to \infty$ and with constant of controlled growth as $p\to 1$. In turn, this pointwise behaviour of the potentials guarantee that for solutions $u^\varepsilon=\lim_{p\to 1} u^\varepsilon_p$ one has an upper bound of the form $diam_\varepsilon (\{ u=s\}) \lesssim e^{s/(Q-1)}$.

The pointwise lower bound is given by the following result.
\begin{proposition}
\label{univ-pt-lbound-infinity}
 Let  $1<p\leq 2$, $\varepsilon \in [0,1]$, and let $\Omega \subset \He^n$ be an open set with $C^2$-smooth boundary and bounded complement satisfying assumption assumption ($\mathbf{HP_\Omega}$) with parameter $R_0$ such that $B_{R_\ast}(g_0) \subset \Omega^c$  and $\He^n \setminus \Omega \subset B_{\bar{R}}(g_0)$ for some $R_\ast<2 <\bar{R}$.  There exist  $C>0$ depending only on $R_\ast$ and $Q$ and there exists $\widehat{C}>1$ depending only on $R_\ast$, $\bar{R}$ and $Q$ such that  for $v=v^\varepsilon_p$ the minimizer of \eqref{p-dirichlet} in the set $\dot{W}^{1,p}_{1, \varepsilon} (\Omega)$ we have 
 \begin{equation}
 \label{pt-lbound-infinity}
 v(g) \geq \widehat{C}^{- \frac{\varepsilon^4}{p-1}}     \left( \frac{\| g_0^{-1}*g\|}{R_\ast} \right)^{-\gamma} \qquad \hbox{on } \quad \overline{\Omega} \, , \qquad \hbox{with} \quad \gamma=\frac{Q-p}{p-1} \, , \quad\widehat{C}= \left( 4\frac{\bar{R}}{R_\ast}\right)^C \, .
 \end{equation}
\end{proposition}
\begin{proof}
First notice that the claim is true for $\varepsilon=0$ by choosing any $\widehat{C}>0$. Indeed in that case both sides in \eqref{pt-lbound-infinity} are continuous and finite energy $p$-harmonic functions in $\Omega$ with respect to the sub-Riemannian metric $|\cdot|_0$, moreover they are clearly well ordered on $\partial \Omega$, so that inequality \eqref{pt-lbound-infinity} holds in the whole $\Omega$ as it follows from the comparison principle for finite energy $p$-harmonic functions. 

Assuming $\varepsilon \in (0,1]$ in the sequel,  up to translations we may also assume $g_0=0$. Notice that the right hand side in \eqref{pt-lbound-infinity} is no longer $p$-harmonic in $\Omega$ but a slight tilt in the exponent still gives a subsolution according to Proposition \ref{SubBarrier}. Then we iterate this property at different scales and with different tilts in the exponent in order to prove that \eqref{pt-lbound-infinity} actually holds for suitably chosen $\widehat{C}>0$.

In what follows for $\varepsilon \in(0,1]$ we set $\mu(\varepsilon)=\frac{\varepsilon^4}{p-1} \left( \frac{8}{R_\ast^2}+ \frac{32Q}{R_\ast^4}\right)$, so that in view of Lemma \ref{propertiesPhialpha} and Proposition \ref{SubBarrier} the function  $\| \cdot \|^{-(\gamma+\mu(\varepsilon))}$ has finite energy and it satisfies \eqref{goal}, i.e., it is weakly $p$-subharmonic w.r.to the metric $|\cdot|_\varepsilon$ outside the ball $B_{R_\ast}(0)$, whence comparison principle yields $ R_\ast^{-\gamma-\mu(\varepsilon)} v \geq \| \cdot \|^{-(\gamma+\mu(\varepsilon))}$ in $\overline{\Omega}$.  Now for $\bar{R}> 2> R_\ast$ such that $\He^n \setminus \Omega \subset B_{\bar{R}}(0)$, if $\widetilde{\Omega}=\He^n \setminus \overline{B_{\bar{R}}(0)}$ then we have $v \in \dot{W}^{1,p}_{\varepsilon} (\widetilde{\Omega}) \cap C(\overline{\widetilde{\Omega}})$. Next, using homogeneous dilations we define by induction
\[ v_0(g)=\bar{R}^\gamma v(\delta_{\bar{R}} (g)) \, , \qquad v_{j+1}(g)=2^\gamma v_j(\delta_2(g)) =\left(2^\gamma \right)^{j+1} v_0 (\delta_{2^{j+1}} (g)) \, ,\quad j\in \mathbb{N} \, . \]
Clearly if we set $\bar{R}_j=\bar{R} 2^j$ and $\varepsilon_j= \varepsilon/\bar{R}_j$ then $\varepsilon_j \in (0,1]$ for all $j$, $\varepsilon_j \to 0$ as $j \to \infty$, moreover $v_j(g)=\bar{R}^\gamma_j v (\delta_{\bar{R}_j}(g))$ for any $j \in \mathbb{N}$ and each $v_j$ is $p$-harmonic in $B_1(0)^c$ w.r. to the metric $|\cdot|_{\varepsilon_j}$.

Note that, as above, if  $w \in \dot{W}^{1,p}_{\tilde{\varepsilon}} (\overline{B_1(0)}^c) \cap C(B_1(0)^c)$ is $p$-harmonic in $\overline{B_1(0)}^c$ w.r. to the metric $|\cdot|_{\tilde{\varepsilon}}$ then $w \geq \left( \min_{\partial B_1(0)} \, w\right)  \| \cdot \|^{-(\gamma+\mu(\tilde{\varepsilon}))}$ on $B_1(0)^c$, hence the same holds on $B_1(0)^c \cap \overline{B_2(0)}$ and in particular we have the inequality $2^\gamma \left( \min_{\partial B_2(0)} \, w \right) \geq  2^{-\mu(\tilde{\varepsilon})} \left( \min_{\partial B_1(0)} \, w\right)$.  
Applying the latter inequalities to $w= v_j$ and $\tilde{\varepsilon}=\varepsilon_j$ as defined above  and setting $m_j= \min_{\partial B_1(0)} \,  v_j$ we obtain for each $j \geq 0$
\begin{equation}
\label{induction-minima}
v_j\geq m_j  \| \cdot \|^{-(\gamma+\mu(\varepsilon_j))}  \quad \hbox{on} \, \, \, B_1(0)^c \cap \overline{B_2(0)} \, , \qquad 
m_{j+1} \geq  2^{-\mu(\varepsilon_j)} m_j \geq m_0 2^{-\sum_{k=0}^j \mu(\varepsilon_k)}\, .
\end{equation} 
Clearly $\mu(\varepsilon_j) \leq \frac{C \varepsilon^4}{p-1} 2^{-j-1}$ for all $j \geq 0$, where $C>0$ depends only on $R_\ast$ and $Q$, whence $\sum_j \mu(\varepsilon_j) \leq \frac{C\varepsilon^4}{p-1}$, $m_j \geq m_0 \left(2^C \right)^{\varepsilon^4/(1-p)}$ and $v_j\geq m_j  \| \cdot \|^{-(\gamma+\mu(\varepsilon_j))}\geq m_0 \left(4^C \right)^{\varepsilon^4/(1-p)}  
\| \cdot \|^{-\gamma}$ on  $B_1(0)^c \cap \overline{B_2(0)}$. 

Thus, by construction 
\[ m_0=\bar{R}^\gamma \, \min_{\partial B_{\bar{R}}(0)} \, v \geq \bar{R}^\gamma  \left( \frac{\bar{R}}{R_\ast}\right)^{-\gamma-\mu(\varepsilon)} \geq  R_\ast^\gamma  \left( \frac{\bar{R}}{R_\ast}\right)^{- \frac{C\varepsilon^4}{p-1}} \]
and in turn for each $j\geq 0$
\[  v_j (g)\geq  R_\ast^\gamma  \left( \frac{\bar{R}}{R_\ast}\right)^{- \frac{C\varepsilon^4}{p-1}}  \left(4^C \right)^{\varepsilon^4/(1-p)}  
\| g \|^{-\gamma}=    \left( 4\frac{\bar{R}}{R_\ast}\right)^{- \frac{C\varepsilon^4}{p-1}}   
 \left( \frac{\| g\|}{R_\ast} \right)^{-\gamma} \quad \hbox{for each} \quad g \in B_1(0)^c \cap \overline{B_2(0)}\, ,\]
therefore  the identity $v_j(g)=\bar{R}^\gamma_j v (\delta_{\bar{R}_j}(g))$ for any $j \in \mathbb{N}$ yield inequality \eqref{pt-lbound-infinity} on $\He^n \setminus B_{\bar{R}}(0)$ by choosing $\widehat{C}= \left( 4\frac{\bar{R}}{R_\ast}\right)^C>1$. Finally, since $R_\ast^{-\gamma-\mu(\varepsilon)} v \geq \| \cdot \|^{-(\gamma+\mu(\varepsilon))}$ in $\overline{\Omega}$ yields $R_\ast^{-\gamma} v \geq  (4 \bar{R}/R_\ast)^{-\mu(\varepsilon)} \| \cdot \|^{-\gamma}$ in $\overline{\Omega} \cap B_{\bar{R}}(0)$, then from the choice of $C$ above and the corresponding definition of $\widehat{C}$ inequality \eqref{pt-lbound-infinity} actually holds in the whole $\overline{\Omega}$. This closes the proof.
\end{proof}



\medskip
In order to obtain a pointwise upper bound for the $p$-capacitary potentials we start with the following auxiliary result announced in Section 2 about the Riemannian perimeters $Per_\varepsilon(\cdot)$ and their relation with the horizontal perimeter $Per_0(\cdot)$. The reader should refer to \cite[Section 5]{MonSeCa} for a similar statement, although for a slightly different definition of  $Per_\varepsilon(\cdot)$ corresponding to a family of vector fields different from $\{ X_i, Y_i, T_\varepsilon\}$.  
\begin{lemma}
\label{perimeter}
Let $\widetilde{\Omega} \subset \He^n$ be a bounded open set with $C^1$-smooth boundary, $\mathcal{H}^{2n}$ be the Euclidean $2n$-dimensional Hausdorff measure and let $\hat{n}$ be the Euclidean outer normal vector field. 

For each $\varepsilon>0$ we have the equality
\begin{equation}
\label{perimeter-rep}
\varepsilon Per_\varepsilon (\widetilde{\Omega})=\int_{\partial \widetilde{\Omega}} \left( \sum_{i=1}^n \left[ \left( X_i \cdot \hat{n} \right)^2 +\left( Y_i \cdot \hat{n} \right)^2 \right] +\varepsilon^2 \left( T \cdot \hat{n} \right)^2 \right)^{1/2} \, d\mathcal{H}^{2n} \, .
\end{equation}
As a consequence $\varepsilon \to \varepsilon Per_\varepsilon (\widetilde{\Omega})$ is non decreasing,  $\varepsilon Per_\varepsilon (\widetilde{\Omega}) \leq  Per_1(\widetilde{\Omega})$ for any $\varepsilon \in (0,1]$, and $\varepsilon Per_\varepsilon (\widetilde{\Omega}) \to Per_0(\widetilde{\Omega})$ as $\varepsilon \to 0$.
\end{lemma}
\begin{proof}
It is clearly enough to establish the validity of \eqref{perimeter-rep}, whence the the other properties follow
easily taking into account the representation formula for $P_0(\widetilde{\Omega})$ already proved in \cite[Section 3]{CapDanGar2} (see also \cite{MonSeCa} and \cite[Section 7]{FraSeSeCa}).

The proof of \eqref{perimeter-rep} is very similar to the one in \cite[Section 3]{CapDanGar2} for $Per_0(\widetilde{\Omega})$. Indeed, recall that   
 for any smooth vector field $Z_\varphi=\sum_{i=1}^n \left[ \varphi_i X_i + \varphi_{n+i} Y_i \right]+ \varphi_{2n+1} T_\varepsilon$ associated to a smooth map $\varphi:\He^n \to \mathbb{R}^{2n+1}$  its Riemannian divergence is given by ${\rm div}_{\eps} Z_\varphi = \sum_{i=1}^{n}\left[(X_i \varphi_i + Y_i \varphi_{n+i}\right] + \eT  \varphi_{2n+1} \, ,$ and moreover that we have the identities 
 $X_i \varphi_i=\partial_{x_i} \varphi_i-\partial_t (\frac{y_i}{2} \varphi_i)$, $Y_i \varphi_i=\partial_{y_i} \varphi_{n+i}+\partial_t (\frac{x_i}{2} \varphi_{n+i})$ and $T_\varepsilon \varphi_{2n+1}=\partial_t (\varepsilon \varphi_{2n+1})$.
 
 Then using the identity $\varepsilon \, dvol_\varepsilon (\cdot)=\mathcal{L}^{2n+1}$ and the Euclidean divergence theorem we obtain
\[  \varepsilon \int_{\widetilde{\Omega} } {\rm div}_{\eps} Z_\varphi \, dvol_\varepsilon= \int_{\partial \widetilde{\Omega}} \left( \sum_{i=1}^n \left[  \varphi_i X_i \cdot \hat{n}  +\varphi_{n+i} Y_i \cdot \hat{n}  \right] +\varphi_{2n+1} \varepsilon T \cdot \hat{n}  \right) \,   d\mathcal{H}^{2n} \, ,\]
whence \eqref{perimeter-rep} follows from the definition of $Per_\varepsilon(\cdot)$ taking the supremum over maps $\varphi \in C_0^\infty(\He^n;\mathbb{R}^{2n+1})$ satisfying $|\varphi| \leq 1$ everywhere in $\He^n$.

\end{proof}

Next, relying on the previous lemma we prove a global weak-$L^\sigma$ integrability of $p$-capacitary potentials for $\sigma=(p-1) \frac{Q}{Q-p}$. Recall that for any $\sigma>0$ the weak-$L^\sigma$ space $L^{\sigma,\infty}(\Omega)$ is the set of measurable functions $w: \Omega \to \mathbb{R}$ (up to a.e. equivalence) having finite (quasi)norm (see, e.g., \cite{Graf}), 
\begin{equation}
\label{weak-Lp-norm}
\| w \|_{L^{\sigma,\infty}(\Omega)}=\sup_{ \rho>0}\rho^\sigma |\{|w|>\rho \}| \, .	
\end{equation}

The following result gives control of $\| v^\varepsilon_p \|_{L^{\sigma,\infty}(\Omega)}$ uniformly in $\varepsilon \in[0,1]$ and quantified w.r.to $p \in (1,2]$.
\begin{lemma}
\label{weak-Lp}
Let  $1<p\leq 2$, $\varepsilon \in (0,1]$ and $\sigma=\frac{Q(p-1)}{Q-p}$. Let $\Omega \subset \He^n$ be an open set with $C^2$-smooth boundary and bounded complement,  $\widetilde{\Omega}=  \He^n \setminus \overline{\Omega}$ and let $v=v^\varepsilon_p$ be the minimizer of \eqref{p-dirichlet} in the set $\dot{W}^{1,p}_{1, \varepsilon} (\Omega)$.

Then there exists $C>0$ depending only on $Q$ such that
\begin{equation}
\label{weak-Lsigma}
	\| v \|_{L^{\sigma,\infty}(\Omega)} \leq C \| \nabla_\varepsilon v\|^\sigma_{L^\infty(\partial \Omega)} Per_1(\widetilde{\Omega})^{Q/(Q-p)} \, .
\end{equation}
As a consequence, if $\Omega$ satisfies also assumption $(\bf{HP_\Omega})$
with parameter $R_0>0$ then there exist $C'>0$ depending only on $Q$, $R_0$ and $\Omega$ such that $\| v \|_{L^{\sigma,\infty}(\Omega)} \leq C' $.
\end{lemma}

\begin{proof}
For fixed $0<\tau < \rho \leq \frac12$ we consider auxiliary piecewise affine functions $g_\tau:[0,1] \to [0,\infty)$ given by $g_\tau (r)= \min \{ r, \rho, \frac{\rho}{\tau} (1-r) \}$, so that in particular $g^\prime_\tau (r)=\chi_{(0,\rho)}(r)-\frac{\rho}{\tau} \chi_{(1-\tau,1)}(r)$ a.e. in $(0,1)$ and $g_\tau(1)=0$.

Since $0\leq v \leq 1$ and it is continuous in $\He^n$, by the chain rule we have $\varphi=g_\tau(v) \in \dot{W}^{1,p}_{\varepsilon,0}(\Omega)$, $\varphi \equiv 0$ in $\widetilde{\Omega}$ and $\nabla_\varepsilon \varphi=(\chi_{\{ v<\rho\}}-\frac{\rho}{\tau} \chi_{\{1-\tau <v\}} )\nabla_\varepsilon v $ a.e. in $\Omega$. Using $\varphi$ as an admissible test function for $v$ we obtain 
\[ \int_{\{ v <\rho\}}|\nabla_\varepsilon v|_\varepsilon^p= \frac{\rho}{\tau}  \int_{\{ 1-\tau <v\}}|\nabla_\varepsilon v|_\varepsilon^p \leq  \|  |\nabla_\varepsilon v|_\varepsilon \|_{L^\infty(\{1-\tau \leq v \})}^{p-1}    \frac{\rho}{\tau}  \int_{\{ 1-\tau <v\}}|\nabla_\varepsilon v|_\varepsilon \, . \]
 
In order to estimate the right hand side as $\tau \to 0$, for fixed $\varepsilon$ we rely on the $C^1$-boundary regularity of $v$ to obtain control of $|\nabla_\varepsilon v|_\varepsilon$. 
Note that $v<1$ in $\Omega$ by the strong maximum principle, moreover $\nabla_\varepsilon v \neq 0$ on $\partial \Omega$ because of the Hopf lemma applied at any (maximum) point of $\partial \Omega$. As a consequence $\cap_{\tau>0} \{ v \geq 1-\tau\}=\widetilde{\Omega}$ and $\| |\nabla_\varepsilon v|_\varepsilon \|_{L^\infty(\{1-\tau \leq v \})} =(1+o(1)) \| |\nabla_\varepsilon v|_\varepsilon \|_{L^\infty(\partial \Omega)}  $ as $\tau \to 0$. 

On the other hand, since $\widetilde{\Omega}$ has smooth boundary $\partial \widetilde{\Omega}=\partial \Omega$, the same holds for $\widetilde{\Omega}_{r,\varepsilon} =\{ v >1-r\}$ for any $0<r<\tau$ and for $\tau$ small enough because  $\nabla_\varepsilon v \neq 0$ on $\partial \Omega$. As a consequence $Per_\varepsilon (\widetilde{\Omega}_{r,\varepsilon}) = (1+o(1)) Per_\varepsilon (\widetilde{\Omega})$ as $r \to 0$.
Since $\mathcal{L}^{n+1}=\varepsilon  \, d vol_\varepsilon (\cdot)$, then the Riemannian coarea formula together with the previous expansions yield 
\begin{equation*}
\begin{aligned}
 \int_{\{ v <\rho\}}|\nabla_\varepsilon v|_\varepsilon^p &\leq  (1+o(1)) \| |\nabla_\varepsilon v|_\varepsilon \|_{L^\infty(\partial \Omega)}^{p-1}    \frac{\rho}{\tau}  \int_{\{ 1-\tau <v\}}|\nabla_\varepsilon v|_\varepsilon  \varepsilon dvol_\varepsilon \\
&=(1+o(1)) \| |\nabla_\varepsilon v|_\varepsilon \|_{L^\infty(\partial \Omega)}^{p-1} \frac{\rho}{\tau}\int_{1-\tau}^{1}  \varepsilon Per_\varepsilon (\widetilde{\Omega}_{r,\varepsilon}) \, dr= (1+o(1))  \rho \| |\nabla_\varepsilon v|_\varepsilon \|_{L^\infty(\partial \Omega)}^{p-1} \varepsilon Per_\varepsilon (\widetilde{\Omega}),
\end{aligned}
\end{equation*}
 hence, from Lemma \ref{perimeter} as $\tau \to 0$, we conclude that
\begin{equation}
\label{Lsigma-weak-1}
\int |\nabla_\varepsilon \min \{v, \rho\}|_\varepsilon^p
=\int_{\{ v <\rho\}}|\nabla_\varepsilon v|_\varepsilon^p\leq
 \rho \| |\nabla_\varepsilon v|_\varepsilon \|_{L^\infty(\partial \Omega)}^{p-1} Per_1 (\widetilde{\Omega}) \, ,
\end{equation}
for any $\varepsilon \in (0,1]$, $1<p\leq 2$ and $\rho \in (0, \frac12]$. 

Next, applying Chebyshev's inequality and the Sobolev inequality \eqref{unifsobolev} for $\rho \in (0,\frac12]$ from \eqref{Lsigma-weak-1} we infer that
\[ \rho^p |\{ v>\rho\}|^{p/p^*} \leq \left( \int( \min \{v, \rho\})^{p^*} \right)^{p/p^*} \leq (C p^*)^p  \rho \| |\nabla_\varepsilon v|_\varepsilon \|_{L^\infty(\partial \Omega)}^{p-1} Per_1 (\widetilde{\Omega}) \, ,
\] 
for $C>0$ depending only on $Q$. 

Finally, for $1<p \leq 2$ we have $p^*\leq 2Q/(Q-2)$ and $\sigma \leq Q/(Q-2)$, therefore raising both sides of the previous inequality to $p^*/p$ we  get
\[ \| v \|_{L^{\sigma,\infty}(\Omega)}=\sup_{ 0<\rho< 1}\rho^\sigma |\{v>\rho \} \cap \Omega| \leq \sup_{ 0<\rho\leq 1/2} 2^\sigma \rho^\sigma |\{v>\rho \} \cap \Omega| \leq 
C  \| |\nabla_\varepsilon v|_\varepsilon \|_{L^\infty(\partial \Omega)}^{\sigma} Per_1 (\widetilde{\Omega})^{Q/(Q-p)} \, , \]
for $C>0$ depending only on $Q$, which completes the proof of \eqref{weak-Lsigma}.

Finally, in order to finish the proof we rely on Proposition \ref{p-harm-bd-bound}, so that inequality \eqref{bd-grad-bound} for some $\overline{C}>0$ depending only on $R_0$ and $Q$ together with \eqref{weak-Lsigma} yield 
\[ \| v \|_{L^{\sigma,\infty}(\Omega)} \leq C \left( \frac{\overline{C}^{p-1}}{(p-1)^{p-1}}Per_1(\widetilde{\Omega})\right)^{\frac{Q}{Q-p}} \, \]
and the conclusion follows as the right hand side is uniformly bounded for $p \in (1,2]$ by $C'>0$ with the claimed dependence of the parameters.
\end{proof}

The following result provides uniform Harnack inequalities on annular regions.

\begin{proposition}
\label{annular-harnack}
Let  $1<p\leq 2$, $p_*=\frac{Q(p-1)}{Q-p/2}$ and for $r>\varepsilon>0$ let us set $\Omega_r= B^\varepsilon_{5r}(\bar{g}) \setminus \overline{B^\varepsilon_r(\bar{g})}$ and $\widetilde{\Omega}_r=B^\varepsilon_{4r}(\bar{g}) \setminus \overline{B^\varepsilon_{2r}(\bar{g})} \subset \Omega_r$ . 
There exists $C_*>1$ depending only on $Q$ such that if $v=v^\varepsilon_p \in W^{1,p}_{\varepsilon} (\Omega_r)$ is nonegative and $p$-harmonic in $\Omega_r$ w.r. to the metric $|\cdot|^2_\varepsilon$ then for any $q \geq p_*$ we have
\begin{equation}
\label{Inv-harnack}
\| v\|_{L^\infty(\widetilde{\Omega}_r)} \leq C_*^{\frac{1}{p-1}}  \left( \frac{1}{|\Omega_r |} \int_{\Omega_r}v^q \right)^{1/q} \, .
\end{equation}

\end{proposition}
\begin{proof}
First notice that $v \in C^1(\Omega_r)$ by the regularity theory for $p$-harmonic functions on Riemannian manifolds and, as in Proposition \ref{garofalo}, $v>0$ by the strong maximum principle unless $v \equiv 0$ so that \eqref{Inv-harnack} is trivially satisfied. 
In addition, in view of H\"{o}lder inequality it is clearly enough to prove \eqref{Inv-harnack} for $q=p_*$, moreover, since both the $p$-harmonic equation and the inequality \eqref{Inv-harnack} are translation invariant, from now on we may assume $\bar{g}=0$ and drop the dependence on $\bar{g}$ in the sequel. 

As in the proof of Proposition \ref{garofalo} we first derive a Caccioppoli-type inequality for suitable powers of $v$. Thus, for $\beta > -1/p$ and for nonnegative $d_\varepsilon$-Lipschitz function $\zeta$ compactly supported in $\Omega_r$ to be specified later, we test the equation for $v$ with $\zeta^p v^{1+\beta p}$, hence integration by parts, H\"{o}lder and Young inequalities combined as in the proofs of \eqref{caccioppoli-garofalo-1.5} and \eqref{caccioppoli-garofalo-2} yield
\begin{equation}
\label{caccioppoli-tau-1}
\int_{\Omega_r}   | \nabla_\varepsilon \left( \zeta v^{1+\beta}\right) |_\varepsilon^p \leq    \left(  \frac{2(1+\beta)p}{1+p\beta} \right)^p \int_{\Omega_r}   \left( v^{(1+\beta) }  |\nabla_\varepsilon \zeta|_\varepsilon \right)^p \, ,
\end{equation} 
where we have used that $\frac{p(1+\beta)}{1+p\beta}>1$ for $p>1$ and $\beta >-1/p$.

Next, in analogy with the previous proof we specialize inequality \eqref{caccioppoli-tau-1} by choosing suitable test functions defined in terms of the metric $d_\varepsilon$ in order to obtain \eqref{Inv-harnack} from \eqref{caccioppoli-tau-1} through Moser iteration. For $\frac{r}2< \frac12 \rho_2 \leq  \rho_1<\rho_2  \leq 2r $  we consider  $h: [0,\infty) \to [0,1]$ a Lipschitz and piecewise linear function, with $\{ h \equiv 1\}=[3r-\rho_1,3r+\rho_1]$, $\{ h \equiv 0\}=[0, 3r-\frac{\rho_1+\rho_2}2] \cup[3r +\frac{\rho_1+\rho_2}2, \infty)$, and $h$ is linear otherwise. Hence, in view of the inequality $|\nabla_\varepsilon d_\varepsilon(\cdot,\bar{g})|_\varepsilon \leq 1$ on $\He^n$ the function $\zeta (g)=h (d_\varepsilon(g,0))$ by construction is $d_\varepsilon$-Lipschitz on $\He^n$ with $\zeta \equiv 1$ on $\overline{B^\varepsilon_{3r+\rho_1}}\setminus B^\varepsilon_{3r-\rho_1}$, $spt \,\zeta \subset \subset B^\varepsilon_{3r+\rho_2} \setminus \overline{B^\varepsilon_{3r-\rho_2}}$, and $| \nabla_\varepsilon  \zeta |_\varepsilon \leq \frac{2}{\rho_2-\rho_1} \, $ a.e. in $\He^n$.

Now note that for $\rho_1$ and $\rho_2$ as above we have the following chain of inclusions, namely $ B^\varepsilon_{\frac72 r} \setminus \overline{B^\varepsilon_{\frac52 r}} \subset B^\varepsilon_{3r+\rho_1} \setminus \overline{B^\varepsilon_{3r-\rho_1}}\subset  B^\varepsilon_{3r+\rho_2} \setminus \overline{B^\varepsilon_{3r-\rho_2}} \subset B^\varepsilon_{3r+\rho_2} \subseteq  B^\varepsilon_{5r}$, hence 
all the volumes are uniformly equivalent for $r>\varepsilon>0$ because in such a range of parameters there exists absolute constant $C_0>1$ depending only on $Q$ such that $C_0^{-1}|B^\varepsilon_{5r}| \leq r^Q \leq C_0 | B^\varepsilon_{\frac72 r} \setminus \overline{B^\varepsilon_{\frac52 r}} |$ as follows from \eqref{doublingproperty} and Lemma \ref{eps-gauge}.

Going back to \eqref{caccioppoli-tau-1} we apply the uniform Sobolev inequality \eqref{unifsobolev} for functions supported on $B^\varepsilon_{3r+\rho_2}$ with uniform constant $C p^* \leq \overline{C}$ for a constant $\overline{C}>0$ independent of $p \in (1,2]$. Thus,  taking the aforementioned equivalence of volumes into account, the previous choice of $\zeta$ and Lemma \ref{lemma:eps-gauge}-{\it (5)} for $\varepsilon<3r+\rho_2 < 5 \rho_2$ yield

\begin{equation}
  \label{revholder-beta}	
  \left( \frac{1}{ |B^\varepsilon_{3r+\rho_1} \setminus \overline{B^\varepsilon_{3r-\rho_1}}|}\int_{  B^\varepsilon_{3r+\rho_1} \setminus \overline{B^\varepsilon_{3r-\rho_1}}}  v^{(1+\beta) p^*}   \right)^{\frac{1}{(1+\beta)p^*}}  
  \leq  \left( \frac{C_0^2}{ |B^\varepsilon_{3r+\rho_2} |}
  \int_{  B^\varepsilon_{3r+\rho_2} }  \left( \zeta v^{(1+\beta)}\right)^{p^*}   
  \right)^{\frac{1}{(1+\beta)p^*}}  
\end{equation}  
\[ \leq {C_0}^{\frac{2}{(1+\beta)p^*}} \, \overline{C}^{1/(1+\beta)}   |B^\varepsilon_{3r+\rho_2} | ^{\frac{1}{Q(1+\beta)}}
\left(   \frac{1}{ |B^\varepsilon_{3r+\rho_2} |} \int_{B^\varepsilon_{3r+\rho_2}}   | \nabla_\varepsilon \left( \zeta v^{1+\beta}\right) |_\varepsilon^p  \right)^{\frac{1}{(1+\beta)p}} \]
  \[\leq   \left[ C \cdot  
 \frac{\rho_2}{\rho_2-\rho_1}  \cdot  \frac{2(1+\beta)p}{1+p\beta}
  \right]^{\frac{1}{1+\beta}}  \left(  \frac{1}{ |B^\varepsilon_{3r+\rho_2} \setminus \overline{B^\varepsilon_{3r-\rho_2}}|}\int_{B^\varepsilon_{\rho_2}(\bar{g})}   v^{(1+\beta)p }  \right)^{\frac{1}{(1+\beta)p}}\, ,
\]
where we have also used that $C_0^{1/p^*}<C_0$ as $C_0>1$ and $p^*=\frac{Qp}{Q-p}>p>1$, and \eqref{caccioppoli-tau-1} to obtain the last inequality. Note that for $p \in (1,2]$ the constant $C>0$ in the square brackets depends only on $Q$.

Next, using \eqref{revholder-beta} we perform Moser iteration with exponents $\{ \beta_i\}_{i \geq 0}$ where $\beta_0>-1/p$ is chosen such that $(1+\beta_0)p=p_*$ and $\beta_{i+1}=p/(Q-p)+\beta_i Q/(Q-p)$, so that $\{\beta_i\} \subset (-1/p, \infty)$ and it is increasing to infinity. Indeed, note that if $s=p (\beta+1)$ then $\beta >-1/p$ iff $s> p-1$ so that for $s_i=p(\beta_i+1)$ we have $s_{i+1}=s_i Q/(Q-p)=p_* \frac{Q^{i+1}}{(Q-p)^{i+1}}$ is increasing and  $s_i > (p-1)  \frac{Q^{i+1}}{(Q-1/2)^{i+1}} \uparrow +\infty$ as $i \to \infty$.   Moreover, since $\beta \to \frac{(1+\beta)p}{1+\beta p} $ is decreasing for $\beta > -1/p$ one has $\frac{(1+\beta_i)p}{1+\beta_i p}  \leq \frac{(1+\beta_0)p}{1+\beta _0 p}\leq 2Q $ for any $i \geq 0$ and $p>1$. 

For fixed $r>\varepsilon >0$ and $\frac{r}2< \frac12 \rho_2 \leq  \rho_1<\rho_2  \leq  2r $ we repeatedly use \eqref{revholder-beta} with the choice $(\rho_1,\rho_2,\beta)=(\tilde{\rho}_{i+1}, \tilde{\rho}_i,\beta_i)$, where $\tilde{\rho}_i=r(1+2^{-i})$ for $i \geq 0$, so that $\frac{\rho_2}{\rho_2-\rho_1} = \frac{\tilde{\rho}_i}{\tilde{\rho}_i - \tilde{\rho}_{i+1}} \leq 2^{i+2}$ for any $i \geq 0$.

Thus, inequality\eqref{revholder-beta} rewrites for all $i\geq 0$ as
\begin{equation}
\label{revholder-beta-2}
\begin{aligned}
&\left( \frac{1}{ |B^\varepsilon_{3r+\tilde{\rho}_{i+1}} \setminus \overline{B^\varepsilon_{3r-\tilde{\rho}_{i+1}}}|}\int_{  B^\varepsilon_{3r+\tilde{\rho}_{i+1}} \setminus \overline{B^\varepsilon_{3r-\tilde{\rho}_{i+1}}}}  v^{s_{i+1}}   \right)^{\frac{1}{s_{i+1}}} \\ 
 & \quad \quad\leq \left(   
 C  2^{i+2} \right)^{\frac{1}{1+\beta_i}}   \left( \frac{1}{ |B^\varepsilon_{3r+\tilde{\rho}_i} \setminus \overline{B^\varepsilon_{3r-\tilde{\rho}_i}}|}\int_{  B^\varepsilon_{3r+\tilde{\rho}_i} \setminus \overline{B^\varepsilon_{3r-\tilde{\rho}_i}}}  v^{s_i}   \right)^{\frac{1}{s_i}},
 \end{aligned}
\end{equation}
for a constant $C>1$ depending only on $Q$.

Since $s_0=p_*$ and $\tilde{\rho}_0=2r$, iterating \eqref{revholder-beta-2} as in the proof of the previous lemma  we easily obtain \eqref{Inv-harnack},
 with a constant factor $C_*>1$ depending only on $Q$ ensuring the bound 
 \[ \Pi_{i=0}^\infty  \left(   
 C  2^{i+2} \right)^{\frac{p}{s_i}} \leq \Pi_{i=0}^\infty  \left(   
 C  2^{i+2} \right)^{\frac{2 (Q-1/2)^{i+1}}{(p-1) Q^{i+1}}}= \left( \Pi_{i=0}^\infty  \left(   
 C^2  4^{i+2} \right)^{\frac{ (Q-1/2)^{i+1}}{ Q^{i+1}}} \right)^{\frac{1}{p-1}}=:C_*^{\frac{1}{p-1}} 
 <\infty \, , \] 
 where the convergence of the infinite product is a simple consequence of  $\sum_{i=0}^\infty  \frac{i (Q-1/2)^i}{ Q^i} <\infty$. This closes the proof.
\end{proof}

The pointwise upper bound for the $p$-capacitary potentials is given by the following result.
\begin{proposition}
\label{univ-pt-ubound-infinity}
 Let  $1<p\leq 2$, $\varepsilon \in [0,1]$, and let $\Omega \subset \He^n$ be an open set with $C^2$-smooth boundary and bounded complement satisfying assumption assumption ($\mathbf{HP_\Omega}$) with parameter $R_0$ such that $B_{R_\ast}(g_0) \subset \Omega^c$  and $\He^n \setminus \Omega \subset B_{\bar{R}}(g_0)$ for some $R_\ast<2 <\bar{R}$.  There exist  $C_0>1$ depending only on $Q$, $R_0$, $R_\ast$, $\bar{R}$ and $\partial \Omega$ such that  for $v=v^\varepsilon_p$ the minimizer of \eqref{p-dirichlet} in the set $\dot{W}^{1,p}_{1, \varepsilon} (\Omega)$ we have 
 \begin{equation}
 \label{pt-ubound-infinity}
 v(g) \leq {C_0}^{\frac1{p-1}}     \left( \frac{\| g_0^{-1}*g\|}{\bar{R}} \right)^{-\gamma} \qquad \hbox{on } \quad \overline{\Omega} \, , \qquad \hbox{with} \quad \gamma=\frac{Q-p}{p-1} \,  .
 \end{equation}
\end{proposition}

\begin{proof} First we notice that inequality \eqref{pt-ubound-infinity} trivially holds in $\overline{\Omega} \cap B_{\bar{R}}(g_0)$ for any $C_0>1$ because $0<v \leq 1$ in $\overline{\Omega}$. Then the claim is still true for $\varepsilon=0$ in the whole $\overline{\Omega}$ by chosing any $C_0>1$. Indeed in that case both sides in \eqref{pt-lbound-infinity} are continuous and finite energy $p$-harmonic functions in $\He^n \setminus B_{\bar{R}}(g_0)$ with respect to the sub-Riemannian metric $|\cdot|_0$, moreover they are clearly well ordered on $B_{\bar{R}}(g_0)$, so that inequality \eqref{pt-lbound-infinity} holds in the whole $\He^n \setminus B_{\bar{R}}(g_0)$ and in turn on $\Omega$, as it follows from the comparison principle for finite energy $p$-harmonic functions. 

When dealing with \eqref{pt-ubound-infinity} for $\varepsilon \in (0,1]$ we rely instead on Proposition \ref{annular-harnack} and Lemma \ref{weak-Lp}.
More precisely, for $1<p \leq 2$ we derive an upper bound for $v=v^\varepsilon_p$ from the intrinsic Harnack inequality \eqref{Inv-harnack} with exponent $q=p_*=\frac{Q(p-1)}{Q-p/2}$ applied to the corresponding potential $v$ on dyadic annuli and we combine it with the global weak-$L^\sigma$ integrability, with $\sigma=\frac{Q(p-1)}{Q-p}>q$, given in \eqref{weak-Lsigma} and further localized on the annuli. 

 Indeed, note that by assumption $\Omega^c \subset B_{\bar{R}}(g_0)$ and $B_{\bar{R}}(g_0) \subset B^\varepsilon_{C\bar{R}}(g_0)$ for $C>1$ depending only on $Q$ given by Lemma \ref{lemma:eps-gauge}, so that $\He^n \setminus B^\varepsilon_{C\bar{R}}(g_0)\subset \Omega$.  For $g \not \in B^\varepsilon_{C2\bar{R}}(g_0) $ we select $ i \in \mathbb{N}$ such that $ C 2^{i+1} \bar{R} \leq d_\varepsilon(g_0,g)< C 2^{i+2} \bar{R}$ and we apply Proposition \ref{annular-harnack} with $r=r_i=C 2^i \bar{R}$, because $v^\varepsilon_p$ is $p$-harmonic in $\Omega_r=\Omega_{r_i}:= \{ C 2^i \bar{R} < d_\varepsilon(g_0,\cdot) <5 C 2^i \bar{R}\} \subset \Omega$ and $g \in \overline{\widetilde{\Omega}_r}=\{ C2^{i+1} \bar{R} \leq d_\varepsilon(g_0,g) \leq C2^{i+2} \bar{R} \}$ in view of the choice of $i \in \mathbb{N}$. 
Thus inequality \eqref{Inv-harnack} yields
\begin{equation}
\label{annuli-ub-1}
 v (g)\leq C_*^{\frac{1}{p-1}}  \left( \frac{1}{|\Omega_{r_i} |} \int_{\Omega_{r_i}}v^{p_*} \right)^{1/p_*} 
\end{equation}
for a constant $C_*>1$ depending only on $Q$. 

On the other hand, notice that still Lemma \ref{lemma:eps-gauge} gives $ 2C \varepsilon < C 2^i \bar{R} \leq  d_\varepsilon(g_0,g)  \leq C \| g_0^{-1}*g \|_\varepsilon$, whence $\| g_0^{-1}*g\| \leq 2 \| g_0^{-1}*g\|_\varepsilon \leq 2C d_\varepsilon(g_0,g) \leq 8 C^2 2^i \bar{R}=R_\ast \tilde{r}_i$ and in turn  $v(g) \geq \widehat{C}^{1/(1-p)} \tilde{r}_i^{-\gamma}$ on $\Omega_{r_i}$, where $\tilde{r}_i= 8 C^2 2^i \bar{R}/R_\ast =8Cr_i/R_\ast$, because of Proposition \ref{univ-pt-lbound-infinity}. 

Thus, from Lemma \ref{weak-Lp} and $\sigma/p_*=1+\frac{p/2}{Q-p}>1$ we infer that 
\[ \int_{\Omega_{r_i}}v^{p_*} = \int_0^\infty p_* \tau^{p_*-1} |\{v> \tau\} \cap \Omega_{r_i}| d\tau =|\Omega_{r_i}|  \left(\widehat{C}^{1/(1-p)} \tilde{r}_i^{-\gamma} \right)^{p_*} +\int_{\widehat{C}^{1/(1-p)} \tilde{r}_i^{-\gamma}}^\infty p_* \tau^{p_*-1} |\{v> \tau\} \cap \Omega_{r_i}| d\tau \]
\[ \leq  |\Omega_{r_i}|  \left(\widehat{C}^{1/(1-p)} \tilde{r}_i^{-\gamma} \right)^{p_*}  +\| v\|_{L^{\sigma,\infty}(\Omega)} \int_{\widehat{C}^{1/(1-p)} \tilde{r}_i^{-\gamma}}^\infty p_* \tau^{p_*-\sigma-1} d\tau\]
\[=  |\Omega_{r_i}|  \left(\widehat{C}^{1/(1-p)} \tilde{r}_i^{-\gamma} \right)^{p_*} +\frac{1}{\sigma/p_*-1} \left(  \widehat{C}^{1/(1-p)} \tilde{r}_i^{-\gamma} \right)^{p_*-\sigma} \| v\|_{L^{\sigma,\infty}(\Omega)}\]
\[ = |\Omega_{r_i}|  \left(\widehat{C}^{1/(1-p)} \tilde{r}_i^{-\gamma} \right)^{p_*} \left(1 +\frac{Q-p}{p/2} \widehat{C}^{Q/(Q-p)} \frac{\tilde{r}_i^{Q}}{|\Omega_{r_i}|} \| v\|_{L^{\sigma,\infty}(\Omega)} \right)\, .\]

Now, observe that $|\Omega_{r_i}| \geq (c_* r_i)^Q$ for a constant $c_*\in (0,1)$ depending only on $Q$ because of \eqref{doublingproperty} and Lemma \ref{lemma:eps-gauge}-{\it (2)}, hence the choice of $r_i=C2^i \bar{R}=\tilde{r}_i R_\ast/(8C) $ combined with the previous inequalities for $1<p\leq 2$ give 
\[ \frac{1}{|\Omega_{r_i} |} \int_{\Omega_{r_i}}v^{p_*} \leq  \left(\widehat{C}^{1/(1-p)} \tilde{r}_i^{-\gamma} \right)^{p_*}
\left(1 +2Q \widehat{C}^{Q/(Q-2)} \left(\frac{8C}{c_*R_\ast}\right)^Q \| v\|_{L^{\sigma,\infty}(\Omega)}  \right) \, ,\]
so that from the last inequality, \eqref{annuli-ub-1} and $\tilde{r}_i \geq \frac{2C}{R_0}d_\varepsilon(g_0,g) \geq \frac{1}{2C\bar{R}} d_\varepsilon(g_0,g)$ we infer that
\begin{equation}
\label{annuli-ub-2}
v (g)\leq \left( \frac{C_*}{\widehat{C}}\right)^\frac{1}{p-1}  \left(1 +2Q \widehat{C}^{Q/(Q-2)} \left(\frac{8C}{c_*R_\ast}\right)^Q \| v\|_{L^{\sigma,\infty}(\Omega)}  \right)^{\frac{1}{p-1}}  \left( \frac{ d_\varepsilon(g_0,g)}{2C \bar{R}}\right)^{-\gamma} \, ,
\end{equation}
whenever $g \not \in B^\varepsilon_{C2\bar{R}}(g_0) $ and in turn for any $g \in \overline{\Omega}$, because for $g \in B^\varepsilon_{C2\bar{R}}(g_0)$ we have $v(g) \leq 1$ but each factor on the right hand side is greater than one.

 As already recalled above, by Lemma \ref{lemma:eps-gauge}-{\it (3)} when $g \not \in B^\varepsilon_{C2\bar{R}}(g_0)$ we have $\| g_0^{-1}*g\|_\varepsilon \geq 2 \bar{R} \geq 2 \varepsilon \bar{R}>2\varepsilon>0$, so that $d_\varepsilon(g,g_0)\geq C^{-1} \| g_0^{-1}*g\|_\varepsilon \geq (2C)^{-1} \| g_0^{-1}*g\|$.  On the other hand, observing that the same lemma yields $B_{R_\ast}(g_0) \subset B_{\bar{R}}(g_0)\subset \mathcal{B}^\varepsilon_{\bar{R}}(g_0) \subset B_{C2\bar{R}}^\varepsilon(g_0) \subset  B^1_{C2\bar{R}}(g_0) $, if we set $M(\bar{R})=\max \{ \|g_0^{-1}*g \| \, , \,\, g \in \overline{B^1_{C2\bar{R}}(g_0)}\}$ and $m(\bar{R})= \min \{ d_1(g_0,g) \, , \, \, g \in \overline{B^1_{C2\bar{R}}(g_0)} \cap \overline{\Omega}\}>0$, then for any $\varepsilon \in (0,1]$ we get
 \begin{equation} 
\label{glob-dist-comp} 
 \|g_0^{-1}*g\|\leq \left( 2C+ \frac{M(\bar{R})}{m(\bar{R})} \right) d_\varepsilon(g_0,g) \qquad \hbox{for any} \quad g \in \overline{\Omega} \, . 
 \end{equation}
 Combining \eqref{annuli-ub-2} with \eqref{glob-dist-comp} and observing that $ 2Q \widehat{C}^{Q/(Q-2)} \left(\frac{8C}{c_*R_\ast}\right)^Q >1$
  we eventually obtain 
\begin{equation}
\label{annuli-ub-3}
v (g)\leq \overline{C}^{\frac{1}{p-1}} \left(1 + \| v\|_{L^{\sigma,\infty}(\Omega)}  \right)^{\frac{1}{p-1}}  \left( \frac{\| g_0^{-1}*g\|}{\bar{R}}\right)^{-\gamma} \, ,
\end{equation}
for a constant $\overline{C}>0$ depending only on $Q$, $R_\ast$ and $\bar{R}$. 
Finally, the uniform bound for  $\| v\|_{L^{\sigma,\infty}(\Omega)}$ contained in Lemma \ref{weak-Lp} combined with \eqref{annuli-ub-3} yield the desired conclusion. 

\end{proof}

\section{Pointwise Bochner's inequality for the energy density.}

The purpose of this section is to derive a Bochner inequality for the energy density $|\egrad v |_\varepsilon^2+(Tv)^2$, where $v=v^\varepsilon_p$ is a weak solution to \eqref{EqForVpeps} and the parameters $\varepsilon \in (0,1]$ and $p \in (1,2]$ are fixed. As already recalled, $v\in C^{1,\beta}(\Omega)$ and such a pointwise inequality will be derived in the region $\widehat{\Omega}:=\{ g \in \Omega \, \, : \, \, \egrad v \neq 0\}$ where actually $v \in C^\infty(\widehat{\Omega})$ by elliptic regularity.

For functions $\psi \in C^2(\widehat{\Omega})$ let us consider the following linear (locally uniformly) elliptic operator
\begin{equation}
\label{L_v}
L_v (\psi) := \ediv (A_v \egrad \psi) \, ,
\end{equation}

\noindent where the matrix $(p-1) I_{2n+1} \leq A_v \leq I_{2n+1}$ for $p \in (1,2]$ is given by
\begin{equation}
\label{def-A_v} A_v := |\egrad v|_\varepsilon^{p-2} \left( I_{2n+1} + (p-2) \dfrac{\egrad v}{|\egrad v|_\varepsilon} \otimes \dfrac{\egrad v}{|\egrad v|_\varepsilon} \right).
\end{equation}
Clearly $L_v$ is formally the linearized operator of $\Delta_p$ at $v$, moreover $L_v(v)=0$ in $\widehat{\Omega}$ because of \eqref{def-A_v} and \eqref{EqForVpeps}. Moreover, since $ v$ is smooth in $\widehat{\Omega}$ we can differentiate \eqref{EqForVpeps} and since $[X_i,T]=[Y_i,T]=0$ for all $i=1, \ldots, n$ we also obtain $L_v(T v)=0$ in $\widehat{\Omega}$.

As a first result for this section we need the following auxiliary identity for solutions to \eqref{EqForVpeps}. The proof here is given just for the sake of completeness, as it is a mere adaptation of the one in \cite[Lemma 2.1]{KotschwarNi} where exactly the same identity is established but for solutions to \eqref{EqForUpeps} instead of \eqref{EqForVpeps}.  

\begin{lemma}\label{lem:2.1_Kotschwar_Ni}
Let $1<p \leq 2$, $\varepsilon >0$, $\widehat{\Omega} \subset \He^n$ an open set and $v\in C^3(\widehat{\Omega})$ such that $\nabla_\varepsilon v \neq 0$ in $\widehat{\Omega}$. If $v$ is $p$-harmonic in $\widehat{\Omega}$ and $L_v$ is the linear operator defined as in \eqref{L_v} then 
\begin{equation}
\label{KN-identity}
L_v(|\egrad v|_\varepsilon^2) = |\egrad v|_\varepsilon^{p-2} \left( 2|D^2_\varepsilon v |_\varepsilon^2 +2 Ric_\varepsilon(\nabla_\varepsilon v) \right) + \dfrac{p-2}{2} |\egrad v|_\varepsilon^{p-4} |\egrad |\egrad v|_\varepsilon^2|_\varepsilon^{2} \, ,
\end{equation}
where $D^2_\varepsilon v$ is the covariant Hessian and $Ric_\varepsilon$ is the Ricci (2,0)-tensor corresponding to the Riemannian metric $|\cdot|^2_\varepsilon$.
\end{lemma}

\begin{proof}
First notice that $p$-harmonicity of $v$ together with the assumption $\nabla_\varepsilon v \neq 0$ in $\widehat{\Omega}$ give $v \in C^\infty (\widehat{\Omega})$ by elliptic regularity.  As $\varepsilon>0$ is fixed, for simplicity we drop the dependence on $\varepsilon$ in the sequel and, following the notation in  \cite{KotschwarNi}, we let $f:=|\nabla v|^{2}$, $v_{ij}$ the Hessian of $v $ and $Ric_{ij}$ the Ricci curvature w.r. to any fixed local orthonormal frame. Then the equation \eqref{EqForVpeps} satisfied by $v$ yields the identity
\begin{equation}\label{eq:Identity_from_p_harmonicity}
f^{p/2 -1}\Delta v + \left(\dfrac{p}{2}-1\right)f^{p/2-2}\langle \nabla f, \nabla v\rangle =0 \, .
\end{equation}
Taking the gradient on both sides of the latter, and then computing its scalar product with $\nabla v$ we obtain
\begin{equation}\label{Eq:grad_e_scal_prod_eq_per_f}
\begin{aligned}
&\left(\dfrac{p}{2}-1\right)f^{p/2-2}\Delta v \langle \nabla f, \nabla v\rangle + f^{p/2-1}\langle \nabla \Delta v, \nabla v\rangle \\
&+\left(\dfrac{p}{2}-1\right)\left(\dfrac{p}{2}-2\right) f^{p/2-3}\langle \nabla f, \nabla v\rangle^2+\left(\dfrac{p}{2}-1\right) f^{p/2-2}(f_{ij}v_i v_j + v_{ij}f_i v_j) =0.
\end{aligned}
\end{equation}
Now, following almost verbatim the computations made in \cite[Lemma 2.1]{KotschwarNi}, we get 
\begin{equation}\label{Eq_from_Kotschwar_Ni}
\begin{aligned}
L_{v}(f) &= \left(\dfrac{p}{2}-1\right) f^{p/2-2}|\nabla f|^2 + f^{p/2-1}\Delta f + (p-2)\Delta v \langle \nabla v, \nabla f\rangle f^{p/2-2}\\
&+(p-2)\left(\dfrac{p}{2}-1\right) f^{p/2-3}\langle \nabla v,\nabla f\rangle^2\\
&+(p-2) (v_{ij}f_i v_j f^{p/2-2}+ f_{ij}v_i v_j f^{p/2-2}- \langle \nabla v, \nabla f\rangle^2 f^{p/2-3}) \, .
\end{aligned}
\end{equation}
On the other hand, recalling that the Bochner's identity gives
\begin{equation*}
\Delta f =  2|D^2 v |^2 + 2\langle\nabla\Delta v,\nabla v\rangle +2 Ric (\nabla v)  \, ,
\end{equation*}
\noindent  inserting this identity in \eqref{Eq_from_Kotschwar_Ni} we get
\begin{equation*}
\begin{aligned}
L_{v}(f) &= f^{p/2-1}\left(2|D^2 v |^2  +2 Ric (\nabla v) \right) 
+\dfrac{p-2}{2} f^{p/2-2}|\nabla f|^2 \\
&+  (p-2)f^{p/2-2}\Delta v \langle \nabla v, \nabla f\rangle + 2f^{p/2-1}\langle\nabla\Delta v,\nabla v\rangle\\
&+(p-2)\left(\dfrac{p}{2}-2\right) f^{p/2-3}\langle \nabla v,\nabla f\rangle^2+(p-2) f^{p/2-2}(v_{ij}f_i v_j + f_{ij}v_i v_j) \, .
\end{aligned}
\end{equation*}
\noindent Finally, taking \eqref{Eq:grad_e_scal_prod_eq_per_f} into account identity \eqref{KN-identity} follows from the previous formula as the last two lines vanish identically.
\end{proof}
Next, in order to rewrite the identity \eqref{KN-identity} in an efficient way in the following result we provide an explicit formula for the Ricci tensor $Ric_{\varepsilon}$ for any $\varepsilon>0$.

\begin{lemma}\label{lem:Ricci}
Let $U, V\in \mathfrak{h}$ be two left invariant vector fields.
Then the Ricci tensor is given by
\begin{equation}\label{eq:Ricci}
Ric_{\varepsilon}(U,V) = -\dfrac{1}{2\varepsilon^2} \langle U,V\rangle_{\varepsilon} + \dfrac{n+1}{2\varepsilon^2}\langle U,T_{\varepsilon}\rangle_{\varepsilon}\, \langle V,T_{\varepsilon}\rangle_{\varepsilon}.
\end{equation}
In particular, $\displaystyle{Ric_\varepsilon(U):=Ric_\varepsilon (U,U)= -\frac{1}{2\varepsilon^2}}|U|_\varepsilon^2+\frac{n+1}{2\varepsilon^2} |<U,T_\varepsilon>_\varepsilon|^2$ for any $U \in \mathfrak{h}$.
\begin{proof}
We write $U, V \in \mathfrak{h}$ as linear combinations with constant coefficients,  
\begin{equation*}
U= \sum_{i=1}^{n}\left(u_i X_i + u_{i+n}Y_i\right) + u_{2n+1}T_{\varepsilon} \quad \textrm{and} \quad V= \sum_{i=1}^{n}\left(v_i X_i + v_{i+n}Y_i\right) + v_{2n+1}T_{\varepsilon}.
\end{equation*}
Taking into account the commutation rules recalled in Section 2, it follows that
\begin{equation*}
[U,X_i] = -\dfrac{u_{i+n}}{\varepsilon}T_{\varepsilon}, \quad  [U,Y_i] = \dfrac{u_i}{\varepsilon}T_{\varepsilon} \quad \textrm{and}\quad [U,T_{\varepsilon}]= 0.
\end{equation*}
Moreover, the action of the Levi-Civita connection $\nabla^\varepsilon$ on the orthonormal base recalled in Section 2 yields
\begin{equation*}
\nabla^{\varepsilon}_{U}V = \sum_{i=1}^{n}\left(\dfrac{u_{i+n}v_{2n+1}+u_{2n+1}v_{i+n}}{2\varepsilon}\right)X_i - \sum_{i=1}^{n}\left(\dfrac{u_{i}v_{2n+1}+u_{2n+1}v_{i}}{2\varepsilon}\right)Y_i + \sum_{i=1}^{n}\left(\dfrac{v_{i+n}u_{i}-u_{i+n}v_{i}}{2\varepsilon}\right)T_{\varepsilon}.
\end{equation*}

Now recall that for $U,V \in \mathfrak{h}$ the corresponding curvature operator $R^\varepsilon(U,V):\mathfrak{h} \to \mathfrak{h}$ is given by 
\begin{equation}\label{eq:curvature_operator} 
R^\varepsilon(U,V)W= \nabla^\varepsilon_U \nabla^\varepsilon_V W-\nabla^\varepsilon_V \nabla^\varepsilon_U W -\nabla^\varepsilon_{[U,V]} W \, , \quad W \in \mathfrak{h},
\end{equation}
so that we have to compute
\begin{equation*}
Ric_\varepsilon(U,V)=-\sum_{i=1}^{n}\left(\langle R^{\varepsilon}(U,X_i)V,X_i \rangle_{\varepsilon} + \langle R^{\varepsilon}(U,Y_i)V,Y_i \rangle_{\varepsilon}\right) - \langle R^{\varepsilon}(U,T_{\varepsilon})V,T_{\varepsilon}\rangle_{\varepsilon}.
\end{equation*}
To this aim, recalling \eqref{eq:curvature_operator} and the explicit computations made for the Levi-Civita connection $\nabla^{\varepsilon}$, it follows that
\begin{align*}
\nabla^{\varepsilon}_{X_i}V &= -\dfrac{v_{2n+1}}{2\varepsilon}Y_i + \dfrac{v_{i+n}}{2\varepsilon} T_{\varepsilon}, \quad \textrm{for every } i =1,\ldots,n,\\
\nabla^{\varepsilon}_{U}\nabla^{\varepsilon}_{X_i}V &= -\dfrac{u_{2n+1}v_{2n+1}}{4\varepsilon^2}X_i + \dfrac{v_{i+n}}{4\varepsilon^2}\sum_{j=1}^{n}u_{j+n}X_j - \dfrac{v_{i+n}}{4\varepsilon^2}\sum_{j=1}^{n}u_{j}Y_j - \dfrac{v_{2n+1}u_i}{4\varepsilon^2}T_{\varepsilon}, \quad \textrm{for every } i =1,\ldots,n,\\
\nabla^{\varepsilon}_{X_i}\nabla^{\varepsilon}_{U}V &=-\sum_{j=1}^{n}\left(\dfrac{u_j v_{j+n} - v_j u_{j+n}}{4\varepsilon^2}\right)Y_i - \left(\dfrac{u_i v_{2n+1}+v_i u_{2n+1}}{4\varepsilon^2}\right) T_{\varepsilon}, \quad \textrm{for every } i =1,\ldots,n,\\
\nabla^{\varepsilon}_{[U,X_i]}V &= -\dfrac{u_{i+n}}{2\varepsilon^2}\sum_{j=1}^{n}v_{j+n}X_j + \dfrac{u_{i+n}}{2\varepsilon^2}\sum_{j=1}^{n}v_{j}Y_j, \quad \textrm{for every } i =1,\ldots,n,
\end{align*}
\noindent and hence
\begin{equation}\label{eq:Ric_1}
\sum_{i=1}^{n} \langle R^{\varepsilon}(U,X_i)V,X_i \rangle_{\varepsilon} = \dfrac{3}{4\varepsilon^2}\sum_{i=1}^{n}u_{i+n}v_{i+n} - \dfrac{n}{4\varepsilon^2}u_{2n+1}v_{2n+1}.
\end{equation}
Similarly,
\begin{align*}
\nabla^{\varepsilon}_{Y_i}V &= \dfrac{v_{2n+1}}{2\varepsilon}X_i - \dfrac{v_{i}}{2\varepsilon} T_{\varepsilon}, \quad \textrm{for every } i =1,\ldots,n,\\
\nabla^{\varepsilon}_{U}\nabla^{\varepsilon}_{Y_i}V &= - \dfrac{v_{i}}{4\varepsilon^2}\sum_{j=1}^{n}u_{j+n}X_j 
-\dfrac{u_{2n+1}v_{2n+1}}{4\varepsilon^2}Y_i +\dfrac{v_i}{4\varepsilon^2} \sum_{j=1}^{n}u_j Y_j - \dfrac{v_{2n+1}u_{i+n}}{4\varepsilon^2}T_{\varepsilon}, \quad \textrm{for every } i =1,\ldots,n,\\
\nabla^{\varepsilon}_{Y_i}\nabla^{\varepsilon}_{U}V &=\sum_{j=1}^{n}\left(\dfrac{u_j v_{j+n} - v_j u_{j+n}}{4\varepsilon^2}\right)X_i - \left(\dfrac{u_{i+n} v_{2n+1}+v_{i+n} u_{2n+1}}{4\varepsilon^2}\right) T_{\varepsilon}, \quad \textrm{for every } i =1,\ldots,n,\\
\nabla^{\varepsilon}_{[U,X_i]}V &=\dfrac{u_{i}}{2\varepsilon^2}\sum_{j=1}^{n}v_{j+n}X_j -\dfrac{u_{i}}{2\varepsilon^2}\sum_{j=1}^{n}v_{j}Y_j, \quad \textrm{for every } i =1,\ldots,n,
\end{align*}
\noindent and hence
\begin{equation}\label{eq:Ric_2}
\sum_{i=1}^{n} \langle R^{\varepsilon}(U,Y_i)V,Y_i \rangle_{\varepsilon} = \dfrac{3}{4\varepsilon^2}\sum_{i=1}^{n}u_{i}v_{i} - \dfrac{n}{4\varepsilon^2}u_{2n+1}v_{2n+1}.
\end{equation}
Finally,
\begin{align*}
\nabla^{\varepsilon}_{T_{\varepsilon}}V &= \sum_{j=1}^{n}\dfrac{v_{j+n}}{2\varepsilon}X_j - \sum_{j=1}^{n}\dfrac{v_{j}}{2\varepsilon}Y_j,\\
\nabla^{\varepsilon}_{U}\nabla^{\varepsilon}_{T_{\varepsilon}}V &= -\dfrac{u_{2n+1}}{4\varepsilon^2}\sum_{j=1}^{n}v_j X_j - \dfrac{u_{2n+1}}{4\varepsilon^2}\sum_{j=1}^{n}v_{j+n}Y_j - \sum_{j=1}^{n}\left(\dfrac{u_j v_j + u_{j+n}v_{j+n}}{4\varepsilon^2}\right)T_{\varepsilon},\\
\nabla^{\varepsilon}_{T_{\varepsilon}}\nabla^{\varepsilon}_{U}V &= -\sum_{j=1}^{n}\dfrac{u_{j}v_{2n+1}+u_{2n+1}v_{j}}{4\varepsilon^2} X_j-\sum_{j=1}^{n}\dfrac{u_{j+n}v_{2n+1}+u_{2n+1}v_{j+n}}{4\varepsilon^2}Y_j,\\
\nabla^{\varepsilon}_{[U,T_{\varepsilon}]}V=0,
\end{align*}
\noindent and hence
\begin{equation}\label{eq:Ric_3}
\langle R^{\varepsilon}(U,T_{\varepsilon})V,T_{\varepsilon}\rangle_{\varepsilon} = -\dfrac{1}{4\varepsilon^2}\sum_{i=1}^{n}(u_{i}v_{i} +u_{i+n}v_{i+n}).
\end{equation}
Combining \eqref{eq:Ric_1}, \eqref{eq:Ric_2} and \eqref{eq:Ric_3}, we finally get \eqref{eq:Ricci} as desired.
\end{proof}
\end{lemma}

In the next lemma we exploit the key cancellation between the covariant Hessian $D^2_\varepsilon v$ and $Ric_\varepsilon$.

\begin{lemma}\label{Hessian-rough} Let $\varepsilon >0$, $\widehat{\Omega} \subset \He^n$ an open set and $v\in C^2(\widehat{\Omega})$ such that $\nabla_\varepsilon v \neq 0$ in $\widehat{\Omega}$. Then 
	\begin{equation}
		\label{Hessian-lb}
		2|D^2_\varepsilon v |_\varepsilon^2 +2 Ric_\varepsilon(\nabla_\varepsilon v) = 2|\overline{D}^2_{\varepsilon}v|^2 + \dfrac{4}{\varepsilon} \langle \nabla_{0,J}v, \nabla_{\varepsilon}T_{\varepsilon}v \rangle_{\varepsilon} = 2|\overline{D}^2_{\varepsilon}v|^2 + 4 \langle \nabla_{0,J}v, \nabla_{\varepsilon}T v \rangle_{\varepsilon}.
	\end{equation}
	
Here for vector fields
$E_j \in \{X_1, \ldots, X_n, Y_1, \ldots ,Y_n,T_\varepsilon \}$ and gradient $\nabla_\varepsilon v=\sum_{j=1}^{2n+1} (E_jv) E_j$ the previous identity involves the covariant Hessian $D^2_\varepsilon v=\nabla^\varepsilon  \nabla_\varepsilon v$, the {\it rough Hessian matrix} defined as
\begin{equation}
\label{rough-Hessian}
	\overline{D}^2_{\varepsilon}v := (E_j E_i v)_{i,j=1,\ldots,2n+1},
\end{equation}
and the rotated (i.e., symplectic) horizontal gradient defined as 
\begin{equation}\label{eq:Grad_simplettico}
	\nabla_{0,J}v := (Y_1 v, \ldots, Y_n v, -X_1 v, \ldots, -X_n v, 0).
\end{equation}
\end{lemma}
\begin{proof}
	Using the notation previously introduced,
	\begin{equation}\label{eq:splitting_Hessian} 
		\begin{aligned}
			| D^2_\varepsilon v|_\varepsilon^2&=\sum_{i=1}^{2n+1} |\nabla^\varepsilon_{E_i} \nabla_\varepsilon v|_\varepsilon^2= \sum_{i=1}^{2n+1} \left\langle \sum_{j=1}^{2n+1}  \nabla^\varepsilon_{E_i} \left( E_jv E_j  \right),  \sum_{k=1}^{2n+1}  \nabla^\varepsilon_{E_i} \left( E_kv E_k  \right)  \right\rangle_\varepsilon \\
			&=\sum_{i,j=1}^{2n+1} (E_i E_j v)^2+\sum_{i=1}^{2n+1} \sum_{j=1}^{2n+1} \sum_{k=1}^{2n+1} (E_j v)(E_kv)   \left\langle \nabla^\varepsilon_{E_i} E_j, \nabla^\varepsilon_{E_i} E_k \right\rangle_\varepsilon \\
			&+ \sum_{i=1}^{2n+1} \sum_{j=1}^{2n+1} \sum_{k=1}^{2n+1} \left[ (E_i E_j v)(E_kv) \left\langle   E_j, \nabla^\varepsilon_{E_i} E_k \right\rangle_\varepsilon  +(E_j v) (E_iE_kv) \langle \nabla^\varepsilon_{E_i} E_j ,  E_k \rangle_\varepsilon \right]\\
			&=I_\varepsilon+II_\varepsilon+III_\varepsilon \, .
		\end{aligned}
	\end{equation}
	The first term is simply the Frobenius norm of the rough Hessian,
	\begin{equation}
		\label{rough-hessian}
		I_\varepsilon=\sum_{i=1}^{2n+1} \sum_{j=1}^{2n+1}(E_i E_j v)^2 = |\overline{D}_{\varepsilon}^{2}v|^{2}.
	\end{equation}
	Next, combining \eqref{orthogonal-derivatives} and \eqref{constant-square} we get
	\begin{equation}\label{extra-square}
		\begin{aligned}
			II_\varepsilon&= \sum_{i=1}^{2n+1} \sum_{j,k=1}^{2n+1}  (E_j v)(E_kv)   \left\langle \nabla^\varepsilon_{E_i} E_j, \nabla^\varepsilon_{E_i} E_k \right\rangle_\varepsilon = \sum_{j=1}^{2n+1} \sum_{i=1}^{2n+1}  (E_j v)^2   \left| \nabla^\varepsilon_{E_i} E_j \right|_\varepsilon^2\\
			&=\frac{1}{2\varepsilon^2} |\nabla_\varepsilon v|_\varepsilon^2+ \frac{n-1}{2\varepsilon^2}(T_\varepsilon v)^2 \, .
		\end{aligned}
	\end{equation}
	Finally, taking into account \eqref{orthogonal-derivatives} we obtain
	\begin{equation*}
		III_\varepsilon=\sum_{i=1}^{2n+1} \sum_{j=1}^{2n+1} \sum_{k=1}^{2n+1} \left[ -(E_i E_j v)(E_kv)    +(E_j v) (E_iE_kv) \right] \langle \nabla^\varepsilon_{E_i} E_j ,  E_k \rangle_\varepsilon
	\end{equation*}
	\[ =\sum_{i=1}^n  \sum_{j=1}^{2n+1} \sum_{k=1}^{2n+1} \left[ -(X_i E_j v)(E_kv)    +(E_j v) (X_iE_kv) \right] \langle \nabla^\varepsilon_{X_i} E_j ,  E_k \rangle_\varepsilon \]
	
	\[ +\sum_{i=1}^n  \sum_{j=1}^{2n+1} \sum_{k=1}^{2n+1} \left[ -(Y_i E_j v)(E_kv)    +(E_j v) (Y_iE_kv) \right] \langle \nabla^\varepsilon_{Y_i} E_j ,  E_k \rangle_\varepsilon \]
	\[  + \sum_{j=1}^{2n+1} \sum_{k=1}^{2n+1} \left[ -(T_\varepsilon E_j v)(E_kv)    +(E_j v) (T_\varepsilon E_kv) \right] \langle \nabla^\varepsilon_{T_\varepsilon} E_j ,  E_k \rangle_\varepsilon\]
	\[= \sum_{i=1}^n   \left[ -(X_i Y_i v)(T_\varepsilon v)    +(Y_i v) (X_iT_\varepsilon v) \right] \frac{1}{2\varepsilon} + \sum_{i=1}^n   \left[ (X_i T_\varepsilon v)(Y_iv)    -(T_\varepsilon v) (X_iY_iv) \right] \frac{1}{2\varepsilon} \]
	
	\[ +\sum_{i=1}^n  \left[ (Y_i X_i v)( T_\varepsilon v)    -(X_i v) (Y_i T_\varepsilon v) \right] \frac{1}{2\varepsilon} +\sum_{i=1}^n   \left[ -(Y_i T_\varepsilon v)(X_iv)    +(T_\varepsilon v) (Y_iX_iv) \right] \frac{1}{2\varepsilon} \]
	\[  + \sum_{j=1}^{n}  \left[ (T_\varepsilon X_j v)(Y_jv)    -(X_j v) (T_\varepsilon Y_j v) \right] \frac{1}{2\varepsilon}   + \sum_{j=1}^{n}  \left[ -(T_\varepsilon Y_j v)(X_j v)    +(Y_j v) (T_\varepsilon X_j v) \right] \frac{1}{2\varepsilon}. \]
	Recalling that $[X_i,Y_i]=\tfrac{T_\varepsilon}{\varepsilon}$ for every $i=1,\ldots,n$, the latter becomes
	\begin{equation}\label{no-mixed-terms}
		III_{\varepsilon} = \dfrac{2}{\varepsilon} \sum_{i=1}^{n}\left[(Y_iv)(X_i T_{\varepsilon}v) - (X_i v)(T_{\varepsilon}Y_i v)\right] - \dfrac{n}{\varepsilon^2}(T_{\varepsilon}v)^2.
	\end{equation}
	By Lemma \ref{lem:Ricci}, see in particular \eqref{eq:Ricci}, we also have
	\begin{equation}\label{eq:Ricci_nabla_eps}
		Ric_{\varepsilon}(\nabla_{\varepsilon}v) = \dfrac{n+1}{2\varepsilon^2}(T_{\varepsilon}v)^{2} - \dfrac{1}{2\varepsilon^2}|\nabla_{\varepsilon}v|_{\varepsilon}^{2}.
	\end{equation}
	Therefore, combining \eqref{eq:splitting_Hessian}, \eqref{extra-square}, \eqref{no-mixed-terms} and \eqref{eq:Ricci_nabla_eps}, and recalling \eqref{eq:Grad_simplettico}, we find
	\begin{equation*}
		\begin{aligned}
			2|D^2_\varepsilon v |_\varepsilon^2 +2 Ric_\varepsilon(\nabla_\varepsilon v) &= 2|\overline{D}^2_\varepsilon v |^2 + \dfrac{1}{\varepsilon^2}|\nabla_{\varepsilon}v|^{2}_{\varepsilon} + \dfrac{n-1}{\varepsilon^2}(T_{\varepsilon}v)^2 \\
			&+ \dfrac{4}{\varepsilon}\langle \nabla_{0,J}v, \nabla_{\varepsilon}T_{\varepsilon}v\rangle_{\varepsilon} -\dfrac{2n}{\varepsilon^2}(T_{\varepsilon}v)^2 + \dfrac{n+1}{\varepsilon^2}(T_{\varepsilon}v)^{2} - \dfrac{1}{\varepsilon^2}|\nabla_{\varepsilon}v|_{\varepsilon}^{2} \\
			&=2|\overline{D}^2_{\varepsilon}v|^2 + \dfrac{4}{\varepsilon} \langle \nabla_{0,J}v, \nabla_{\varepsilon}T_{\varepsilon}v \rangle_{\varepsilon}\\
			&=2|\overline{D}^2_{\varepsilon}v|^2 + 4 \langle \nabla_{0,J}v, \nabla_{\varepsilon}T v \rangle_{\varepsilon},
		\end{aligned}
	\end{equation*}
	\noindent and this closes the proof.
\end{proof}	

The next result is essentially a simple rewriting of Lemma \ref{lem:2.1_Kotschwar_Ni} combined with the above Lemma \ref{Hessian-rough} using the notations introduced therein and summarizes all the relevant identities we obtained so far. 


\begin{proposition}\label{lem:Bochner_riscritta}
	Let $1<p \leq 2$, $\varepsilon \in (0,1]$, $\widehat{\Omega} \subset \He^n$ an open set and $v\in C^3(\widehat{\Omega})$ such that $\nabla_\varepsilon v \neq 0$ in $\widehat{\Omega}$. If $v$ is $p$-harmonic in $\widehat{\Omega}$ and $L_v$ is the linear operator defined as in \eqref{L_v} then 
	\begin{equation}\label{eq:Lv_di_v2}
		L_{v}(v^2) = 2(p-1) |\nabla_{\varepsilon}v|_{\varepsilon}^{p} 
	\end{equation}
	\noindent together with
	\begin{equation}\label{eq:Lv_di_Tv2}
		\begin{aligned}
			L_{v}((T v)^2) &= 2(p-1)|\nabla_{\varepsilon}v|_{\varepsilon}^{p-2}|\nabla_{\varepsilon}T v|^{2}_{\varepsilon} \\
			&+ 2(2-p)|\nabla_{\varepsilon}v|^{p-2}_{\varepsilon} \left(|\nabla_{\varepsilon}T v|^{2}_{\varepsilon} - \left\langle\dfrac{\nabla_{\varepsilon}v}{|\nabla_{\varepsilon}v|_{\varepsilon}}, \nabla_{\varepsilon}T v\right\rangle_{\varepsilon}^2\right) \, ,
		\end{aligned}
	\end{equation}
	whereas
	\begin{equation}
		\label{eq:KN-identity-Improved}
		L_v(|\egrad v|_\varepsilon^2) = |\egrad v|_\varepsilon^{p-2} \left( 2|\overline{D}^2_{\varepsilon}v|^2 + 4 \langle \nabla_{0,J}v, \nabla_{\varepsilon}T v \rangle_{\varepsilon} \right) + \dfrac{p-2}{2} |\egrad v|_\varepsilon^{p-4} |\egrad |\egrad v|_\varepsilon^2|_\varepsilon^{2}.
	\end{equation}
	\end{proposition}	
	
	\begin{proof}
		By the very definition of $L_v$, and recalling that $v$ is $p$-harmonic, we have that
		\begin{equation}
			L_{v}(v^2) = 2v L_{v}(v) + 2 \langle A_{v}\nabla_{\varepsilon}v, \nabla_{\varepsilon}v\rangle_{\varepsilon} = 2|\nabla_{\varepsilon}v|^{p}_{\varepsilon}+ 2(p-2) |\nabla_{\varepsilon}v|_{\varepsilon}^{p},
		\end{equation}
		\noindent which gives \eqref{eq:Lv_di_v2}.
		Regarding the next identity, let us first recall that
		\begin{equation}
			L_{v}(T v) = T (L_{v}v) =0.
		\end{equation}
		Therefore,
		\begin{equation}
			\begin{aligned}
				L_{v}((T v)^2) &= 2 T (L_{v}T v) + 2\langle A_v \nabla_{\varepsilon}T v, \nabla_{\varepsilon}T v \rangle_{\varepsilon}\\
				&= 2 |\nabla_{\varepsilon}v|_{\varepsilon}^{p-2}\,|\nabla_{\varepsilon}T v|_{\varepsilon}^{2} + 2(p-2)|\nabla_{\varepsilon}v|_{\varepsilon}^{p-2}\, \left\langle\dfrac{\nabla_{\varepsilon}v}{|\nabla_{\varepsilon}v|_{\varepsilon}}, \nabla_{\varepsilon}T v\right\rangle_{\varepsilon}^2\\
				&=2 (2-p+p-1)|\nabla_{\varepsilon}v|_{\varepsilon}^{p-2}\,|\nabla_{\varepsilon}T v|_{\varepsilon}^{2} + 2(p-2)|\nabla_{\varepsilon}v|_{\varepsilon}^{p-2}\, \left\langle\dfrac{\nabla_{\varepsilon}v}{|\nabla_{\varepsilon}v|_{\varepsilon}}, \nabla_{\varepsilon}T v\right\rangle_{\varepsilon}^2\\
				&2(p-1)|\nabla_{\varepsilon}v|_{\varepsilon}^{p-2}\,|\nabla_{\varepsilon}T v|_{\varepsilon}^{2} + 2(2-p)|\nabla_{\varepsilon}v|^{p-2}_{\varepsilon} \left(|\nabla_{\varepsilon}T v|^{2}_{\varepsilon} - \left\langle\dfrac{\nabla_{\varepsilon}v}{|\nabla_{\varepsilon}v|_{\varepsilon}}, \nabla_{\varepsilon}T v\right\rangle_{\varepsilon}^2\right),
			\end{aligned}
		\end{equation}
		\noindent and together with Lemma \ref{Hessian-rough} this closes the proof.		
	\end{proof}
	
	The following result yields the key improved Kato-type inequality for the squared norm of the rough Hessian to be used in combination with identity \eqref{eq:KN-identity-Improved}. 

\begin{lemma}\label{lem:Kato_Improved}
	Let $1<p\leq 2$, $\varepsilon \in (0, 1]$, $\widehat{\Omega} \subset \He^n$ an open set and $v\in C^2(\widehat{\Omega})$ such that $\nabla_\varepsilon v \neq 0$ in $\widehat{\Omega}$. If $v$ is $p$-harmonic in $\widehat{\Omega}$ then
	\begin{equation}\label{eq:Kato_Improved}
		2 |\nabla_{\varepsilon}v|_{\varepsilon}^2 \, \, |\overline{D}^2_{\varepsilon}v|^2 \geq \dfrac{1}{2}\left(1+\dfrac{(p-1)^2}{2n}\right)\left\langle \dfrac{\nabla_{\varepsilon}v}{|\nabla_{\varepsilon}v|_{\varepsilon}}, \nabla_{\varepsilon}|\nabla_{\varepsilon}v|_{\varepsilon}^2\right\rangle_{\varepsilon}^2.
	\end{equation}
	\begin{proof}
		Let $\overline{x}\in \hat{\Omega}$. We consider an orthonormal basis of $\mathbb{R}^{2n+1}$ w.r.to the standard Euclidean metric, denoted by $\{v_k\}_{k=1,\ldots,2n+1}$. In particular, we choose
		\begin{equation*}
			v_1 := \dfrac{\nabla_{\varepsilon}v(\overline{x})}{|\nabla_{\varepsilon}v(\overline{x})|_{\varepsilon}}.
		\end{equation*}	
		We further denote by $B:= \overline{D}^{2}_{\varepsilon}v(\overline{x}) = \left((E_jE_i v)(\overline{x})\right)$.
		The following holds:
		\begin{equation}\label{eq:Hessian1}
			\begin{aligned}
				|\overline{D}^{2}_{\varepsilon}v(\overline{x})|^2 &= \sum_{i,j=1}^{2n+1}|E_j(E_i v)(\overline{x})|_{\varepsilon}^2 = \mathrm{tr}(BB^t) = \sum_{k=1}^{2n+1}|B^{t}v_k|^{2} \geq \sum_{k=1}^{2n+1}\langle v_k, B^{t}v_k\rangle^2 \\
				&= \langle v_1, B^{t}v_1 \rangle^2 + \sum_{k=2}^{2n+1}\langle v_k, B^{t}v_k\rangle^2 \geq \langle v_1, B^{t}v_1 \rangle^2 + \dfrac{1}{2n}\left|\sum_{k=2}^{2n+1}\langle v_k, B^{t}v_k\rangle \right|^2\\
				&= \langle v_1, B^{t}v_1 \rangle^2 + \dfrac{1}{2n}|\langle v_1,B^{t}v_1 \rangle - \mathrm{tr}(B^{t})|^{2}.
			\end{aligned}
		\end{equation}
		On the other hand, recalling \eqref{eq:Identity_from_p_harmonicity}, at $\bar{x}$ we have
		\begin{equation}\label{eq:trace_Bt}
			\mathrm{tr}(B^{t}) = \sum_{k=1}^{2n+1}E_k E_k v = \Delta_{\varepsilon}v = \dfrac{2-p}{2} \left\langle \nabla_{\varepsilon}v, \dfrac{\nabla_{\varepsilon}|\nabla_{\varepsilon}v|_{\varepsilon}^2}{|\nabla_{\varepsilon}v|_{\varepsilon}^2}\right\rangle_{\varepsilon},
		\end{equation}
		\noindent and
		\begin{equation}\label{eq:v1_scal_Btv1}
			\langle v_1, B^{t}v_1 \rangle = \sum_{k=1}^{2n+1}\sum_{i=1}^{2n+1}\dfrac{E_k v}{|\nabla_{\varepsilon}v|_{\varepsilon}} E_k(E_i v)\dfrac{E_i v}{|\nabla_{\varepsilon}v|_{\varepsilon}} = \dfrac{1}{|\nabla_{\varepsilon}v|_{\varepsilon}^2} \left\langle \nabla_{\varepsilon}v
			, \nabla_{\varepsilon}	\left(\dfrac{|\nabla_{\varepsilon}v|_{\varepsilon}^2}{2}\right)\right\rangle_{\varepsilon}.
		\end{equation}
		Combining \eqref{eq:Hessian1} with \eqref{eq:trace_Bt} and \eqref{eq:v1_scal_Btv1} we finally get in $\overline{x}$ for $1<p \leq 2$
		\begin{equation*}
			\begin{aligned}
				|\nabla_{\varepsilon}v|^{2}_\varepsilon\, \, |\overline{D}^{2}_{\varepsilon}v|^2 & \geq \dfrac{1}{4} \left\langle \dfrac{\nabla_{\varepsilon}v}{|\nabla_{\varepsilon}v|_{\varepsilon}}, \nabla_{\varepsilon}|\nabla_{\varepsilon}v|_{\varepsilon}^2 \right\rangle^2_{\varepsilon} + \dfrac{1}{4} \dfrac{1}{2n} (1-(2-p))^2 \left\langle \dfrac{\nabla_{\varepsilon}v}{|\nabla_{\varepsilon}v|_{\varepsilon}}, \nabla_{\varepsilon}|\nabla_{\varepsilon}v|_{\varepsilon}^2 \right\rangle^2_{\varepsilon}\\
				&= \dfrac{1}{4}\left(1+\dfrac{(p-1)^2}{2n}\right)\left\langle \dfrac{\nabla_{\varepsilon}v}{|\nabla_{\varepsilon}v|_{\varepsilon}}, \nabla_{\varepsilon}|\nabla_{\varepsilon}v|_{\varepsilon}^2\right\rangle_{\varepsilon}^2,
			\end{aligned}
		\end{equation*}
		\noindent and this closes the proof.
	\end{proof}
\end{lemma}

As a simple consequence of Lemma \ref{lem:Kato_Improved}, we get the following
\begin{corollary}\label{cor:from_kato_improved}
	Let $1<p\leq 2$, $\varepsilon \in (0,1]$, $\widehat{\Omega} \subset \He^n$ an open set and $v\in C^2(\widehat{\Omega})$ such that $\nabla_\varepsilon v \neq 0$ in $\widehat{\Omega}$. If $v$ is $p$-harmonic in $\widehat{\Omega}$ then
	\begin{equation}
		\begin{aligned}
			\dfrac{p-2}{2}&|\nabla_{\varepsilon}v|_{\varepsilon}^{p-4} \left\langle \dfrac{\nabla_{\varepsilon}v}{|\nabla_{\varepsilon}v|_{\varepsilon}}, \nabla_{\varepsilon}|\nabla_{\varepsilon}v|_{\varepsilon}^2\right\rangle_{\varepsilon}^2 + 2|\nabla_{\varepsilon}v|^{p-2}_{\varepsilon}\, \, |\overline{D}^{2}_{\varepsilon}v|^2\\
			&\geq \left(\dfrac{p-1}{2} + \dfrac{(p-1)^2}{4n}\right)|\nabla_{\varepsilon}v|_{\varepsilon}^{p-4} \left\langle \dfrac{\nabla_{\varepsilon}v}{|\nabla_{\varepsilon}v|_{\varepsilon}}, \nabla_{\varepsilon}|\nabla_{\varepsilon}v|_{\varepsilon}^2\right\rangle_{\varepsilon}^2
		\end{aligned}
	\end{equation}		
\end{corollary}
Finally, the following auxiliary result gives an estimate for the remaining term in the right hand side of \eqref{eq:KN-identity-Improved} by crucially exploiting the orthogonality property of $\nabla_{0,J}$.

\begin{lemma}\label{lem:Stima_4_Grad_simplettico}
	Let $1<p\leq 2$, $\varepsilon \in(0,1]$ and $\widehat{\Omega} \subset \He^n$ an open set. If $v\in C^2(\widehat{\Omega})$ is such that $\nabla_\varepsilon v \neq 0$ in $\widehat{\Omega}$ and $\nabla_{0,J}v$ is as in \eqref{eq:Grad_simplettico} then   
	\begin{equation}\label{eq:Stima_4_Grad_simplettico}
		\begin{aligned}
			-4 |\nabla_{\varepsilon}v|_{\varepsilon}^{p-2} & \langle \nabla_{0,J}v, \nabla_{\varepsilon}T v \rangle_{\varepsilon} \leq \dfrac{p-1}{2}|\nabla_{\varepsilon}v|_{\varepsilon}^{p-2} |\nabla_{\varepsilon}T v|_{\varepsilon}^{2} + 8 |\nabla_{\varepsilon}v|_{\varepsilon}^{p}\\
			&+ (2-p) |\nabla_{\varepsilon}v|_{\varepsilon}^{p-2} \left(|\nabla_{\varepsilon}T v|_{\varepsilon}^{2} - \left\langle\dfrac{\nabla_{\varepsilon}v}{|\nabla_{\varepsilon}v|_{\varepsilon}}, \nabla_{\varepsilon}T v\right\rangle_{\varepsilon}^2\right)
		\end{aligned}
	\end{equation}
	\begin{proof}
	We may assume that $\nabla_{0,J}v \neq 0$, otherwise the claim holds trivially, all the terms on the right hand side being nonnegative. Thus, writing $1=(p-1)+(2-p)$ and using Cauchy-Schwarz inequality and Young inequality, we get
		\begin{equation}\label{eq:Prima_stima_grad_simplettico}
			\begin{aligned}
				-4 &\langle \nabla_{0,J}v, \nabla_{\varepsilon}T v \rangle_{\varepsilon}  \leq 4 (p-1) |\nabla_{0,J}v|_{\varepsilon} |\nabla_{\varepsilon}T v|_{\varepsilon} + 4(2-p) \left|\left\langle \dfrac{\nabla_{0,J}v}{|\nabla_{0,J}v|_{\varepsilon}}, \nabla_{\varepsilon}T v \right\rangle_{\varepsilon}\right|\,\,|\nabla_{0,J}v|_{\varepsilon}\\
				&\leq \dfrac{p-1}{2}|\nabla_{\varepsilon}T v|_{\varepsilon}^2 + 8(p-1)|\nabla_{0,J}v|_{\varepsilon}^2  + (2-p)\left|\left\langle \dfrac{\nabla_{0,J}v}{|\nabla_{0,J}v|_{\varepsilon}}, \nabla_{\varepsilon}T v \right\rangle_{\varepsilon}\right|^2 + 4(2-p)|\nabla_{0,J}v|_{\varepsilon}^2 \\
				&\leq \dfrac{p-1}{2}|\nabla_{\varepsilon}T v|_{\varepsilon}^2 + (2-p)\left|\left\langle \dfrac{\nabla_{0,J}v}{|\nabla_{0,J}v|_{\varepsilon}}, \nabla_{\varepsilon}T v \right\rangle_{\varepsilon}\right|^2 + 8|\nabla_{\varepsilon}v|_{\varepsilon}^2.
			\end{aligned}
		\end{equation}
		\noindent where in the last inequality we used that $|\nabla_{0,J}v|_{\varepsilon} \leq |\nabla_{\varepsilon}v|_{\varepsilon}$ and that $1<p\leq 2$.
		It remains to notice that the orthogonality of $\nabla_{0,J}v$ and $\nabla_{\varepsilon}v$ (w.r.to $\langle \cdot,\cdot\rangle_{\varepsilon}$) implies
		\begin{equation}\label{eq:Stima_da_ortogonalita}
			\left|\left\langle \dfrac{\nabla_{0,J}v}{|\nabla_{0,J}v|_{\varepsilon}}, \nabla_{\varepsilon}T v \right\rangle_{\varepsilon}\right|^2 \leq |\nabla_{\varepsilon}T v|_{\varepsilon}^{2} - \left\langle\dfrac{\nabla_{\varepsilon}v}{|\nabla_{\varepsilon}v|_{\varepsilon}}, \nabla_{\varepsilon}T v\right\rangle_{\varepsilon}^2.
		\end{equation}
		Plugging \eqref{eq:Stima_da_ortogonalita} into \eqref{eq:Stima_4_Grad_simplettico} and multiplying on both sides by $|\nabla_{\varepsilon}v|_{\varepsilon}^{p-2}$, we get the desired conclusion.	 	
	\end{proof}
\end{lemma}

Combining \eqref{eq:KN-identity-Improved} with \eqref{eq:Lv_di_Tv2}, Corollary \ref{cor:from_kato_improved} and Lemma \ref{lem:Stima_4_Grad_simplettico} we finally get the following Bochner inequality.
\begin{proposition}\label{prop:Lv_dinabla_2_e_Tv2}
	Let $1<p\leq 2$, $\varepsilon \in(0,1]$, $\widehat{\Omega} \subset \He^n$ an open set and $v\in C^3(\widehat{\Omega})$ such that $\nabla_\varepsilon v \neq 0$ in $\widehat{\Omega}$. If $v$ is $p$-harmonic in $\widehat{\Omega}$ and $L_v$ is the linear operator defined as in \eqref{L_v} then 
	\begin{equation}
		\begin{aligned}
			L_{v}(|\nabla_{\varepsilon}v|_{\varepsilon}^2 + (T v)^2) &\geq 
			\dfrac{3}{2}(p-1)|\nabla_{\varepsilon}v|_{\varepsilon}^{p-2} \,\, |\nabla_{\varepsilon}T v|_{\varepsilon}^2 - 8 |\nabla_{\varepsilon}v|_{\varepsilon}^p\\
			&+ (2-p)|\nabla_{\varepsilon}v|_{\varepsilon}^{p-2} \, \left(|\nabla_{\varepsilon}T v|_{\varepsilon}^2 - \left\langle \dfrac{\nabla_{\varepsilon}v}{|\nabla_{\varepsilon}v|_{\varepsilon}}, \nabla_{\varepsilon}T v \right\rangle^{2}_{\varepsilon}\right)\\
			&-\dfrac{2-p}{2}|\nabla_{\varepsilon}v|_{\varepsilon}^{p-4} \left(|\nabla_{\varepsilon}|\nabla_{\varepsilon}v|_{\varepsilon}^2|^2_{\varepsilon} - \left\langle \dfrac{\nabla_{\varepsilon}v}{|\nabla_{\varepsilon}v|_{\varepsilon}}, \nabla_{\varepsilon}|\nabla_{\varepsilon}v|_{\varepsilon}^2 \right\rangle_{\varepsilon}^2\right)\\
			& + \left(\dfrac{p-1}{2} + \dfrac{(p-1)^2}{4n}\right)|\nabla_{\varepsilon}v|_{\varepsilon}^{p-4} \left\langle \dfrac{\nabla_{\varepsilon}v}{|\nabla_{\varepsilon}v|_{\varepsilon}}, \nabla_{\varepsilon}|\nabla_{\varepsilon}v|_{\varepsilon}^2\right\rangle_{\varepsilon}^2.
		\end{aligned}
	\end{equation}
\end{proposition}

\section{Uniform gradient estimates for $p$-capacitary potentials.}

In this section we provide the necessary tools for the proof of Theorem \ref{Thm:differential-harnack} and in turn Corollary \ref{Cor:gradbound}. Thus, for $1<p\leq 2$ and $\varepsilon \in (0,1]$, we obtain the pointwise gradient estimate \eqref{diffharnack} for the $p-$capacitary potentials $v_p^\varepsilon$ associated to $\overline{\Omega}^c$ and in turn the corresponding bound \eqref{globalgradbound} for solutions $u_p^\varepsilon$ of \eqref{EqForUpeps}, while the proofs of the main results above are postponed to the next section.

 As announced in the Introduction, here we first establish in  Proposition \ref{p-harm-bd-bound} the boundary gradient estimate and then we show how  to propagate in the interior under the assumption $|\egrad v^\varepsilon_p|=o(v^\varepsilon_p)$ at infinity (such an assumption will be removed in the next section). This second step will require preliminarily to obtain the differential Harnack inequality for the vertical derivative in Lemma \ref{vert-bd-int-bound} and in turn this result will yield in Theorem \ref{bd-interior-bound}) a gradient estimate in the interior.  

The next result provides uniform control for the gradient of each Riemannian $p-$capacitary potential at the boundary as a consequence of the exterior gauge-ball condition. 

\begin{proposition}
\label{p-harm-bd-bound}
Let $\Omega \subset \He^n$ be an open set with $C^2$-smooth boundary and bounded complement satisfying assumption ($\mathbf{HP_\Omega}$) with parameter $R_0$. For $1<p <Q$ and $\varepsilon \in(0,1]$ let $v=v^\varepsilon_p \in   \dot{W}^{1,p}_{1,\varepsilon}(\Omega) \cap C^{1,\beta}(\overline{\Omega})$ be the solution to \eqref{EqForVpeps} corresponding to the unique minimizers of \eqref{p-dirichlet}  
in the set $\dot{W}^{1,p}_{1, \varepsilon} (\Omega)$.

Then, there exist $K>0$ and $\overline{C}>0$ depending only on $R_0$ and $Q$ such that the following inequalities on $\partial \Omega$ hold
\begin{equation}
\label{bd-grad-bound}
 |\nabla_0  v|_0  \leq \frac{4 K}{R_0(p-1)} \, , \qquad  | T v  | \leq \frac{8 K }{R_0^2(p-1)} \, , \qquad 
 |\nabla_\varepsilon  v|_\varepsilon \leq \frac{\overline{C}}{p-1} v \, .
\end{equation}

\end{proposition}


\begin{proof}
The claim will follow from the comparison principle between $v$ and suitable subsolutions  constructed in terms of the functions given by \eqref{PhiR0}.  More precisely, inequalities \eqref{bd-grad-bound} will follow with the choice $\overline{C}=C_0 K$, with $C_0=\frac{4}{R_0} \sqrt{1+\frac{4}{R_0^2}}$ and $K= \frac{Q-1}4 +  \left(\frac{2}{R_0^2} +  \frac{8 Q}{R_0^4} \right) $, so that \eqref{normderPhir0} in Lemma \ref{propertiesPhialpha} hold and Proposition \ref{SubBarrier} applies, both statements being therefore valid in the whole range $1<p <Q$ and $\varepsilon \in (0,1]$. 

Thus, for $R_0$ as in assumption ($\mathbf{HP_\Omega}$) and for any $g \in \partial \Omega$ let us choose  $g_0\in \He^n$ such that $B_{R_0}(g_0) \subset \Omega^{c}$ and $\partial \Omega \cap \partial B_{R_0}(g_0)=\{ g\}$. Now consider $\Phi_\alpha=\Phi_\alpha^{(g_0,R_0)}$ as in \eqref{PhiR0} and $v \in \dot{W}^{1,p}_{1, \varepsilon} (\Omega)$ the Riemannian $p$-harmonic function associated to $\Omega$, so that $v \leq 1$ in $\overline{\Omega}$ in view of the the weak maximum principle (equivalently, $v=v \wedge 1$ since truncation decreases the energy).
Here $\alpha=-K/(p-1)$ and $K=K(R_0,Q)$ is chosen as above so that inequality \eqref{goal} holds for any $g_0$ as above and any $\varepsilon \in (0,1]$, i.e., the function $\Phi_\alpha$ is $p$-subharmonic in $\He^n\setminus B_{R_0}(g_0)$, hence in $\Omega$, for any $\varepsilon \in (0,1]$. 
Note that by construction $B_{R_0}(g_0) \subset \Omega^{c} $, hence $ \Phi_\alpha \leq 1$ on $\overline{\Omega}$. Thus, by the weak comparison principle for (sub/super) $p$-harmonic functions in $\dot{W}^{1,p}_\varepsilon(\Omega)$ we have $\Phi_\alpha \leq v$ a.e. in $\Omega$ and in turn everywhere in $\overline{\Omega}$ because such functions are both continuous in $\overline{\Omega}$, moreover equality holds at $g \in \partial \Omega\cap \partial B_{r_0}(g_0)$.

 Next, note that in view of ($\mathbf{HP_\Omega}$) both $ B_{R_0}(g_0)^c= \{ \Phi_\alpha \leq 1\}$ and $\Omega$ have smooth boundaries which are tangent at $g \in \partial \Omega \cap \partial \left(B_{R_0}(g_0)^c \right)$, therefore they have the same outer normal vector there. Now, $\Phi_\alpha$ (resp. $v$) has zero tangential gradient at $g$ because it is identically one on $\partial \left( B_{R_0}(g_0)^c \right)$ (resp. on $\partial \Omega$) by construction. On the other hand since $0<\Phi_\alpha \leq v \leq 1$ on $\overline{\Omega}$ then their derivatives are ordered when computed along the outer normal (more precisely, even along its horizontal and its vertical component separately) therefore $|\nabla_\varepsilon  v (g)|_\varepsilon \leq |\nabla_\varepsilon  \Phi_\alpha|_\varepsilon (g) $ and $|T v (g)| \leq |T \Phi_\alpha (g)|$. Hence, inequality \eqref{bd-grad-bound} on $\partial \Omega$ follows from inequality \eqref{normderPhir0} in Lemma \ref{propertiesPhialpha} since $v(g)=1$ and $g \in \partial \Omega$ is arbitrary.  
 \end{proof}

\smallskip

  In the next two results we propagate the boundary gradient estimates for the $p$-capacitary potentials $v^\varepsilon_p$ to the whole exterior domain. In the next lemma we first deal with the vertical derivative.

\begin{lemma}
\label{vert-bd-int-bound}

Let $\Omega \subset \He^n$ be an open set with $C^2$-smooth boundary and bounded complement such that $\Omega$ satisfies an exterior uniform gauge-ball condition ($\mathbf{HP_\Omega}$) with parameter $R_0$. For $1<p <Q$ and $0<\varepsilon \leq 1$ let $v=v_p^\varepsilon \in  \dot{W}^{1,p}_{1,\varepsilon}(\Omega) \cap C^{1,\beta}(\overline{\Omega})$ be the solution to \eqref{EqForVpeps} corresponding to the unique minimizers of \eqref{p-dirichlet}  
in the set $\dot{W}^{1,p}_{1, \varepsilon} (\Omega)$. Assume that $|\nabla_\varepsilon  v |_\varepsilon=o(v)$ as $\| g\| \to \infty$.

Then
\begin{equation}
\label{Tv-bd-int-bound}
\left\| \frac{\left(T  v \right)^2}{v^2}   \right\|_{L^\infty ( \Omega)} 
\leq \left\| \frac{\left(T  v \right)^2 }{v^2}  \right\|_{L^\infty ( \partial\Omega)} \, ,
\end{equation}
\end{lemma}
\begin{proof}
First notice that for $v=v^\varepsilon_p$ the r.h.side of \eqref{Tv-bd-int-bound} is finite in view of Proposition \ref{p-harm-bd-bound}. We argue by contradiction and suppose \eqref{Tv-bd-int-bound} is false.   Note that in view of the $C^{1,\beta}$-regularity of $v$ in the whole $\overline{\Omega}$ and the assumption  $|\nabla_\varepsilon  v|_\varepsilon=o(v)$ as $\| g\| \to \infty$ the left hand side of \eqref{Tv-bd-int-bound} is by continuity a maximum denoted by $M$ in what follows. Being by contradiction an interior maximum, there exists $\bar{g} \in \Omega$ such that
\begin{equation}
\label{def:M}
M=\frac{\left(T  v\right)^2}{v^2}  (\bar{g}) > \left\| \frac{\left(T  v\right)^2 }{v^2}  \right\|_{L^\infty ( \partial\Omega)} \, . 
\end{equation}

Next, in view of  the assumption $|\nabla_\varepsilon  v|_\varepsilon=o(v)$ as $\| g\| \to \infty$ we fix can fix $R>0$ so large that $\bar{g} \in \Omega_R:= \Omega \cap B_R(0)$ for any maximum point $\bar{g}$ as in \eqref{def:M}, $\partial B_R(0) \cap \partial \Omega=\emptyset$ and $(T v)^2  \leq \frac{M}2 v^2$
 on $\partial B_R(0)$. 

Now let us define the auxiliary function 
\[ w= \left(T   v \right)^2-M v^2 \, \]
and notice that $w \leq 0$ in $\overline{\Omega_R}$, $w(\bar{g})=0$ and $w <0$ on $\partial \Omega_R$ by our choice or $R>0$, because $v >0$ on $\overline{\Omega_R}$ and \eqref{def:M} holds.
Thus, any $\bar{g}$ as in \eqref{def:M} is an interior local (actually global) maximum for $w$ on $\overline{\Omega_R}$ and $w(\bar{g})=0$ yields $(T v )^2(\bar{g})=M (v (\bar{g}))^2 > 0$, hence $\nabla_\varepsilon v (\bar{g}) \neq 0$; in particular, $v$ being $C^1$ it satisfies $T v \neq 0$ and $\nabla_\varepsilon v  \neq 0$ near $\bar{g}$ and therefore $v$ is smooth near $\bar{g}$ by elliptic regularity.

Similarly, for any $\sigma>0$ define the auxiliary functions 
\[ w_\sigma:= \left(T  v \right)^2-(M+\sigma) v^2 =w-\sigma v^2 \, , \]

\noindent and notice that $w_\sigma< 0$ in $\overline{\Omega_R}$, because $w\leq 0$ and $ \min_{\overline{\Omega_R}} v  >0$. Clearly $w_\sigma \to w$ uniformly on $\overline{\Omega_R}$, hence for $\sigma$ small enough the set of maximum points for $w_\sigma$ is well inside $\Omega_R$, as the same property for $w$ has been shown above;  furthermore for $\sigma$ small enough there are maxima $\bar{g}_\sigma \in \Omega_R$ for $w_\sigma$ such that up to subsequences $\bar{g}_\sigma \to \bar{g} \in \Omega_R$ for some $\bar{g}$ as above. 

Sufficiently close to $\bar{g}$ we consider the uniformly elliptic operator
$L_v$ as defined in \eqref{L_v} with matrix $A_v$ as defined in \eqref{def-A_v}.
Note that $L_v(v)=0$ near $\bar{g}$ because $v$ is $p$-harmonic. Moreover, since $ v$ is smooth near $\bar{g}$ we also obtain $L_v(T_\varepsilon v)=0$ near $\bar{g}$ as already observed in the previous section.

Since $L_v$ is an elliptic operator and $\bar{g}_\sigma$ is an interior local maximum point for $w_\sigma$, from the smoothness of $v$ and in turn of $w_\sigma$ near $\bar{g}$ (i.e., near $\bar{g}_\sigma$ for $\sigma$ small enough) we have 
\begin{equation}\label{AtMaxwsigma}
L_v (w_\sigma)(\bar{g}_\sigma) \leq 0.
\end{equation}

On the other hand, the identities $L_v(v)=L_v(T_v)=0$ and  $L_v(f^2)=2 f L_v(f)+2\escal{A_v \egrad f}{\egrad f} $ for sufficiently smooth functions $f$, yield
\begin{equation}
\label{Lvw}
L_v(w_\sigma) = L_v \left((T_ v)^2 - (M+\sigma) v^2\right) = 2 \escal{A_v \egrad T v}{\egrad T v} - 2(M+\sigma)  \escal{A_v \egrad  v}{\egrad  v}
 \, .
\end{equation}
Since $\bar{g}_\sigma$ is an interior local maximum point for $w_\sigma$, at that point 
$\egrad w_\sigma = 0 $ yields $\egrad T v = (M+\sigma) \dfrac{v \egrad v}{T  v} \, ,$
\noindent therefore, again in $\bar{g}_\sigma$, we get
$$ 2 \escal{A_v \egrad Tv}{\egrad Tv} = 2 (M+\sigma)^2 \dfrac{v^2}{(Tv)^2}   \escal{A_v \egrad  v}{\egrad  v} \, .$$
\noindent Combining the last identity with \eqref{AtMaxwsigma}, \eqref{Lvw}  and \eqref{def-A_v}, since $w_\sigma(\bar{g}_\sigma)<0$ we obtain  
$$0 \geq L_v (w_\sigma)(\bar{g}_\sigma)=  2 (p-1) (M+\sigma)|\egrad v (\bar{g}_\sigma)|_\varepsilon^{p} \left(   (M+\sigma) \dfrac{v^2}{(T v)^2}(\bar{g}_\sigma) -1\right) >0.$$
\noindent Thus, the contradiction proves that
$\dfrac{ (T_\varepsilon v)^2}{v^2}$ attains its maximum at the boundary $\partial \Omega$ as desired.
\end{proof}

 The next result is the analogue of the previous lemma when dealing with the full gradient.

Once a global differential Harnack's inequality is obtained for the vertical derivative we are finally in the position to prove a similar result for the full gradient. As in the previous lemma the proof will rely on a maximum principle argument for an auxiliary function $\Psi_p$ for which in the next lemma we derive the corresponding useful identities at interior maximum points.

\begin{lemma}\label{lem:conti_in_xp}
Let $\Omega \subset \He^n$ be an open set with $C^2$-smooth boundary and bounded complement such that $\Omega$ satisfies an exterior uniform gauge-ball condition ($\mathbf{HP_\Omega}$) with parameter $R_0$. For $1<p\leq 2$ and $\varepsilon \in (0,1]$ let $v=v_p^\varepsilon \in  \dot{W}^{1,p}_{1,\varepsilon}(\Omega) \cap C^{1,\beta}(\overline{\Omega})$ be the solution to \eqref{EqForVpeps} corresponding to the unique minimizers of \eqref{p-dirichlet}  
in the set $\dot{W}^{1,p}_{1, \varepsilon} (\Omega)$. 

Under the assumption that $|\nabla_\varepsilon  v |_\varepsilon=o(v)$ as $\| g\| \to \infty$ let us define
\begin{equation}\label{eq:def_Cp2}
	C_p^2:= \max_{\overline{\Omega}}(p-1)^{2}\dfrac{|\nabla_{\varepsilon}v|_{\varepsilon}^{2} + (Tv)^2}{v^2} \, ,
\end{equation}	
so that $C_{p}^2 < +\infty$ for any $1 <p \leq 2$, and let us also define the auxiliary function
\begin{equation}\label{eq:Def_Psi_p}
	\Psi_{p} := (p-1)^{2}|\nabla_{\varepsilon}v|_{\varepsilon}^{2} + (p-1)^{2}(Tv)^2 - C_p^2 v^2,
\end{equation}
so that $\psi_p \in C_b(\overline{\Omega})$ and $\psi_p \leq 0$ in $\overline{\Omega}$ by its very definition.

	Assume that there exists a point $x_p \in \Omega$ such that $\psi_{p}(x_p)=0$.  Then  at $x_p$ we have $\egrad v (x_p) \neq 0$ and the following statements hold:
	\begin{itemize}
		\item[i)]  \[(p-1)^2 |\nabla_{\varepsilon}v|_{\varepsilon}^{2} + (p-1)^2(Tv)^2 = C_p^2 \, v^2 \, . \]
		\item[ii)] \[ (p-1)^2 \nabla_{\varepsilon}|\nabla_{\varepsilon}v|_{\varepsilon}^2 = 2 C_p^2 v \nabla_{\varepsilon}v - 2(p-1)^2 Tv \nabla_{\varepsilon}Tv \, . \]
		\item[iii)] \[ (p-1)^2 \langle \nabla_{\varepsilon}|\nabla_{\varepsilon}v|_{\varepsilon}^{2}, \nabla_{\varepsilon}v \rangle_{\varepsilon} = 2C_{p}^2v |\nabla_{\varepsilon}|_{\varepsilon}^{2} - 2 (p-1)^{2}Tv \langle \nabla_{\varepsilon}Tv, \nabla_{\varepsilon}v\rangle_{\varepsilon} \, . \]
		\item[iv)] \[ \dfrac{|\nabla_{\varepsilon}|\nabla_{\varepsilon}v|_{\varepsilon}^{2}|_{\varepsilon}^{2}}{4|\nabla_{\varepsilon}v|_{\varepsilon}^2} = \dfrac{C_p^4 v^2}{(p-1)^4} - 2\dfrac{C_p^2}{(p-1)^2} v \dfrac{Tv}{|\nabla_{\varepsilon}v|_{\varepsilon}}\left\langle \dfrac{\nabla_{\varepsilon}v}{|\nabla_{\varepsilon}v|_{\varepsilon}}, \nabla_{\varepsilon}Tv \right\rangle_{\varepsilon} + \dfrac{(Tv)^2}{|\nabla_{\varepsilon}v|_{\varepsilon}^2} |\nabla_{\varepsilon}Tv|_{\varepsilon}^2. \]
		\item[v)] 
		\begin{equation*}
		\begin{aligned} \dfrac{1}{4|\nabla_{\varepsilon}v|_{\varepsilon}^2} \left\langle \nabla_{\varepsilon}|\nabla_{\varepsilon}v|_{\varepsilon}^2, \dfrac{\nabla_{\varepsilon}v}{|\nabla_{\varepsilon}v|_{\varepsilon}}\right\rangle_{\varepsilon}^2 
		& = \dfrac{C_p^4 v^2}{(p-1)^4} - 2\dfrac{C_p^2}{(p-1)^2} v \dfrac{Tv}{|\nabla_{\varepsilon}v|_{\varepsilon}}\left\langle \dfrac{\nabla_{\varepsilon}v}{|\nabla_{\varepsilon}v|_{\varepsilon}}, \nabla_{\varepsilon}Tv \right\rangle_{\varepsilon} \\& + \dfrac{(Tv)^2}{|\nabla_{\varepsilon}v|_{\varepsilon}^2} \left\langle \dfrac{\nabla_{\varepsilon}v}{|\nabla_{\varepsilon}v|_{\varepsilon}},\nabla_{\varepsilon}Tv \right\rangle_{\varepsilon}^2. 
		\end{aligned}
		\end{equation*}
		\item[vi)]
		\[ \dfrac{1}{4|\nabla_{\varepsilon}v|_{\varepsilon}}\left(|\nabla_{\varepsilon}|\nabla_{\varepsilon}v|_{\varepsilon}^{2}|_{\varepsilon}^{2} - \left\langle \nabla_{\varepsilon}|\nabla_{\varepsilon}v|_{\varepsilon}^2, \dfrac{\nabla_{\varepsilon}v}{|\nabla_{\varepsilon}v|_{\varepsilon}}\right\rangle_{\varepsilon}^2 \right) =  \dfrac{(Tv)^2}{|\nabla_{\varepsilon}v|_{\varepsilon}^2} \left( |\nabla_{\varepsilon}Tv|_{\varepsilon}^2 -  \left\langle \dfrac{\nabla_{\varepsilon}v}{|\nabla_{\varepsilon}v|_{\varepsilon}},\nabla_{\varepsilon}Tv \right\rangle_{\varepsilon}^2 \right). \]
		\item[vii)]$\dfrac{\left\langle \nabla_{\varepsilon}|\nabla_{\varepsilon}v|_{\varepsilon}^2, \dfrac{\nabla_{\varepsilon}v}{|\nabla_{\varepsilon}v|_{\varepsilon}}\right\rangle_{\varepsilon}^2}{|\nabla_{\varepsilon}v|_{\varepsilon}^2} \geq \dfrac{4C_p^2}{(p-1)^2}(|\nabla_{\varepsilon}v|_{\varepsilon}^2 + (Tv)^{2}) - 2|\nabla_{\varepsilon}Tv|_{\varepsilon}^2 - 8 |\nabla_{\varepsilon}v|_{\varepsilon}^2 \left(\dfrac{T_v}{v}\right)^2 \left(1+ \dfrac{(Tv)^2}{|\nabla_{\varepsilon}v|_{\varepsilon}^2}\right)^{2}.$
	\end{itemize}
	\begin{proof}
	All the statements $i)$-$vi)$ are straightforward consequence of \eqref{eq:def_Cp2} and \eqref{eq:Def_Psi_p} except the last one.

In order to prove $vii)$, starting from $v)$, neglecting a square and using Young's inequality we infer
		\begin{equation}
			\begin{aligned}
				\dfrac{1}{|\nabla_{\varepsilon}v|_{\varepsilon}^2} \left\langle \nabla_{\varepsilon}|\nabla_{\varepsilon}v|_{\varepsilon}^2, \dfrac{\nabla_{\varepsilon}v}{|\nabla_{\varepsilon}v|_{\varepsilon}}\right\rangle_{\varepsilon}^2 &\geq \dfrac{4\, C_p^4 v^2}{(p-1)^4} - 2\,\left\langle \dfrac{\nabla_{\varepsilon}v}{|\nabla_{\varepsilon}v|_{\varepsilon}}, \nabla_{\varepsilon}Tv \right\rangle_{\varepsilon} \cdot  4\dfrac{C_p^2}{(p-1)^2} v \dfrac{Tv}{|\nabla_{\varepsilon}v|_{\varepsilon}}\\
				&\geq \dfrac{4\, C_p^4 v^2}{(p-1)^4} - 2|\nabla_{\varepsilon}Tv|_{\varepsilon}^2 - 8\dfrac{C_p^4}{(p-1)^4}v^2 \dfrac{(Tv)^2}{|\nabla_{\varepsilon}v|_{\varepsilon}^2},							
			\end{aligned}
		\end{equation}
		\noindent  Finally, inserting the formula for $C_p^2$ from $i)$ we reach the desired the conclusion.
	\end{proof}
\end{lemma}

We are finally in the position to state and prove the main result of this section, i.e., to establish the global Cheng-Yau's inequality assuming smallness of the gradient at infinity as a consequence of the uniform boundary estimates proved above.

Indeed, under the assumption that $|\nabla_\varepsilon  v |_\varepsilon=o(v)$ as $\| g\| \to \infty$ and in view of Proposition \ref{p-harm-bd-bound} and Lemma \ref{vert-bd-int-bound} let us define $M_{\ast}=M_{\ast}(Q,R_0)$ as
\begin{equation}\label{eq:def_M_ast}
	M_{\ast} := \sup_{1<p\leq 2} \max_{ \overline{\Omega}} \, \, (p-1)\dfrac{|Tv|}{v}= \sup_{1<p\leq 2} \max_{\partial \Omega} \, \, (p-1)\dfrac{|Tv|}{v} <\infty \, ,
\end{equation} 
and, similarly, let us define $\overline{M}=\overline{M}(Q, R_0)$ as 

\begin{equation}\label{eq:def_M_bar}
	\overline{M}:= \sup_{1<p\leq 2} \max_{\partial\Omega}(p-1)\dfrac{|\nabla_{\varepsilon}v|_{\varepsilon}}{v} <\infty \, .
\end{equation}



\begin{theorem}
\label{bd-interior-bound}

Let $\Omega \subset \He^n$ be an open set with $C^2$-smooth boundary and bounded complement such that $\Omega$ satisfies an exterior uniform gauge-ball condition ($\mathbf{HP_\Omega}$) with parameter $R_0$. For $1<p\leq 2$ and $\varepsilon \in (0,1]$ let $v=v_p^\varepsilon \in  \dot{W}^{1,p}_{1,\varepsilon}(\Omega) \cap C^{1,\beta}(\overline{\Omega})$ be the solution to \eqref{EqForVpeps} corresponding to the unique minimizers of \eqref{p-dirichlet}  
in the set $\dot{W}^{1,p}_{1, \varepsilon} (\Omega)$. 
Under the assumption that $|\nabla_\varepsilon  v |_\varepsilon=o(v)$ as $\| g\| \to \infty$ let us define $M_{\ast}$ and $\overline{M}$ as in \eqref{eq:def_M_ast} and \eqref{eq:def_M_bar}.
Thus, if we set for any $1<p\leq 2$
\begin{equation}\label{eq:def_M_p}
	M_p := \max_{\overline{\Omega}}(p-1)\dfrac{|\nabla_{\varepsilon}v|_{\varepsilon}}{v}
\end{equation}
then we have the uniform bound 
\begin{equation}\label{eq:stima_M_p}
M_{p}^{2} \leq \overline{M}^{2} + 5QM_{\ast}^2 + 4Q.
\end{equation}

\begin{proof}
Assume by contradiction that \eqref{eq:stima_M_p} is false. Then, recalling the definition of $C^2_p$ in \eqref{eq:def_Cp2} combined with \eqref{eq:def_M_p} and recalling that $Q=2n+2$ we obtain
\begin{equation}\label{eq:C_p2_dal_basso}
C^2_{p} \geq M_{p}^2 > \overline{M}^{2} + 5QM_{\ast}^2 + 4Q > \overline{M}^{2} + 10 n \, M_{\ast}^2 + 8n > \overline{M}^{2} + M_{\ast}^2.
\end{equation}	
In particular, \eqref{eq:C_p2_dal_basso} shows that $C_p^2$ is not achieved on $\partial\Omega$, otherwise we would clearly have $C_p^2 \leq \overline{M}^2 + M_{\ast}^2$, which cannot be true in view of \eqref{eq:C_p2_dal_basso}. Therefore, $C^2_p$ is achieved at some $x_p \in \Omega$, i.e, we can apply Lemma \ref{lem:conti_in_xp} to the function $\Psi_p$ defined in \eqref{eq:Def_Psi_p} at the point $x_p \in \Omega$, where $\Psi_p(x_p)=0$ and $\egrad v (x_p)\neq 0$, so that $v$ is smooth in a neighborhood of this point. Moreover, still in view of \eqref{eq:C_p2_dal_basso} we have that $\nabla_{0}v(x_p)\neq 0$ (otherwise we would have $C_p^{2} \leq 2 M_{\ast}^2$ for $\varepsilon \in (0,1]$ because of its very definition in \eqref{eq:def_Cp2}, which would trivially contradict \eqref{eq:C_p2_dal_basso}) and hence in particular $\nabla_{0,J}v(x_p)\neq 0$.
		
Moreover, $x_p$ being an interior maximum point for $\Psi_p$ in $\Omega$, ellipticity of $L_v$ in a neighborhood of $x_p$ implies $L_{v}\psi(x_p) \leq 0$, therefore the Bochner inequality in Proposition \ref{prop:Lv_dinabla_2_e_Tv2} together with \eqref{eq:Lv_di_v2} at the point $x_p$ yield 
\begin{equation}\label{eq:Lv_psi_p_negativo}
\begin{aligned}
				0 &\geq L_{v}\Psi_{p} =(p-1)^2 L_{v}(|\nabla_{\varepsilon}v|_{\varepsilon}^2 + (Tv)^2) - C_{p}^{2}L_{v}v^2 \\
				&\geq 
				\dfrac{3}{2}(p-1)^{3}|\nabla_{\varepsilon}v|_{\varepsilon}^{p-2} \,\, |\nabla_{\varepsilon}T v|_{\varepsilon}^2 - 8 (p-1)^{2} |\nabla_{\varepsilon}v|_{\varepsilon}^p\\
				&+ (2-p)(p-1)^2 \, |\nabla_{\varepsilon}v|_{\varepsilon}^{p-2} \, \left(|\nabla_{\varepsilon}T v|_{\varepsilon}^2 - \left\langle \dfrac{\nabla_{\varepsilon}v}{|\nabla_{\varepsilon}v|_{\varepsilon}}, \nabla_{\varepsilon}T v \right\rangle^{2}_{\varepsilon}\right)\\
				&-\dfrac{2-p}{2}(p-1)^{2}\, |\nabla_{\varepsilon}v|_{\varepsilon}^{p-4} \left(|\nabla_{\varepsilon}|\nabla_{\varepsilon}v|_{\varepsilon}^2|^2_{\varepsilon} - \left\langle \dfrac{\nabla_{\varepsilon}v}{|\nabla_{\varepsilon}v|_{\varepsilon}}, \nabla_{\varepsilon}|\nabla_{\varepsilon}v|_{\varepsilon}^2 \right\rangle_{\varepsilon}^2\right)\\
				& + (p-1)^{2}\left(\dfrac{p-1}{2} + \dfrac{(p-1)^2}{4n}\right)|\nabla_{\varepsilon}v|_{\varepsilon}^{p-4} \left\langle \dfrac{\nabla_{\varepsilon}v}{|\nabla_{\varepsilon}v|_{\varepsilon}}, \nabla_{\varepsilon}|\nabla_{\varepsilon}v|_{\varepsilon}^2\right\rangle_{\varepsilon}^2\\
				&-2 \, C_{p}^{2} \, (p-1) \, |\nabla_{\varepsilon}v|_{\varepsilon}^{p}.
\end{aligned}
\end{equation}

Notice now that, since $1<p \leq 2$ and $n\geq 1$, it follows that $\tfrac{p-1}{2} + \tfrac{(p-1)^{2}}{4n} \leq \tfrac{3}{4}(p-1)$. Therefore, rearranging the terms in \eqref{eq:Lv_psi_p_negativo} and using claim vii) from Lemma \ref{lem:conti_in_xp}, we have that
		\begin{equation}\label{eq:Rearranging}
			\begin{aligned}
				4\, C_{p}^{2}\, &\left(\dfrac{p-1}{2} + \dfrac{(p-1)^2}{4n}\right)\,  |\nabla_{\varepsilon}v|_{\varepsilon}^{p-2}(|\nabla_{\varepsilon}v|_{\varepsilon}^{2}+(Tv)^2) + \cancel{\dfrac{3}{2}(p-1)^{3}|\nabla_{\varepsilon}v|_{\varepsilon}^{p-2} \,\, |\nabla_{\varepsilon}T v|_{\varepsilon}^2} \\
				&+ (2-p)(p-1)^2 \, |\nabla_{\varepsilon}v|_{\varepsilon}^{p-2} \, \left(|\nabla_{\varepsilon}T v|_{\varepsilon}^2 - \left\langle \dfrac{\nabla_{\varepsilon}v}{|\nabla_{\varepsilon}v|_{\varepsilon}}, \nabla_{\varepsilon}T v \right\rangle^{2}_{\varepsilon}\right)\\
				&\leq 2 \, C_{p}^{2} \, (p-1) \, |\nabla_{\varepsilon}v|_{\varepsilon}^{p} + 8 (p-1)^{2} |\nabla_{\varepsilon}v|_{\varepsilon}^p\\
				&+ \dfrac{2-p}{2}(p-1)^{2}\, |\nabla_{\varepsilon}v|_{\varepsilon}^{p-4} \left(|\nabla_{\varepsilon}|\nabla_{\varepsilon}v|_{\varepsilon}^2|^2_{\varepsilon} - \left\langle \dfrac{\nabla_{\varepsilon}v}{|\nabla_{\varepsilon}v|_{\varepsilon}}, \nabla_{\varepsilon}|\nabla_{\varepsilon}v|_{\varepsilon}^2 \right\rangle_{\varepsilon}^2\right)\\
				&+ \left(\dfrac{p-1}{2} + \dfrac{(p-1)^2}{4n}\right)\, \left( 2(p-1)^{2}|\nabla_{\varepsilon}v|_{\varepsilon}^{p-2}\,|\nabla_{\varepsilon}Tv|_{\varepsilon}^2 + 8(p-1)^{2}|\nabla_{\varepsilon}v|_{\varepsilon}^{p}\,  \left(\dfrac{T_v}{v}\right)^2 \left(1+ \dfrac{(Tv)^2}{|\nabla_{\varepsilon}v|_{\varepsilon}^2}\right)^{2}\right)\\
				&\leq 2 \, C_{p}^{2} \, (p-1) \, |\nabla_{\varepsilon}v|_{\varepsilon}^{p} + 8 (p-1)^{2} |\nabla_{\varepsilon}v|_{\varepsilon}^p\\
				&+ \dfrac{2-p}{2}(p-1)^{2}\, |\nabla_{\varepsilon}v|_{\varepsilon}^{p-4} \left(|\nabla_{\varepsilon}|\nabla_{\varepsilon}v|_{\varepsilon}^2|^2_{\varepsilon} - \left\langle \dfrac{\nabla_{\varepsilon}v}{|\nabla_{\varepsilon}v|_{\varepsilon}}, \nabla_{\varepsilon}|\nabla_{\varepsilon}v|_{\varepsilon}^2 \right\rangle_{\varepsilon}^2\right)\\
				&+  \dfrac{3}{4}(p-1)\, \left( \cancel{2(p-1)^{2}|\nabla_{\varepsilon}v|_{\varepsilon}^{p-2}\,|\nabla_{\varepsilon}Tv|_{\varepsilon}^2} + 8(p-1)^{2}|\nabla_{\varepsilon}v|_{\varepsilon}^{p}\,  \left(\dfrac{T_v}{v}\right)^2 \left(1+ \dfrac{(Tv)^2}{|\nabla_{\varepsilon}v|_{\varepsilon}^2}\right)^{2}\right).
			\end{aligned}
		\end{equation}
		Now, as a particular instance of \eqref{eq:C_p2_dal_basso}, we have that $C_p^2 > 10 M_{\ast}^{2}$. This in turn implies that 
		\begin{equation}\label{eq:gradesp_maggiore_3T}
			|\nabla_{\varepsilon}v(x_p)|_{\varepsilon} > 3 |Tv(x_p)|.
		\end{equation}
		Indeed, if this were not the case, for $1<p \leq 2$ we would find that
		\begin{equation*}
			C_{p}^{2} \leq 10 (p-1)^2 \dfrac{|Tv(x_p)|^2}{(v(x_p))^2} \leq 10 M_{\ast}^{2},
		\end{equation*}
		\noindent which we showed to be not possible.
		From Lemma \ref{lem:conti_in_xp}-vi) and \eqref{eq:gradesp_maggiore_3T}, we get
		\begin{equation}\label{eq:Stima_dal_Lemma}
			\begin{aligned}
				\dfrac{2-p}{2}(p-1)^{2} & \, |\nabla_{\varepsilon}v|_{\varepsilon}^{p-4} \left(|\nabla_{\varepsilon}|\nabla_{\varepsilon}v|_{\varepsilon}^2|^2_{\varepsilon} - \left\langle \dfrac{\nabla_{\varepsilon}v}{|\nabla_{\varepsilon}v|_{\varepsilon}}, \nabla_{\varepsilon}|\nabla_{\varepsilon}v|_{\varepsilon}^2 \right\rangle_{\varepsilon}^2\right)\\
				&\leq \dfrac{2}{9}(2-p)(p-1)^{2} |\nabla_{\varepsilon}v|_{\varepsilon}^{p-2} \, \left(|\nabla_{\varepsilon}T v|_{\varepsilon}^2 - \left\langle \dfrac{\nabla_{\varepsilon}v}{|\nabla_{\varepsilon}v|_{\varepsilon}}, \nabla_{\varepsilon}T v \right\rangle^{2}_{\varepsilon}\right)\\
				&\leq (2-p)(p-1)^{2} |\nabla_{\varepsilon}v|_{\varepsilon}^{p-2} \, \left(|\nabla_{\varepsilon}T v|_{\varepsilon}^2 - \left\langle \dfrac{\nabla_{\varepsilon}v}{|\nabla_{\varepsilon}v|_{\varepsilon}}, \nabla_{\varepsilon}T v \right\rangle^{2}_{\varepsilon}\right)
			\end{aligned}
		\end{equation}
		Combining \eqref{eq:Rearranging} with \eqref{eq:gradesp_maggiore_3T} and \eqref{eq:Stima_dal_Lemma}, we get
		\begin{equation}\label{eq:penultimo_conto}
			\begin{aligned}
				4\, &C_{p}^{2}\, \left(\cancel{\dfrac{p-1}{2}} + \dfrac{(p-1)^2}{4n}\right)\,  |\nabla_{\varepsilon}v|_{\varepsilon}^{p} \\
				&\leq \cancel{2C_p^2 (p-1)|\nabla_{\varepsilon}v|_{\varepsilon}^{p}} + 8(p-1)^2 \, |\nabla_{\varepsilon}v|_{\varepsilon}^{p}
				+ 6(p-1)^3 \, |\nabla_{\varepsilon}v|_{\varepsilon}^{p} M_{\ast}^{2} \left(1+\dfrac{1}{9}\right)^2.
			\end{aligned}
		\end{equation}
		Dividing on both sides of \eqref{eq:penultimo_conto} by $(p-1)^2|\nabla_{\varepsilon}v|_{\varepsilon}^{p}\neq 0$, and multiplying by $n$, we finally get for $1<p\leq 2$
		\begin{equation*}
			C_{p}^{2} \leq 8n + 6n(p-1)M_{\ast}^2 \, \dfrac{100}{81} < 8n + 8n M_{\ast}^2,
		\end{equation*}
		\noindent which yields a contradiction with \eqref{eq:C_p2_dal_basso}. This closes the proof.
	\end{proof}
\end{theorem}


\section{Proof of the main results}

We start this final section by showing that the key smallness condition $|\nabla_\varepsilon v^\varepsilon_p |_\varepsilon=o(v_p^\varepsilon)$ as $\| g\| \to \infty$ for the gradient at infinity used in the previous section is indeed satisfied. Note that, both in the setting of the Euclidean space treated in \cite{Moser} and in the case of sub-Riemannian Heisenberg group $\He^1$ as attempted in \cite{CuiZhao}, such property for the $p$-capacitary potentials follows from a scale-invariant Harnack inequality for $p-$harmonic functions and $L^p-L^\infty$ estimates for their gradient obtained in \cite{Zhong}. On the other hand in the Riemannian Heisenberg group no such scale invariance is available, therefore in the present case we rely on the two uniform properties for $p$-harmonic functions established in Section 5.

\begin{proposition}
\label{domain-convergence}
Let  $1<p<Q$, $\varepsilon \in (0,1]$, and let $\Omega \subset \He^n$ be an open set with bounded complement and smooth boundary $\partial \Omega$. Let $v=v^\varepsilon_p$ be the minimizer of \eqref{p-dirichlet} in the set $\dot{W}
^{1,p}_{1, \varepsilon} (\Omega)$. 
Then $v\in C^1(\overline{\Omega})$ and  there exists a constant $C>0$ depending only on $p$ and $Q$ such that for every $\varepsilon \in (0,1]$ the gradient bound \eqref{capogna-citti-bound} holds. Moreover,  $|\nabla_\varepsilon  v_p^\varepsilon|_\varepsilon=o(v_p^\varepsilon)$ as $\| g\| \to \infty$.
\end{proposition}

\begin{proof}

As already recalled, the $C^1$-regularity property up to the boundary is a consequence of the interior and boundary regularity theory for $p$-harmonic functions from \cite{Dibenedetto} and \cite{Lieberman}. Then, the global $L^\infty$-bound for $\nabla_\varepsilon v$ follows directly from Proposition \ref{uniform-capogna-citti} applied to arbitrary balls $B_1^\varepsilon(\bar{g})\subseteq \Omega$. 

In order to justify the asymptotic decay at infinity we adapt the original proof from \cite{Moser}, relying on Proposition \ref{uniform-capogna-citti}, Proposition \ref{garofalo} and a simple scaling argument. Thus, setting fo brevity $v=v^\varepsilon_p$ we have $v \in C^1(\Omega)$ by elliptic regularity and for any sequence $\{g_n\} \subset \Omega$ such that $\|g_n\| \to \infty$ we are going to show that $|\nabla_\varepsilon v (g_n) |_\varepsilon =o(v(g_n))$ as $n \to \infty$. To do this, we first chose a sequence $\{r_n\} \subset (1,\infty)$ such that $r_n \to +\infty$ and $r_n=o(\|g_n\|)$ as $n \to \infty$ and then we consider rescaled functions 
$w_n(g)=v(L_{g_n}\circ \delta_{r_n}(g))$. 
Setting $\tilde{\varepsilon}_n=\varepsilon r_n^{-1}<1$, by construction for $n$ large enough we have $w_n\in  C^1 (B^{\tilde{\varepsilon}_n}_{36})$, where we have also the identity
\begin{equation}
\label{equivariance} 
|\nabla_{\tilde{\varepsilon}_n} w_n|^2_{\tilde{\varepsilon}_n}= r^2_n |\nabla_\varepsilon v|^2_\varepsilon \circ (L_{g_n} \circ\delta_{r_n})\, , \end{equation}
 Note that in particular the function $w_n$ is weakly $p$-harmonic in $B^{\tilde{\varepsilon}_n}_{36 }$ w.r. to the rescaled metric $|\cdot|^2_{\tilde{\varepsilon}_n}$ because $v$ is weakly $p$-harmonic in $B^\varepsilon_{ 36  r_n}(g_n)$ w.r.to the metric $|\cdot|^2_\varepsilon$ for $n$ large enough.   
 
Now, applying Proposition \ref{garofalo} to $\{w^{\tilde{\varepsilon}_n}\} \subseteq  C^1(B_2^{\tilde{\varepsilon}_n})$ it follows from the gradient bound \eqref{capogna-citti-bound}  that
\begin{equation}
	\label{claim2wn}
	|\nabla_{\tilde{\varepsilon}_n} w_n (0)|^p_{\tilde{\varepsilon}_n} \leq C    \int_{ B_{1}^{\tilde{\varepsilon}_n}(0)} |\nabla_{\tilde{\varepsilon}_n} w_n|_{\tilde{\varepsilon}_n}^p \, ,
\end{equation} 
for a constant $C>0$ independent of $n$. 

Next, we let $\psi_{\tilde{\varepsilon}}(g)=\min \{ 1, \max \{ -1+d_{\tilde{\varepsilon}}(g,0) , 0 \} \}$, 
so that $0 \leq \psi_{\tilde{\varepsilon}} \leq 1$ on $\He^n$, $\psi_{\tilde{\varepsilon}}\equiv 1$ on 
$B_1^{ \tilde{\varepsilon} }$, 
$spt \,  \psi_{\tilde{\varepsilon}}  \subset  \overline{B_2^{\tilde{\varepsilon}}}$ and, being 1-Lipschitz w.r. to $d_{\tilde{\varepsilon}}$, we have $|\nabla_{\tilde{\varepsilon} }  \psi_{\tilde{\varepsilon}} |_{\tilde{\varepsilon}} \leq 1$. Testing the equation for $w_n$ with $\psi_n^p w_n$, where $\psi_n=\psi_{\tilde{\varepsilon}_n}$, and using Young's inequality we have 
 \[ \int_{\He^n}| \nabla_{\tilde{\varepsilon}_n} w_n|^p_{\tilde{\varepsilon}_n} \psi_n^p \leq p \int_{\He^n}  | \nabla_{\tilde{\varepsilon}_n} w_n|^{p-1}_{\tilde{\varepsilon}_n} \psi_n^{p-1} w_n |\nabla_{\tilde{\varepsilon}_n }  \psi_{\tilde{\varepsilon}_n} |_{\tilde{\varepsilon}_n} \leq \frac{p-1}p  \int_{\He^n}| \nabla_{\tilde{\varepsilon}_n} w_n|^p_{\tilde{\varepsilon}_n} \psi_n^p +p^{p-1}  \int_{\He^n}| \nabla_{\tilde{\varepsilon}_n} \psi_n|^p_{\tilde{\varepsilon}_n} w_n^p \, ,\]
whence we have the following Caccioppoli inequality with uniform constant, 
\begin{equation}
\label{caccioppoli}
	 \int_{ B_1^{\tilde{\varepsilon}_n}} | \nabla_{\tilde{\varepsilon}_n} w_n|^p_{\tilde{\varepsilon}_n} \leq p^p   \int_{ B_2^{\tilde{\varepsilon}_n}}  w_n^p \, .
\end{equation}
Finally, combining \eqref{equivariance}, \eqref{claim2wn}, and \eqref{caccioppoli} with the Harnack inequality \eqref{harnack} in Proposition \ref{garofalo} (with the choice $2r=4=\bar{r}$) we have
\[ r^p_n |\nabla_\varepsilon v(g_n)|^p_\varepsilon \leq C    \int_{ B_1^{\tilde{\varepsilon}_n}(0)} |\nabla_{\tilde{\varepsilon}_n} w_n|_{\tilde{\varepsilon}_n}^p  \leq C  \int_{ B_2^{\tilde{\varepsilon}_n}}  w_n^p \leq  C \, \, | B_2^{\tilde{\varepsilon}_n}  |\, \,\underset{B_2^{\tilde{\varepsilon}_n}}{\rm{ess} \inf} \,\,w^p_n\leq C (w_n(0))^p=C v(g_n)^p \, ,\]
for constant $C>0$ depending on $p$ but independent of $n$ (compare also Lemma \ref{lemma:eps-gauge}-$(2)$).  Since $r_n \to \infty$ as $n \to \infty$ then the conclusion follows.
\end{proof}

In order to prove Theorem \ref{Thm:differential-harnack} the last technical ingredient is following auxiliary result. 
\begin{proposition}
\label{eps-limit}
	Let  $1<p<Q$, $\varepsilon \in [0,1]$, and let $\Omega \subset \He^n$ be an open set with bounded complement and smooth boundary $\partial \Omega$. If $v^\varepsilon_p$ are the minimizers of \eqref{p-dirichlet} in the set $\dot{W}
^{1,p}_{1, \varepsilon} (\Omega)$ then $v^\varepsilon_p \to v^0_p$ strongly in $\dot{W}
^{1,p}_{1, 0} (\Omega)$ and, up to subsequences, $|\nabla_\varepsilon v^\varepsilon_p|_\varepsilon \to |\nabla_0 v^0_p|_0$ a.e. as $\varepsilon \to 0$.
\end{proposition}

\begin{proof}
First we fix $R>0$ so large that $\He^n \setminus \Omega \subset B_R(0)$. 
In view of Lemma \ref{propertiesPhialpha}, we have $\tilde{\Phi}^{(0,R)}_{\alpha}\in \dot{W}^{1,p}_{1,\varepsilon}(\Omega)$ with finite norms for each $\varepsilon\in (0,1]$, and therefore for each $\varepsilon \in (0,1]$ the function $\tilde{\Phi}^{(0,R)}_{\alpha}$ is admissible competitor for each $v^\varepsilon_p$. Hence, it follows from energy minimality that
$$E^{0}_{p}\left(v^{\varepsilon}_{p}\right)\leq E^{\varepsilon}_{p}\left(v^{\varepsilon}_{p}\right)\leq E^{\varepsilon}_{p}\left(\tilde{\Phi}^{(0,R)}_{\alpha}\right)\leq E^{1}_{p}\left(\tilde{\Phi}^{(0,R)}_{\alpha}\right)<\infty$$
 for each $\varepsilon \in (0,1]$, i.e., the sequence $\{v^\varepsilon_p\}_{\varepsilon} \subset \dot{W}^{1,p}_{1,0}(\Omega)$ is bounded. 
 
 Next, for any sequence $\varepsilon_{j}\to 0$ as $j\to+\infty$, possibly passing to a subsequence, there exists $v_*\in H\dot{W}^{1,p}(\mathbb{H}^{n})$ such that $v^{\varepsilon_{j}}_{p}$ converges to $v_*$ weakly in $H\dot{W}^{1,p}(\mathbb{H}^{n})$, strongly in $L^{1}_{\mathrm{loc}}$ and a.e. in $\mathbb{H}^{n}$, so that in particular we deduce that $v_* \in \dot{W}^{1,p}_{1,0}(\Omega)$.
It is now enough to prove the convergence in norm together with the energy minimality of $v_*$ on $\dot{W}^{1,p}_{1,0}(\Omega)$ , whence $v_*=v_p^\varepsilon$ by strict convexity of $E^0_p(\cdot)$ and in turn the full convergence holds as $\varepsilon \downarrow 0$. 

 Note that for every $\hat{v}\in \dot{W}^{1,p}_{1,0}(\Omega)$, since $\Omega$ is smooth with bounded complement and $\hat{v}- \tilde{\Phi}^{(0,R)}_{\alpha} =0$ a.e. on $\He^n \setminus \Omega$, there exists a sequence $\varphi_{j} \subset C^{\infty}_{0}(\Omega)$ such that $\varphi_{j} \to \hat{v} - \overline{\Phi}^{(0,R_0)}_{\alpha}$ in $H\dot{W}^{1,p}(\He^n)$ as $j \to +\infty$. Since $ \varphi_{m}+\tilde{\Phi}^{(0,R)}_{\alpha}  \in \dot{W}^{1,p}_{1,0}(\Omega) $, the weak lower semicontinuity of $E^{0}_{p}(\cdot )$, the monotonicity of $ \varepsilon \to E^\varepsilon_p(\cdot)$ and the energy minimality of each $v_p^\varepsilon$ yield for each $m \geq 1$ and in turn as $m \to \infty$ that 
\begin{equation}\label{eq:Limite_Eug}
\begin{aligned}
&E^{0}_{p}(v_*)  \leq \liminf_{j\to +\infty} E_p^{0} (v_p^{\varepsilon_j}) \leq \limsup_{j\to +\infty} E_p^{0}(v_p^{\varepsilon_j}) \leq \limsup_{j\to +\infty} E_p^{\varepsilon_j}(v_p^{\varepsilon_j})\\
\leq & \limsup_{m \to \infty} \limsup_{j\to +\infty} E_p^{\varepsilon_j}\left(\varphi_{m}+\tilde{\Phi}^{(0,R)}_{\alpha}\right) = \lim_{m \to \infty} E_p^0 \left( \varphi_{m}+\tilde{\Phi}^{(0,R)}_{\alpha}\right) =  E^{0}_{p}(\hat{v}).
\end{aligned}
\end{equation}
Since $\hat{v}\in \dot{W}_{1,0}^{1,p}(\Omega)$ is arbitrary, the latter shows that $v_*$ is actually the minimizer and hence $v_*=v_p^0$. Moreover, choosing $\hat{v}=v^{0}_{p}$ in \eqref{eq:Limite_Eug} we infer that $T_\varepsilon v_p^\varepsilon \to 0$ in $L^p(\He^n)$ as $\varepsilon \to 0$ and $E^0_p(v_p^{\varepsilon}) \to E^0_p(v^0_p)$ as $\varepsilon \to 0$, from which the desired strong convergence follows. Finally, once $\nabla_0 v^\varepsilon_p \to \nabla_0 v^0_p$ in $L^p$, passing to a subsequence we also get $|\nabla_\varepsilon v^\varepsilon_p|_\varepsilon \to |\nabla_0 v^0_p|_0$ a.e. as $\varepsilon \to 0$. 
\end{proof}

Now we are ready to prove the first result of the Introduction.

\begin{proof}[Proof of Theorem \ref{Thm:differential-harnack}]
Let $1<p\leq 2$ be fixed and for $\varepsilon \in(0,1]$ let us consider $v^\varepsilon_p \in \dot{W}^{1,p}_{1,\varepsilon}(\Omega)$ the minimizer of \eqref{p-dirichlet} in the set \eqref{convex-set} which is the unique finite energy solution to \eqref{EqForVpeps}. As already recalled in the introduction, since $\Omega$ has $C^2$-smooth boundary by the interior and boundary regularity theory for $p$-harmonic function we have $v^\varepsilon_p \in C^{1}(\overline{\Omega})$ (the gradient $\nabla_\varepsilon v^\varepsilon_p$ being actually H\"{o}lder continuous up to the boundary) and it is strictly positive in $\overline{\Omega}$ by the strong maximum principle. 

Next, we are going to prove that for $\varepsilon \in (0,1]$ when the domain $\Omega$ satisfies the uniform gauge-ball condition $(\bf{HP_\Omega})$ with parameter $R_0>0$ each solution $v^\varepsilon_p$ satisfies \eqref{diffharnack} for a constant $C>0$ depending on $R_0$ and $Q$. To do this, first notice that the inequality holds on $\partial \Omega$ in view of Proposition \ref{p-harm-bd-bound} with some constant $\widetilde{C}>0$ depending only on $R_0$ and $Q$.  Moreover, in view of Proposition \ref{domain-convergence} we have $|\nabla_\varepsilon  v^\varepsilon_p|_\varepsilon=o(v^\varepsilon_p)$ as $\| g\| \to \infty$, therefore we can apply Theorem \ref{bd-interior-bound} and the inequality \eqref{diffharnack} follows just by taking $C=\sqrt{ \overline{M}^{2} + 5QM_{\ast}^2 + 4Q }$, with $\overline{M}$ and $M_*$ as defined in \eqref{eq:def_M_bar} and \eqref{eq:def_M_ast} respectively.

Finally, if $\varepsilon=0$ and $v^0_p \in \dot{W}^{1,p}_{1,0}(\Omega)$ is the unique finite energy solution to 
\eqref{EqForVpeps} corresponding to the minimizer of \eqref{p-dirichlet} then the interior regularity theory for $p$-harmonic function in the sub-Riemannian Heisenberg group proved in \cite{Zhong} yields $v^0_p \in C(\Omega)$ and $\nabla_0 v^0_p\in C(\Omega)$. Moreover $v^\varepsilon_p \to v^0_p$ as $\varepsilon \to 0$ strongly in $\dot{W}^{1,p}_{1,0}(\Omega)$ and $|\nabla_\varepsilon  v^\varepsilon_p|_\varepsilon \to |\nabla_0  v^\varepsilon_p|_0$ a.e. in $\Omega$ because of Proposition \ref{eps-limit}. Since $d_\varepsilon\leq d_0$ for any $\varepsilon>0$, the functions $\{ v^\varepsilon_p\}_{\varepsilon \in (0,1]}$ are $d_0$-equiLipschitz in $\He^n$ due to the bound \eqref{diffharnack} and the fact that $0 < v^\varepsilon_p \leq 1$ on $\overline{\Omega}$. As a consequence $v^\varepsilon_p \to v^0_p$ locally uniformly on $\He^n$ by the Ascoli-Arzel\`a Theorem, hence $v^0_p \in C(\overline{\Omega})$, and \eqref{diffharnack} follows for $v^0_p$ from the same inequality for each $v^\varepsilon_p$ as $\varepsilon \to 0$.  
\end{proof}

In order to prove the main result of the paper, i.e., Theorem \ref{Thm:existence}, we need three auxiliary results which are actually still valid in the range $1<p<Q$. The first one relates weak solutions of \eqref{EqForVpeps} and \eqref{EqForUpeps}.

\begin{lemma}
\label{lem:v_sol_then_u_sol}
Let $1<p<Q$, $\varepsilon \in [0,1]$ and $v^{\varepsilon}_{p} \in W^{1,p}_\varepsilon (\Omega)$ nonnegative and nonconstant. If  $v^{\varepsilon}_{p}$ is a weak solution of \eqref{EqForVpeps}, i.e., for test functions $\psi \in W^{1,p}_\varepsilon (\Omega)$ with compact support in $\Omega$, then $v^{\varepsilon}_{p}$ is continuous in $\Omega$, it is  strictly positive and $\nabla_\varepsilon v^\varepsilon_p $ is continuous in $\Omega$. As a consequence, if $v^\varepsilon_p \not \equiv 0$ then the function $u^{\varepsilon}_{p}:=(1-p)\log(v^{\varepsilon}_{p})$ is well-defined, it is $d_\varepsilon$-locally Lipschitz and it is a weak solution of \eqref{EqForUpeps}, i.e., for test functions $\varphi \in W^{1,p}_\varepsilon (\Omega)\cap L^\infty (\Omega)$ with compact support in $\Omega$.
\end{lemma}
\begin{proof}
First, to simplify the notation we just write throughout the proof  $v=v^{\varepsilon}_{p}$ and $u=u^{\varepsilon}_{p}$. Note that by regularity theory recalled in the introduction both $v$ and $\nabla_\varepsilon v$ are continuous in $\Omega$, moreover $v>0$ in $\overline{\Omega}$ by the strong maximum principle.

 Under this assumption we take $\varphi \in W^{1,p}_\varepsilon (\Omega)\cap L^\infty(\Omega)$ with compact support and we consider $\psi=\dfrac{\varphi}{v^{p-1}}$. Since $v>0$ in $\Omega$, $\psi \in W^{1,p}_\varepsilon (\Omega)$ and it is an admissible test function for \eqref{EqForVpeps}, so that
\begin{equation}\label{Eq_p_armonica_testata}
0=\int_{\Omega}|\nabla_\varepsilon v|_\varepsilon^{p-2}\langle \nabla_\varepsilon v, \nabla_\varepsilon \left(\dfrac{\varphi}{v^{p-1}}\right)\rangle_\varepsilon = \int_{\Omega}\dfrac{|\nabla_\varepsilon v|_\varepsilon^{p-2}}{v^{p-1}}\langle \nabla_\varepsilon v,\nabla_\varepsilon \varphi \rangle_\varepsilon - (p-1)\int_{\Omega}\dfrac{|\nabla_\varepsilon v|_\varepsilon^{p}}{v^{p}}\varphi \, .
\end{equation}
Clearly $u=(1-p)\log v$ is $d_\varepsilon$-locally Lipschitz by construction, because $\nabla_\varepsilon u = (1-p)\tfrac{\nabla_\varepsilon v}{v}$, and $u$ being a solution of \eqref{EqForUpeps} amounts to
\begin{equation*}
-\int_{\Omega}|\nabla_\varepsilon u|^{p-2}\langle \nabla_\varepsilon u, \nabla_\varepsilon \varphi\rangle_\varepsilon = \int_{\Omega}|\nabla u_\varepsilon |_\varepsilon^{p}\varphi \, , \end{equation*}
for every test function $\varphi \in W^{1,p}_\varepsilon (\Omega)\cap L^\infty(\Omega)$ with compact support in $\Omega$. Finally, since the latter equation can be inferred from \eqref{Eq_p_armonica_testata} up to multiply it by $(p-1)^{p-1}$ and $\varphi$ is arbitrary the proof is complete.
\end{proof}

The next lemma describes solutions to \eqref{EqForUpeps} in terms of the minimality property for the associated energy functional \eqref{HIpeps-energy}.

\begin{lemma}
\label{equa-vs-min}
Let $1<p<Q$, $\varepsilon \in [0,1]$ and $\Omega \subset \He^n$ an open set. For a function $ u^\varepsilon_p \in W^{1,p}_{\varepsilon, loc} (\Omega) \cap L^\infty_{loc}(\Omega)$ the following are equivalent
\begin{itemize}
\item[(i)] the function $u^{\varepsilon}_{p}$ is a weak solution of \eqref{EqForUpeps} with respect to test functions in $W^{1,p}_\varepsilon (\Omega)\cap L^\infty(\Omega)$ with compact support in $\Omega$;
\item[(ii)] for every compact set $K \subset \Omega$ the function $u^{\varepsilon}_{p}$ satisfies $
J_{u_p^{\varepsilon}}^{p,\varepsilon}(u_p^{\varepsilon};K) \leq J_{u_p^{\varepsilon}}^{p,\varepsilon}(w;K)$ for every competitor $w \in W^{1,p}_{\varepsilon, loc} (\Omega) \cap L^\infty_{loc}(\Omega)$ such that $w= u_{p}^{\varepsilon}$ a.e. in $\Omega \setminus K$.

\end{itemize}
\begin{proof}
As for the previous Lemma, we simply write $u=u^\varepsilon_p$ wherever it appear.
We start showing that $(i) \Rightarrow (ii)$ arguing as in \cite{Moser,CuiZhao}. 
Since $u$ is a weak solution of \eqref{EqForUpeps}, for every compact set $K \subset \Omega$ it holds that
\begin{equation}\label{Eq_Formulazione_var_di_eq_di_u}
\int_{K}|\nabla_\varepsilon u|_\varepsilon^{p-2}\langle\nabla_\varepsilon u, \nabla_\varepsilon \varphi \rangle_\varepsilon = - \int_{K}|\nabla_\varepsilon u|_\varepsilon^{p}\varphi 
\end{equation}
\noindent for every $\varphi \in W^{1,p}_\varepsilon(\Omega)\cap L^\infty(\Omega)$ such that $\varphi=0$ a.e. in $\Omega \setminus K$.
Now, for $w\in W^{1,p}_{\varepsilon,loc}(\Omega)  \cap L^\infty_{loc}(\Omega) $ such that $w= u_{p}^{\varepsilon}$ a.e. in $\Omega \setminus K$ and $\varphi = w-u$ equation \eqref{Eq_Formulazione_var_di_eq_di_u} yields
\begin{equation*}
\int_{K}|\nabla_\varepsilon u|_\varepsilon^{p}(u-w) = \int_{K}|\nabla_\varepsilon u|_\varepsilon^{p-2}\langle \nabla_\varepsilon u, \nabla_\varepsilon w\rangle_\varepsilon - \int_{K}|\nabla_\varepsilon u|_\varepsilon^{p}.
\end{equation*}
We now estimate from above the first term on the right hand side, first by using Cauchy-Schwarz inequality and then Young inequality (with conjugate exponents $\tfrac{p}{p-1}$ and $p$), getting
\begin{equation}
\int_{K}|\nabla_\varepsilon u|_\varepsilon^{p}(u-w) \leq \left(1-\frac1p\right) \int_{K}|\nabla_\varepsilon u|_\varepsilon^{p} +\int_{K}\dfrac{|\nabla_\varepsilon w|_\varepsilon^{p}}{p} - \int_{K}|\nabla_\varepsilon u|_\varepsilon^{p} = \int_{K}\dfrac{|\nabla_\varepsilon w|_\varepsilon^{p}}{p} - \int_{K}\dfrac{|\nabla_\varepsilon u|_\varepsilon^{p}}{p}. 
\end{equation}
Finally, a simple rearrangement of the latter gives the desired conclusion.

Let us now prove that $(ii)\Rightarrow (i)$. For every compact set $K \subset \Omega$, we can fix any $\varphi \in W^{1,p}_\varepsilon(\Omega)\cap L^\infty(\Omega)$ such that $\varphi=0$ a.e. in $\Omega \setminus K$ and for any $t \in \mathbb{R}$ the function $u+t\varphi$ is now an admissible competitor for the minimality property of the functional $J^{p,\varepsilon}_{u}(\cdot,K)$. Therefore, differentiating under integral sign gives
\begin{equation*}
0 = \dfrac{d}{dt}\Big|_{t=0}J^{p,\varepsilon}_u(u+t\varphi,K) = \int_{K}|\nabla_\varepsilon u|_\varepsilon^{p-2}\langle \nabla_\varepsilon u, \nabla_\varepsilon \varphi\rangle_\varepsilon +\int_{K}|\nabla_\varepsilon u|_\varepsilon^{p}\varphi,
\end{equation*}
\noindent and this concludes the proof.
\end{proof}
\end{lemma}

The last auxiliary result describes how the minimality property for the energy functional \eqref{HIpeps-energy} stated in the previous lemma passes to the limit. 

\begin{proposition}
\label{limiting-minimality}
	Let  $1<p<Q$, $\varepsilon \in [0,1]$, and $\Omega \subset \He^n$ be an open set with bounded complement and smooth boundary $\partial \Omega$. For $v^\varepsilon_p:\He^n \to (0,1]$ the minimizer of \eqref{p-dirichlet} in the set $\dot{W}
^{1,p}_{1, \varepsilon} (\Omega)$ let $u^\varepsilon_p: \He^n \to [0,\infty)$ the corresponding functions given by \eqref{substitution}, so that 
they minimize the functionals $J^{p,\varepsilon}_{u^\varepsilon_p}$ defined in \eqref{HIpeps-energy}
 in the sense that for every compact set $K \subset \Omega$ it holds that
\begin{equation}\label{eq:u_pMinimizer}
J_{u_p^{\varepsilon}}^{p,\varepsilon}(u_p^{\varepsilon};K) \leq J_{u_p^{\varepsilon}}^{p,\varepsilon}(w;K),
\end{equation}
\noindent for every locally $d_{\varepsilon}$-Lipschitz function $w:\Omega \to \mathbb{R}$ such that $w\equiv u_{p}^{\varepsilon}$ in $\Omega \setminus K$.

Assume that $\{ u^\varepsilon_p\}$ satisfy the uniform bound \eqref{globalgradbound}, so that, up to subsequences, $u_{p}^{\varepsilon} \to u$ as $(\varepsilon,p) \to (0,1)$ locally uniformly in $\He^n$, where $u$ is a $d_0$-Lipschitz function. Then $u:\He^n \to [0,\infty)$ is a weak solution of \eqref{EulerForU}, i.e., for every compact set $K\subset \Omega$, it holds
\begin{equation}\label{eq:u_solution}
J^{0}_{u}(u;K) \leq J^{0}_{u}(w;K),
\end{equation}
\noindent for every locally $d_{0}$-Lipschitz function $w:\Omega \to \mathbb{R}$ such that $w \equiv u$ in $\Omega \setminus K$. Analogous conclusion holds for $\varepsilon \equiv const$ (resp., for $p \equiv 1$) as $p\to 1$ (resp., as $\varepsilon \to 0$).
\end{proposition}

\begin{proof} We give the proof only in the case   $(\varepsilon,p) \to (0,1)$, as the two others are similar and even simpler, adapting the argument in \cite{Moser} for the Euclidean case (see also \cite{CuiZhao} for the case $\He^1$ when $\varepsilon=0$).

  We consider a nonnegative test function $\eta \in C^{\infty}_{0}(\Omega)$ such that $\rm{supp}(\eta)\subset\subset \Omega$ satisfying the normalization condition
$0\leq \eta\leq 1$ in $\Omega$, together with a bounded open set $V \subset \subset \Omega$ such that $\rm{supp}(\eta)\subset\subset V$. Then, for a compact set $K \subset \Omega$  we fix a locally $d_{0}$-Lipschitz function $w:\Omega \to \mathbb{R}$ such that $w \equiv u$ in $\Omega \setminus K$.

 
 Note that, according to \cite[Theorem 2.15]{ComiMagnani}, there exists a sequence of smooth functions $(w_{k})_{k\in\mathbb{N}} \subset C^{\infty}(V)$ such that 
\begin{itemize}
\item[i)] $w_k \to w$ uniformly on compact subsets of $V$;
\item[ii)] $\|\nabla_{0}w_k\|_{L^{\infty}(W)}$ is bounded for each $W\subset\subset V$ and $k$ sufficiently large;
\item[iii)] $\nabla_{0}w_k \to \nabla_{0}w$ a.e. in $V$.
\end{itemize}

Thus, for each fixed $k$ we use the function $\eta w_{k} + (1-\eta) u_p^{\varepsilon}$ as test function in \eqref{eq:u_pMinimizer} for the functional $J^{p,\varepsilon}_{u^{\varepsilon}_{p}}(\cdot;\rm{supp}(\eta))$ and we find that
\begin{equation*}
\begin{aligned}
\int_{\rm{supp}(\eta)}&\left(\dfrac{1}{p}|\nabla_{\varepsilon}u_{p}^{\varepsilon}|_{\varepsilon}^p + u_{p}^{\varepsilon} |\nabla_{\varepsilon}u_{p}^{\varepsilon}|_{\varepsilon}^p \right)\\
&\leq \int_{\rm{supp}(\eta)}\left(\dfrac{1}{p}|\nabla_{\varepsilon}(\eta w_k + (1-\eta)u_{p}^{\varepsilon})|_{\varepsilon}^p + (\eta w_k + (1-\eta)u_{p}^{\varepsilon})|\nabla_{\varepsilon}u_{p}^{\varepsilon}|_{\varepsilon}^p \right),
\end{aligned}
\end{equation*}
\noindent which readily yields
\begin{equation}\label{eq:Moser9}
\begin{aligned}
\int_{\rm{supp}(\eta)}&\left( \dfrac{1}{p}|\nabla_{\varepsilon}u_{p}^{\varepsilon}|_{\varepsilon}^{p} + \eta (u_{p}^{\varepsilon}-w_k) |\nabla_{\varepsilon}u_{p}^{\varepsilon}|_{\varepsilon}^{p}\right)\\
&\leq \dfrac{1}{p}\int_{\rm{supp}(\eta)}\left| \eta \nabla_{\varepsilon}w_k + (1-\eta)\nabla_{\varepsilon}u_p^{\varepsilon} + (w_k-u_{p}^{\varepsilon})\nabla_{\varepsilon}\eta\right|_{\varepsilon}^{p}\\
&\leq \dfrac{3^{p-1}}{p}\int_{\rm{supp}(\eta)}\left( \eta^p |\nabla_{\varepsilon}w_k|_{\varepsilon}^{p} + (1-\eta)^p |\nabla_{\varepsilon}u_{p}^{\varepsilon}|_{\varepsilon}^{p} + |w_k-u_{p}^{\varepsilon}|^p |\nabla_{\varepsilon}\eta|_{\varepsilon}^{p}\right).
\end{aligned}
\end{equation}
Since $0\leq 1-\eta \leq 1$ we can rewrite \eqref{eq:Moser9} as 
\begin{equation*}
\begin{aligned}
\int_{\rm{supp}(\eta)}&\left( \dfrac{\eta +(1-\eta)}{p}|\nabla_{\varepsilon}u_{p}^{\varepsilon}|_{\varepsilon}^{p} + \eta (u_{p}^{\varepsilon}-w_k) |\nabla_{\varepsilon}u_{p}^{\varepsilon}|_{\varepsilon}^{p}\right)\\
&\leq \dfrac{3^{p-1}}{p}\int_{\rm{supp}(\eta)}\left( \eta^p |\nabla_{\varepsilon}w_k|_{\varepsilon}^{p} + (1-\eta) |\nabla_{\varepsilon}u_{p}^{\varepsilon}|_{\varepsilon}^{p} + |w_k-u_{p}^{\varepsilon}|^p |\nabla_{\varepsilon}\eta|_{\varepsilon}^{p}\right),
\end{aligned}
\end{equation*}
\noindent which in turn is equivalent to
\begin{equation}\label{eq:Moser9Bis}
\begin{aligned}
\int_{\rm{supp}(\eta)} &\left( \dfrac{1}{p}\eta |\nabla_{\varepsilon}u_{p}^{\varepsilon}|_{\varepsilon}^{p} + \eta (u_{p}^{\varepsilon}-w_k) |\nabla_{\varepsilon}u_{p}^{\varepsilon}|_{\varepsilon}^{p}\right)\\
&\leq \dfrac{3^{p-1}}{p}\int_{\rm{supp}(\eta)}\left( \eta^p |\nabla_{\varepsilon}w_k|_{\varepsilon}^{p} + |w_k-u_{p}^{\varepsilon}|^p |\nabla_{\varepsilon}\eta|_{\varepsilon}^{p}\right)\\
&\qquad + \dfrac{3^{p-1}-1}{p}\int_{\rm{supp}(\eta)} (1-\eta) |\nabla_{\varepsilon}u_{p}^{\varepsilon}|_{\varepsilon}^{p}.
\end{aligned}
\end{equation}
Now we estimate the right hand side of \eqref{eq:Moser9Bis}. Since $0 \leq \eta \leq 1$ and each $w_k$ is smooth on $supp (\eta)$ we obtain
\begin{equation}
\label{Moser-est1}
\dfrac{3^{p-1}}{p}\int_{\rm{supp}(\eta)} \eta^p |\nabla_{\varepsilon}w_k|_{\varepsilon}^{p} \leq \dfrac{3^{p-1}}{p}\int_{\rm{supp}(\eta)} \eta|\nabla_{\varepsilon}w_k|_{\varepsilon}^{p} \longrightarrow \int_{\rm{supp}(\eta)}\eta |\nabla_{0}w_k|_{0}, \quad \textrm{as } (\varepsilon, p) \to (0,1).
\end{equation}
Concerning the second term in the right hand side of \eqref{eq:Moser9Bis},  we have $ |w_k-u_{p}^{\varepsilon}|^p \to |w_k-u|$ and $|\nabla_\varepsilon \eta|_\varepsilon^p \to |\nabla_0 \eta|_0$ uniformly on $supp(\eta)$ as $(\varepsilon,p) \to (0,1)$, hence
\begin{equation}
\label{Moser-est2}
\dfrac{3^{p-1}}{p}\int_{\rm{supp}(\eta)} |w_k-u_{p}^{\varepsilon}|^p |\nabla_{\varepsilon}\eta|_{\varepsilon}^{p}  \longrightarrow \int_{\rm{supp}(\eta) }|w_k-u| |\nabla_{0}\eta|_{0} \quad \textrm{as } (\varepsilon, p) \to (0,1).\end{equation}
Finally, since $\tfrac{3^{p-1}-1}{p} \to 0$ as $p \to 1$ then for the last term of \eqref{eq:Moser9Bis} it follows from \eqref{globalgradbound} that
\begin{equation}
\label{Moser-est3}
\dfrac{3^{p-1}-1}{p}\int_{\rm{supp}(\eta)} (1-\eta) |\nabla_{\varepsilon}u_{p}^{\varepsilon}|_{\varepsilon}^{p} \to 0, \quad \textrm{as } (\varepsilon, p) \to (0,1),
\end{equation}

Concerning the left hand side of \eqref{eq:Moser9Bis}, choosing $w=u$ and combining \eqref{globalgradbound}, the uniform convergence $u_p^\varepsilon \to u$ and (i) we infer that
\[ \limsup_{k \to \infty }\limsup_{(\varepsilon, p) \to (0,1)} \int_{\rm{supp}(\eta)}\eta |u_{p}^{\varepsilon}-w_k| |\nabla_{\varepsilon}u_{p}^{\varepsilon}|_{\varepsilon}^{p}  \leq  \limsup_{k \to \infty }   \int_{\rm{supp}(\eta)} C \eta |u-w_k| =0 \, ,\]
whence \eqref{eq:Moser9Bis}-\eqref{Moser-est3} together with (ii) and (iii) and dominated convergence yield

\begin{equation}\label{moser-usc}
\begin{aligned}
 \limsup_{(\varepsilon, p) \to (0,1)} \int_{\rm{supp}(\eta)}  \dfrac{1}{p}\eta |\nabla_{\varepsilon}u_{p}^{\varepsilon}|_{\varepsilon}^{p} 
 & \leq \limsup_k    \limsup_{(\varepsilon, p) \to (0,1)} 
 \int_{\rm{supp}(\eta)} \left( \dfrac{1}{p}|\nabla_{\varepsilon}u_{p}^{\varepsilon}|_{\varepsilon}^{p} + \eta (u_{p}^{\varepsilon}-w_k) |\nabla_{\varepsilon}u_{p}^{\varepsilon}|_{\varepsilon}^{p}\right) \\
& \leq \limsup_{k \to \infty } \int_{\rm{supp}(\eta)}\eta |\nabla_{0}w_k|_{0} 
+|w_k-u| |\nabla_{0}\eta|_{0}=\int_{\rm{supp}(\eta)}\eta |\nabla_{0} u|_{0} \, .
\end{aligned}
\end{equation}

On the other hand, note that $\nabla_0 u_p^\varepsilon \to \nabla_0 u$ *-weakly in $L^\infty (\He^n)$ because \eqref{globalgradbound} holds uniformly and $u^\varepsilon_p \to u$ locally uniformly, whence weak lower semicontinuity and Young's inequality give
\[ \int_{\rm{supp}(\eta)}\eta |\nabla_{0} u|_{0} \leq \liminf_{(\varepsilon, p) \to (0,1)} \int_{\rm{supp}(\eta)}\eta |\nabla_{0} u_p^\varepsilon|_{0} \leq \liminf_{(\varepsilon, p) \to (0,1)} \int_{\rm{supp}(\eta)} \dfrac{1}{p}\eta |\nabla_{\varepsilon}u_{p}^{\varepsilon}|_{\varepsilon}^{p} \, , \]
because $ \frac{p-1}{p} \int_{\rm{supp}(\eta)}  \eta \to 0$ as $p \to 1$.

Combining the last inequality with \eqref{moser-usc} we conclude that
\begin{equation}
\label{conv-measures}
\int_{\rm{supp}(\eta)} \dfrac{1}{p} \eta |\nabla_{\varepsilon}u_{p}^{\varepsilon}|_{\varepsilon}^{p} \to \int_{\rm{supp}(\eta)}\eta |\nabla_{0}u|_{0}, \quad \textrm{as } (\varepsilon, p) \to (0,1),
\end{equation}
\noindent  for any nonnegative test function $\eta \in C^{\infty}_{0}(\Omega)$ such that $\rm{supp}(\eta)\subset\subset \Omega$  and such that $0 \leq \eta \leq 1$, whence the same holds if $\eta$ has the same properties but it is merely continuous, by approximation and \eqref{globalgradbound}. In turn, \eqref{conv-measures} holds for any  test function $\eta \in C^{\infty}_{0}(\Omega)$ such that $\rm{supp}(\eta)\subset\subset \Omega$, by decomposition in positive and negative part and normalization.

Now for $w \neq u$ and $K$ as above  we fix $\eta \in C^{\infty}_{0}(\Omega)$ such that $\rm{supp}(\eta)\subset\subset \Omega$, $0 \leq \eta \leq 1$ and $\eta \equiv 1$ in an open neighborhood of $K$. Once more we  take a bounded open set $V \subset \subset \Omega$ such that $\rm{supp}(\eta)\subset\subset V$ and a sequence $\{ w_k\}$ satisfying properties (i)-(ii)-(iii) above.  It follows from \eqref{conv-measures} that
$$\int_{\rm{supp}(\eta)}\eta (u_{p}^{\varepsilon}-w_k) |\nabla_{\varepsilon}u_{p}^{\varepsilon}|_{\varepsilon}^{p} \to \int_{\rm{supp}(\eta)} \eta (u-w_k)|\nabla_{0}u|_{0} \quad \textrm{as } (\varepsilon, p) \to (0,1),$$
\noindent where for fixed $k$ we used the uniform convergence of $\eta (u_{p}^{\varepsilon}-w_k)$ to $\eta (u-w_k)$ as a test function, 
so that \eqref{conv-measures} together with \eqref{eq:Moser9Bis}-\eqref{Moser-est3} imply
\begin{equation*}
\int_{\rm{supp}(\eta)}\eta |\nabla_{0}u|_{0} + \eta (u-w_k)|\nabla_{0}u|_{0}
\leq \int_{\rm{supp}(\eta)} \eta |\nabla_{0}w_k|_{0} + |w_k-u| |\nabla_{0}\eta|_{0}.
\end{equation*}
Finally, taking into account (i), (ii), and (iii), and passing to the limit as $k\to +\infty$ on both sides by dominated convergence, we get
\begin{equation*}
\int_{\rm{supp}(\eta)}\eta |\nabla_{0}u|_{0} + \eta (u-w)|\nabla_{0}u|_{0}
\leq \int_{\rm{supp}(\eta)} \eta |\nabla_{0}w|_{0}+|w-u| |\nabla_{0}\eta|_{0} \, ,
\end{equation*}
\noindent hence $\eta \equiv 1$ near $K$ and $w\equiv u$ in $\Omega \setminus K$ yield
\begin{equation*}
\int_{K} \left(|\nabla_{0}u|_{0} + u|\nabla_{0}u|_{0} \right)\leq \int_{K}\left(|\nabla_{0}w|_{0} + w|\nabla_{0}u|_{0} \right) \, ,
\end{equation*}
i.e., \eqref{eq:u_solution} holds and the proof is complete. 
\end{proof}

Finally we are in the position to prove the main result of the paper.

\begin{proof}[Proof of Theorem \ref{Thm:existence}]

Let  us fix $\varepsilon \in [0,1]$ and  for any $1<p\leq 2$  let $v^\varepsilon_p$ be the minimizer of \eqref{p-dirichlet} in the set $\dot{W}
^{1,p}_{1, \varepsilon} (\Omega)$ and let $u^\varepsilon_p$ the corresponding functions given by \eqref{substitution}. According to Theorem \ref{Thm:differential-harnack} we know that \eqref{diffharnack} holds for $\{ v^\varepsilon_p\}$ for a constant $C>0$ independent of $p$ as well as of $\varepsilon$. Hence, in view of Corollary \ref{Cor:gradbound} functions in  $\{ u^\varepsilon_p\}$ satisfy the uniform bound \eqref{globalgradbound}, so that $\{ u^\varepsilon_p\}_{p\in (1,Q)}$ are equi-Lipschitz with respect to $d_\varepsilon$, strictly positive in $\Omega$ and vanishing identically on $\He^n \setminus \Omega$. Moreover, they are weak solutions to \eqref{EqForUpeps} due to Lemma \ref{lem:v_sol_then_u_sol}, so that 
they minimize the functionals $J^{p,\varepsilon}_{u^\varepsilon_p}$ defined in \eqref{HIpeps-energy} with respect to competitors  $w \in\dot{W}
^{1,p}_{1, \varepsilon} (\Omega)$ such that $w-u^\varepsilon_p$ is compactly supported because of Lemma \ref{equa-vs-min}.

According to Proposition \ref{limiting-minimality}, up to subsequences, for each $\varepsilon \in [0,1]$ there exist functions $u^\varepsilon:\He^n \to [0,\infty)$ which are $d_\varepsilon$-Lipschitz on $\He^n$, with bound on Lipschitz constant independent on $\varepsilon \in [0,1]$,  such that  $u_{p}^{\varepsilon} \to u^\varepsilon$ as $p \to 1$ when $\varepsilon>0$ or $u_{p}^{\bar{\varepsilon}} \to u^0$ as $(\bar{\varepsilon}, p) \to (0,1)$ if $\varepsilon=0$ locally uniformly in $\He^n$.  Alternatively, for $\varepsilon=0$ we can let first $\bar{\varepsilon} \to 0$ and then $p \to 1$. In addition, $\{u^\varepsilon\}_{\varepsilon \in [0,1]}$  vanish identically on $\He^n \setminus \Omega$ and they are weak solution of \eqref{EulerForU} and \eqref{EulerForUeps} in the sense of Definition \ref{def:solution}.

In order to conclude the proof it is enough to establish \eqref{two-side-bound-u} and we treat first the case $\varepsilon \equiv 0$ because it is much simpler. Indeed, combing weak comparison principle for sub-Riemannian $p$-harmonic function applied on the open sets $\Omega$ and $\He^n \setminus \overline{B_{\bar{R}}(g_0)}$ with the explicit form of $p$-capacitary potentials from Lemma \ref{p-Harm} we have 
\begin{equation}
\label{two-side-bound-u0p}
 \left( \frac{\|  g_0^{-1} *g\|}{R_\ast} \right)^{-\gamma}  \leq v^0_p (g) \leq   \left( \frac{\|  g_0^{-1} *g\|}{\bar{R}} \right)^{-\gamma}  \quad \hbox{on } \quad \overline{\Omega} \, , \qquad \hbox{with} \quad \gamma=\frac{Q-p}{p-1} \, ,
\end{equation}
where the second inequality trivially holds in $\overline{\Omega}\cap B_{\bar{R}}(g_0)$ because $v^0_p \leq 1$ in $\He^n$. Taking logarithm and multiplying by $1-p<0$, as $p \to 1$ inequality \eqref{two-side-bound-u} follows in the form stated in the case $\varepsilon=0$.

As for the case $\varepsilon>0$, concerning the upper bound for each $u^\varepsilon=\lim  u^\varepsilon_p$ we rely on Proposition \ref{univ-pt-lbound-infinity}.  Thus, for $\varepsilon \in (0,1]$ and  $1<p\leq 2$, since $\Omega \subset \He^n$ satisfies assumption  ($\mathbf{HP_\Omega}$) with parameter $R_0$ and  for each $R_\ast<2<\bar{R}$ with $B_{R_\ast}(g_0) \subset \He^n \setminus \Omega \subset B_{\bar{R}}(g_0)$ there exists  $\widehat{C}>1$ depending only on $Q$, $R_\ast$ and $\bar{R}$ such that for $v^\varepsilon_p$, $u^\varepsilon_p=(1-p) \log v^\varepsilon_p$,  we have 
 \begin{equation}
 \label{pt-lbound-infinity-bis}
 v^\varepsilon_p(g) \geq \widehat{C}^{- \frac{\varepsilon^4}{p-1}}     \left( \frac{\| g_0^{-1}*g\|}{R_\ast} \right)^{-\gamma} \qquad \hbox{on } \quad \overline{\Omega} \, , \qquad \hbox{with} \quad \gamma=\frac{Q-p}{p-1} \, .
 \end{equation}
 Thus, taking logarithm of both sides in \eqref{pt-lbound-infinity-bis} as $p \to 1$ the upper bound follows because $u^\varepsilon=\lim_p  u^\varepsilon_p$.

When dealing with the lower bound in \eqref{two-side-bound-u} for $\varepsilon \in (0,1]$ we rely instead on Proposition \ref{univ-pt-ubound-infinity} so that inequality \eqref{pt-ubound-infinity} holds in $\overline{\Omega}$ for some  $C_0>1$ depending only on $Q$,$R_0$, $R_\ast$, $\bar{R}$ and $\partial \Omega$. Taking logarithm of both sides for $\gamma=\frac{Q-p}{p-1}$  and $g\in \overline{\Omega}$ we obtain   
 \[
u^\varepsilon(g)= \lim_{p \to 1} u^\varepsilon_{p}(g) =-\lim_{p \to 1} \log \left(v^\varepsilon_p(g)\right)^{p-1} \]
\[\geq  - \lim_{p\to 1}\log \left({C_0}     \left( \frac{\| g_0^{-1}*g\|}{\bar{R}} \right)^{p-Q} \right)=(Q-1) \log     \frac{\| g_0^{-1}*g\|}{\bar{R}} - \log C_0\,  ,
 \]
 which completes the proof of the lower bound. Note that the previous argument still applies in the case of the joint limit $(\varepsilon,p) \to (0,1)$, since inequality \eqref{pt-ubound-infinity} can be applied no matter how $u= \lim_{\varepsilon, p} u^\varepsilon_p$ is constructed.

\end{proof}

\medskip

\section*{Acknowledgments} We thank Luca Capogna and Giovanna Citti for several interesting and useful discussions on the content of some results of Section 4 and the connections with their papers \cite{CapoCitti3} and \cite{CapoCitti4}. \\ 
A.P. is partially funded by PRIN project 2022PJ9EFL ``Geometric Measure Theory: Structure of Singular Measures, Regularity Theory and Applications in the Calculus of Variations''.
E.V. is partially funded by PRIN project 2022R537CS ``$NO^3$ - Nodal Optimization, NOnlinear elliptic equations, NOnlocal geometric problems, with a focus on regularity''.

\section*{Data availability statement}
Data sharing not applicable to this article as no datasets were generated or analysed during
the current study.

\section*{Declarations}
\textbf{Conflict of interest.} On behalf of all authors, the corresponding author states that there is no conflict of interest.


\end{document}